\documentclass[reqno]{amsart}
\usepackage{enumerate}
\usepackage{enumitem}

\usepackage[utf8]{inputenc}
\usepackage{graphicx,xcolor}
\usepackage{amsmath,amsfonts,amsthm,amssymb}
\usepackage{hyperref}
\usepackage{xfrac}
\usepackage{mathrsfs}
\newcommand\res{\mathop{\hbox{\vrule height 7pt width .5pt depth 0pt
			\vrule height .5pt width 6pt depth 0pt}}\nolimits}
\newcommand\weakto{\rightharpoonup}
\newcommand\eps{\varepsilon}
\newcommand\R{\mathbb{R}}
\newcommand\Q{\mathbb{Q}}
\newcommand\N{\mathbb{N}}
\newcommand\Z{\mathbb{Z}}
\newcommand\Sf{\mathbb{S}}
\newcommand\corA{\mathscr{A}}
\newcommand\calL{\mathcal{L}}
\newcommand\calT{\mathcal{T}}
\newcommand\calH{\mathcal{H}}
\newcommand\calI{\mathcal{I}}
\newcommand\dx{{\mathrm d}x}
\newcommand\dy{{\mathrm d}y}

\newcommand\dt{{\mathrm d}t}
\newcommand\dd{{\mathrm d}}
\newcommand\Functeps{F}
\newcommand\Functminv{E}
\newcommand\Functmins{S}
\newcommand\minprobv{m_{\textup{b}}}
\newcommand\minprobs{m_{\textup{s}}}

\newcommand\homm{_\textup{hom}}
\newcommand\trunca{\alpha}
\newcommand\truncb{\beta}
\newcommand\Fsucc{\widehat{F}}

\newcommand\Bor{\mathscr{B}}
\newcommand\ie{\textit{; i.e., }}

\newtheorem{theorem}{Theorem}[section]
\newtheorem{definition}[theorem]{Definition}
\newtheorem{proposition}[theorem]{Proposition}
\newtheorem{lemma}[theorem]{Lemma}
\newtheorem{remark}[theorem]{Remark}
\newtheorem{corollary}[theorem]{Corollary}

\numberwithin{equation}{section}

\title[Homogenisation of phase-field functionals with linear growth]{Homogenisation of phase-field functionals\\ with linear growth}
\author[F. Colasanto]{Francesco Colasanto} 
\address[Francesco Colasanto]{DiMaI U.\ Dini, Universit\`a di Firenze, V.le G.B. Morgagni 67/A, 50134 Firenze, Italy}
\email[Francesco Colasanto]{francesco.colasanto@unifi.it}

\author[M. Focardi]{Matteo Focardi} 
\address[Matteo Focardi]{DiMaI U.\ Dini, Universit\`a di Firenze, V.le G.B. Morgagni 67/A, 50134 Firenze, Italy}
\email[Matteo Focardi]{matteo.focardi@unifi.it}

\author[C. I. Zeppieri]{Caterina Ida Zeppieri}
\address[Caterina Ida Zeppieri]{Applied Mathematics M\"unster, University of M\"unster\\
	Einsteinstrasse 62, 48149 M\"unster, Germany}
\email{caterina.zeppieri@uni-muenster.de}

\begin{document}
\begin{abstract}
We propose a first rigorous homogenisation procedure in image-segmentation models by analysing   
the relative impact of (possibly random) fine-scale oscillations and phase-field regularisations for a family of elliptic functionals 
of Ambrosio and Tortorelli type, when the regularised volume term grows \emph{linearly} in the gradient variable. 
In contrast to the more classical case of superlinear growth, we show that our functionals homogenise to a free-discontinuity energy whose surface term explicitly depends 
on the jump amplitude of the limit variable. The convergence result as above is obtained under very mild assumptions which allow us to treat, among other, the case of \emph{stationary random integrands}.   
 \end{abstract}

\maketitle


\section{Introduction}
In this paper we study the combined effect of \emph{homogenisation} and \emph{elliptic regularisation} for phase-field functionals of the form
\begin{equation}\label{e:random functional 1}
    \Functeps_{\eps}(u,v,A)=\int_{A}v^2{\textstyle f(\frac{x}{\eps},\nabla u)}\dx +\int_A\Big(\textstyle{\frac{(1-v)^2}{\eps}}+\eps |\nabla v|^2 \Big)\dx,
\end{equation}
where $\eps>0$ describes both the oscillation and the regularisation scale, and $f$ grows \emph{linearly} in the gradient variable.  In \eqref{e:random functional 1} $A \subset \R^n$ is open, bounded, with Lipschitz boundary,  $u$ is a vector-valued function which belongs to $W^{1,1}(A,\R^N)$, while $v$ is a phase-field variable lying in $W^{1,2}(A)$. 

As mentioned above, we require that the integrand $f: \R^n \times \R^{N\times n}\to [0,+\infty)$ obeys \emph{linear} growth and coercivity conditions in the second variable; that is
 \begin{equation}\label{intro:linear}        
    C^{-1}|\xi|\leq f(x,\xi)\leq C(|\xi|+1),
    \end{equation}
for every $(x,\xi)\in \R^n \times \R^{N\times n}$ and for some $C\in (0,+\infty)$. Besides \eqref{intro:linear}, we work under very mild assumptions on $f$ which \emph{do not include any spatial periodicity} (cf.\ Definition \ref{d:Admissible deterministic integrand}). Working in such a general setting allows us to prove a homogenisation result which also covers the case of random \emph{stationary} integrands, as we are going to explain below.   

The elliptic functionals in \eqref{e:random functional 1} are reminiscent of the celebrated phase-field model given by
\begin{equation*}
    AT_{\eps}(u,v,A)=\int_{A}v^2{|\nabla u|}^2\dx +\int_A\Big(\textstyle{\frac{(1-v)^2}{\eps}}+\eps |\nabla v|^2 \Big)\dx,
\end{equation*}
which was proposed by Ambrosio and Tortorelli in the seminal works \cite{AT90,AT} to approximate the (relaxed) Mumford-Shah functional \cite{MumfordShah}. The latter was introduced in the $2$d framework of image segmentation to recover shapes in noisy images via curve evolution. In this setting the  Ambrosio-Tortorelli functional is employed for implementation by gradient descent, where  curves are replaced by a continuous edge-strength function ($1-v$ in our notation) which gives the probability of an object boundary to be present at any point in the image domain. Then, the actual shape boundaries are determined in the form of geodesics defined in a metric determined by $v$ itself (cf. \cite{Shah95,Shah96-d}).

After the revisitation of Griffith's brittle-fracture theory due to Francfort and Marigo \cite{Marigo1998} (see also \cite{Bourdin2000b,Bourdin2008}),
a number of variants of the Ambrosio-Tortorelli model have been proposed and extensively used 
also to approximate brittle fracture models 
\cite{BBZ,BachCicaleseRuf, BMZ, BEZ,CicaleseFocardiZeppieri,Focardi2001},  just to mention few examples.
The advantage of this kind of approximations is twofold: on the one hand they establish a rigorous connection between variational fracture models and gradient-damage models \cite{Pham2010c,Pham2010a}, on the other hand, in most of the cases, they provide efficient algorithms for numerical simulations \cite{Bourdin2008,Bourdin2000b,Marigo1998}. 

If instead  in  \eqref{e:random functional 1} we choose $f(x,\xi)= |\xi|$, the corresponding phase-field functionals 
were originally proposed by Shah \cite{Shah96-b, Shah96-a} as possible regularisations of 
an image-segmentation model, alternative to the Mumford-Shah's functional, which provides a common framework for image segmentation and isotropic curve evolution in Computer Vision. Moreover, Shah's functional overcomes a number of limitations of the earlier models. Loosely speaking, in this framework the domain $A$ is interpreted as a Riemannian manifold endowed with a metric defined by the image properties so that the image-segmentation problem amounts to finding a minimal cut in a Riemannian manifold (cf. \cite{Shah96-c}).

The main difference between the Ambrosio-Tortorelli functionals (and their more ``classical'' variants) 
and \eqref{e:random functional 1}-\eqref{intro:linear} rests on the growth of the function $f$: superlinear in the former versus linear in the latter.
Such different behaviours lead to some structural differences in the corresponding, attainable limit models. In fact, the weaker gradient penalisation in \eqref{intro:linear} allows for an
interaction between the two competing terms in \eqref{e:random functional 1}, as it also typical of free-discontinuity functionals in the linear setting \cite{BouchitteFonsecaMascarenhas, CDMSZGlobal}.  
As a result, the surface energy densities obtained in this case are of \emph{cohesive} type as proven in \cite{AlicBrShah}, in the scalar isotropic case, and in \cite{AlFoc}, in the vector-valued anisotropic case. That is, the resulting limit surface integrands in the linear setting are bounded, increasing, and concave functions of the jump amplitude $[u]$ of the (possibly discontinuous) limit variable $u$, moreover they exhibit a linear growth at the origin. We observe though, that the linear growth of $f$ 
for large gradients is not justified in the applications to Fracture Mechanics, so that more recently other variants of the Ambrosio-Tortorelli functional related to the gradient-damage models in \cite{Pham2010c,Pham2010a} were designed to provide a variational approximation of  cohesive energies (cf. \cite{ContiFocardiIurlano2016,Wu2017,Wu2018b,Feng2021,ContiFocardiIurlano2024,
ContiFocardiIurlano2025,Lammen2025,ACF-I,Colasanto2024}). 
It is also worth mentioning that in these models the parameters can be tuned to approximate prescribed cohesive laws (satisfying suitable assumptions) as shown in \cite{ACF-II} (see also \cite{ACF-III} for applications to an engineering problem).

Furthermore, we observe that the coercivity assumption in \eqref{intro:linear} yields the ``weaker'' lower bound  
\begin{equation}\label{intro:trivial-lb}
 C^{-1}\int_{A}v^2{|\nabla u|}\dx +\int_A\Big(\textstyle{\frac{(1-v)^2}{\eps}}+\eps |\nabla v|^2 \Big)\dx   \leq \Functeps_{\eps}(u,v,A),
\end{equation}
where the functionals on the left-hand side are those proposed by Shah and studied in \cite{AlicBrShah}. Hence, from \eqref{intro:trivial-lb} and the analysis in \cite{AlicBrShah} (see also \cite{AlFoc}) we readily deduce that if $(u_\eps)\subset W^{1,1}(A,\R^N)$ is a sequence with equi-bounded energy which additionally satisfy
$\sup_\eps \|u_\eps\|_{L^q}<+\infty$, for some $q>1$, then (up to subsequences) $u_\eps \to u$ with respect to the strong $L^1(A,\R^N)$-convergence, for some $u\in (G)BV(A,\R^N)$. Therefore, in the linear setting, the limit functional shall contain a term depending on the Cantor part of the measure derivative $Du$. 
These features are  in sharp contrast with the case where $f$ grows superlinearly in the gradient variable. Indeed in this case the limit functional is defined on the smaller space $(G)SBV(A, \R^N)$. Additionally, the superlinear growth of $f$ in $|\nabla u|$ makes it energetically unfavourable to approximate a pure jump function with elastic deformations, so that the only surface energy densities which can be obtained in the limit are necessarily  independent of the jump amplitude of $u$, as recently proven in \cite{BEMZ, BEMZ-1, BMZ}. 

Motivated by the applications to anisotropic curve evolution \cite{Shah96-b, Shah96-a,Shah96-c}, in this paper we study the homogenisation of the phase-field functionals in \eqref{e:random functional 1} which encompass the case of highly oscillating, possibly random metrics.
Moreover,  since image-segmentation models are highly sensitive to the presence of heterogeneities in regions or objects due to noise,  it is in general of great importance to incorporate a homogenisation procedure in these models and in their phase-field counterparts. ln fact, the presence of noise can cause random variations in the image intensity values, which in turn produce false detections in the image so that homogenisation may help to reduce the impact of noise, shadows, and changes in the illumination intensity, which usually make it difficult to accurately segment the image into its relevant parts. Therefore, in practice, by removing such, it can be easier to detect boundaries between different objects in the image, and to distinguish between foreground and background regions. 
More specifically,  in the present work we rigorously analyse the interplay between fine-scale oscillations and phase-field approximations in linear models as in  \eqref{e:random functional 1}. Due to the presence of microscopic heterogeneities, as $\eps$ tends to zero we expect to obtain an effective model where the (cohesive) energy density depends both on the homogenised integrand $f\homm$ (through its recession function) and on the regularised surface-term in \eqref{e:random functional 1}. On the other hand, on account of the analysis in \cite{AlicBrShah, AlFoc} we also expect a limit volume energy which only depends on the first term in  \eqref{e:random functional 1} and therefore in this case on $f\homm$.  A central feature of our analysis is that we study the homogenisation of $\Functeps_{\eps}$ without imposing any periodicity of $f$ in the spatial variable. In fact, in the same spirit as in \cite{CDMSZStochom, CDMSZGlobal}, we work under more general assumptions which, notably, are satisfied in the random stationary case.  

Finally, it is also worth noticing that the homogenisation problem analysed in this paper can be seen as a case study of a homogenisation problem for the gradient-damage models proposed in \cite{ContiFocardiIurlano2016,Wu2017,Wu2018b,Feng2021,ContiFocardiIurlano2024,
ContiFocardiIurlano2025,Lammen2025,ACF-I,Colasanto2024} to approximate cohesive energies in Fracture Mechanics. Indeed, on account of the analysis performed in these papers, also in this case we expect effective surface integrands defined by asymptotic minimisation problems in which all the terms in the approximating functionals interact with one another. Moreover, when working with such approximations, a more technically demanding analysis shall be expected 
due to the superlinear growth of the bulk energy density and to the more complex, parameter-dependent, choice of the degenerate function
multiplying $f$.    

\medskip
{
Below we briefly outline the proof strategy leading to our homogenisation result. Our approach is inspired by \cite{CDMSZStochom, CDMSZGlobal}, where the authors extend to the setting of free-discontinuity functionals the seminal work of Dal Maso and Modica \cite{dalmaso-modica}. In that foundational contribution, ergodic theory is combined with $\Gamma$-convergence for the first time, in the context of volume functionals. We also refer to \cite{AbddaimiMichailleLicht1997} for a related homogenisation result concerning volume functionals with linear growth.}

Loosely speaking, our strategy consists of two main steps: a purely deterministic one, where we devise sufficient conditions (on $f$) leading to homogenisation and a probabilistic step, where we show that if $f$ is a stationary random variable, then the sufficient conditions mentioned above are indeed fulfilled. Therefore a stochastic homogenisation result readily follows as a corollary of the deterministic analysis.
\subsection{Deterministic homogenisation}
Here we assume that  $f$ satisfies the assumptions listed in Definition~\ref{d:Admissible deterministic integrand}. Besides \eqref{intro:linear} these require that
the recession function $f^{\infty}$ is defined at every point. 
We stress here that we do not require any continuity of $f$ in the spatial variable, since this would
be unnatural for the applications.

Under these general assumptions, using the localisation method of $\Gamma$-convergence \cite{Dalmaso1993}, we can
prove the existence of a subsequence $(\eps_j)$ such that, for every $A\subset \R^n$ open and bounded, the functionals $\Functeps_{\eps_j}(\cdot, \cdot, A)$ $\Gamma$-converge to an
abstract functional $\Fsucc(\cdot,\cdot,A)$. Furthermore, the latter has the property that for every $u\in BV_{\textup{loc}}(\R^n,\R^N)$ the set function $A\mapsto \Fsucc(u,1,A)$ is the restriction to the open subsets of $\R^n$ of a Borel measure (cf.\ Theorem~\ref{t:Sottosucc gamma-conv.}). We observe that since we do not assume any spatial periodicity of $f$, the continuity of $z\mapsto \Fsucc(u(\cdot-z),1,A+z)$ may fail and therefore we cannot directly use the integral representation result in $BV$ \cite{BouchitteFonsecaMascarenhas} to deduce the form of $\Fsucc$. Our integral representation result is then obtained under some additional assumptions,
which are however more general than periodicity. We require that the limits of some scaled minimisation
problems, defined in terms of $f$ and $f^{\infty}$, exist and are independent of the spatial variable. These
limits will then define the volume and surface integrands of $\Fsucc$. Eventually, the Cantor integrand will be automatically identified due to the lower semicontinuity of $\Fsucc$.

Specifically, we make the two following assumptions.  If $Q_r(rx)$ denotes the open cube with side-length $r$ centred at $rx$ and $\ell_\xi(x)=\xi x$, the first assumption amounts to asking that for every $\xi \in \R^{N\times n}$ the limit 
\begin{equation}\label{e:esistenza limite per fhom}
    \lim_{r\to \infty}\frac{1}{r^n}\inf \left \{\int_{Q_r(rx)} f(y,\nabla u)\dy \colon u\in W^{1,1}(Q_r(rx),\R^N), \; u= \ell_{\xi} \textup{ on } \partial Q_r(rx) \right\}
 \end{equation}
exists and it is independent of $x\in \R^n$. The value of \eqref{e:esistenza limite per fhom} is denoted by $f\homm(\xi)$. 

Moreover, if $Q^\nu_r(rx)$ denotes the open cube with side-length $r$ centred at $rx$, one side orthogonal to $\nu \in \Sf^{n-1}$, and
 \begin{equation*}
        u_{rx,\zeta,\nu}(y)=
        \begin{cases}
            \zeta \; & \; \textup{if $(y-rx)\cdot \nu \geq 0$} \\
            0 \; & \; \textup{if $(y-rx)\cdot \nu < 0$},
        \end{cases}
\end{equation*}
we also require that for every $\zeta\in \R^N$ and $\nu \in \Sf^{n-1}$ the limit
\begin{multline}\label{e:esistenza limite per ghom}
    \lim_{r\to +\infty}\frac{1}{r^{n-1}}\inf \bigg\{\int_{Q^\nu_r(rx)} \big(v^2 f^\infty (y,\nabla u) + (1-v)^2 + |\nabla v|^2\big) \dy \colon u \in  W^{1,1}(Q^\nu_r(rx),\R^N),\\  v\in W^{1,2}(Q^\nu_r(rx)) \textup{ and } (u,v)=(u_{rx,\zeta,\nu},1) \textup{ on } \partial Q^\nu_r(rx) \bigg\}
    \end{multline}
exists and is independent of $x\in \R^N$. The value of \eqref{e:esistenza limite per ghom} is denoted by $g\homm(\zeta,\nu)$. 

It is worth mentioning here that $f\homm$ and $g\homm$ satisfy a number of properties (cf.\ Section \ref{s:homgenised densities}) which ensure, in particular, that they are Borel measurable.  

Then, assuming \eqref{e:esistenza limite per fhom} and \eqref{e:esistenza limite per ghom} we resort to the blow-up technique in $BV$ \cite{BouchitteFonsecaMascarenhas} to show that for every $u\in BV(A,\R^N)$ the following identities hold true
\begin{align*}
    \frac{\dd \Fsucc(u,1,\cdot)}{\dd \calL^n}(x)&=f\homm(\nabla u(x)) \quad \textup{for $\calL^n$-a.e. $x\in A$},  \\  \frac{\dd \Fsucc(u,1,\cdot)}{\dd |D^cu|}(x) &= f^{\infty}\homm(\nabla u(x)) \quad \textup{for $|D^cu|$-a.e. $x\in A$}, \\   \frac{\dd \Fsucc(u,1,\cdot)}{\dd \mathcal{H}^{n-1}}(x)&= g\homm([u](x),\nu_{u}(x)) \quad \textup{for $\mathcal{H}^{n-1}$-a.e. $x\in J_u\cap A$}.
\end{align*}
In their turn, these allow us to represent $\Fsucc$ in an integral form first on $BV$, and then by standard truncation arguments on the domain of the $\Gamma$-limit, that is, on $GBV$. 
   Furthermore, since in the equalities above the right-hand side does not depend on the subsequence $(\eps_j)$, under assumptions \eqref{e:esistenza limite per fhom} and \eqref{e:esistenza limite per ghom} we obtain a $\Gamma$-convergence result for the whole sequence $(\Functeps_{\eps})$ (see Theorem~\ref{t:Deterministic Gamma-conv}).  

\subsection{Stochastic homogenisation}
Here we consider an underlying complete probability space $(\Omega, \mathcal T, P)$ endowed with a group of $P$-preserving transformations, and allow the integrand $f$ to additionally depend on $\omega \in \Omega$, in a suitable measurable way. Then,  
if $f$ is a \emph{stationary random integrand} in the sense of Definition \ref{d:def stationary integrand}, we show that assumptions \eqref{e:esistenza limite per fhom} and \eqref{e:esistenza limite per ghom} are automatically satisfied for $P$-a.e.\ $\omega \in \Omega$, that is, \emph{almost surely}.  As it is by-now costumery (see \cite{CDMSZStochom, CDMSZGlobal}) this is done by appealing to the Ackoglu and Krengel Subadditive Ergodic Theorem \cite{Ackoglu}.   
More specifically, the proof that \eqref{e:esistenza limite per fhom} holds is standard and follows as in \cite{RufZepp}. On the other hand, the verification of \eqref{e:esistenza limite per ghom} is highly non trivial, as it is always the case when working with ``surface terms'' where there is a dimensional mismatch between the domain of integration and the scaling, and, moreover, a boundary datum which is inherently inhomogeneous (cf.\ \ref{e:esistenza limite per ghom}).   

Once  \eqref{e:esistenza limite per fhom} and \eqref{e:esistenza limite per ghom} are shown to hold (cf. Proposition \ref{p:Volume part stochastic proposition} and Proposition \ref{p:surface part stochastic}) we can immediately resort to the deterministic analysis to deduce that the random functionals 
\[
\Functeps_{\eps}(\omega)(u,v,A)=\int_{A}v^2{\textstyle f(\omega, \frac{x}{\eps},\nabla u)}\dx +\int_A\Big(\textstyle{\frac{(1-v)^2}{\eps}}+\eps |\nabla v|^2 \Big)\dx,
\]
homogenise, almost surely, to the random, autonomous, free-discontinuity functional  
\begin{equation*}
    \Functeps\homm(\omega)(u,v,A)=\int_A f\homm(\omega,\nabla u) \dx + \int_A f^{\infty}\homm\textstyle{(\omega,\frac{\dd D^cu}{\dd |D^cu|})}\dd |D^cu| + \displaystyle{\int_{J_u\cap A}g\homm(\omega,[u],\nu_u) \dd \calH^{n-1}},
\end{equation*}
if $u\in GBV(A,\R^N)$ and $v=1\; \calL^n$-a.e in $A$, where $f\homm$ and $g\homm$ are defined, respectively, by  \eqref{e:esistenza limite per fhom} and \eqref{e:esistenza limite per ghom} while $f^\infty\homm$ is the recession function of $f\homm$ (cf.\ Theorem \ref{t:Homogenised Formulas} and Theroem \ref{t:Stochastic homogenisation}). Eventually, if $f$ is stationary with respect to an ergodic group of $P$-preserving transformations on $(\Omega, \mathcal T, P)$, then the homogenisation procedure becomes effective and thus $ \Functeps\homm$ is deterministic. 

\section{Preliminaries and set up}

\subsection{Notation}
We introduce some notation which will be used throughout the paper.
\begin{enumerate} [label=(\alph*)]

\item\label{e:scalar product}  Let $n,N \in \N$ be fixed with $n\geq 2$. For $x,y\in \R^n$ and $\zeta\in \R^N$, $x\cdot y:=x_1y_1+\dots+x_ny_n$ 
is the euclidean scalar product of $x$ and $y$, while $\zeta \otimes x:=(\zeta_ix_j)_{ij}\in \R^{N\times n}$ is the tensor product of $\zeta$ and $x$.

\item\label{e:iperpiani} For $x\in \R^n$ and $\nu\in \Sf^{n-1}$, we set
$\Pi^{\nu}:=\{y\in \R^n \colon y\cdot \nu=0\}$ and $\Pi^{\nu}_{x}:=x+\Pi^{\nu}$.

\item\label{e:mappa lineare} For $\xi \in \R^{N\times n}$, $\ell_{\xi}$ denotes the linear function from $\R^n$ to $\R^N$ with gradient $\xi$.

\item\label{e:Sf} For $k\in \N$ and $x=(x_1,\dots,x_k)\in \R^k $, $|x|:=\sqrt{x^2_1+\dots+x_k^2}$ is the euclidean norm of the vector $x$. $\Sf^{k-1}:=\{x\in \R^k \; | \; |x|=1\}$ is the $k-1$-dimensional sphere centered in the origin and $\Hat{\Sf}^{k-1}_{\pm}:=\{x\in \Sf^{k-1} \; | \; \pm x_{i(x)}>0\}$, where $i(x)$ is the largest $i\in \{1,\dots,k\}$ such that $x_i\neq 0$. Note that $\Sf^{k-1}=\Hat{\Sf}^{k-1}_{+}\cup \Hat{\Sf}^{k-1}_{-}$ and $\Hat{\Sf}^{k-1}_{\pm}$ is a Borel set.

\item\label{e:Rnu} For $\nu\in \Sf^{n-1}$, let $R_{\nu}$ be an orthogonal $n\times n$ matrix such that $R_{\nu}e_n=\nu$; we assume that the restriction of the function $\nu \mapsto R_{\nu}$ to the sets $\Hat{\Sf}^{n-1}_{\pm}$, defined in \ref{e:Sf} of the notation list, are continuous and that $R_{-\nu}Q_1=R_{\nu}Q_1$; moreover we assume that $R_{\nu}\in \Q^{n\times n}$ if $\nu\in \Q^n$. A map $\nu \mapsto R_{\nu}$ satisfying these properties is provided in \cite[Example~A.1 and Remark~A.2]{CDMSZGammaconv}.

\item\label{e:cubo ruotato} For $x\in \R^n$ and $\rho>0$ we set $B_{\rho}(x):=\{y\in \R^n \colon \; |y-x|<\rho\}$ and $Q_{\rho}(x):=\{y\in \R^n \colon \; |(y-x)\cdot e_i|<\rho/2 \textup{ for $i=1,\dots,n$}\}$, where $\{e_1,\dots,e_n\}$ is the standard basis of $\R^n$. Moreover $B_{\rho}$ and $Q_{\rho}$ stand, respectively, for $B_{\rho}(0)$ and $Q_{\rho}(0)$.
    
For $x\in \R^n$, $\rho>0$, and $\nu \in \Sf^{n-1}$ we set
    \begin{equation*}
        Q^{\nu}_{\rho}(x):=x+R_{\nu}Q_{\rho}.
    \end{equation*}
    For $k\in \N$ we define the rectangle
    \begin{equation*}
        Q^{\nu,k}_{\rho}(x):=x+Q^{\nu,k}_{\rho}
    \end{equation*}
    where $Q^{\nu,k}_{\rho}:=R_{\nu}\textstyle{((-\frac{k\rho}{2},\frac{k\rho}{2})^{n-1}\times (-\frac{\rho}{2},\frac{\rho}{2}))}$. Moreover we set
    \begin{equation*}
      \partial^{\perp}Q^{\nu,k}_{\rho}(x):=\partial Q^{\nu,k}_{\rho}(x) \cap R_{\nu}\textstyle{((-\frac{k\rho}{2},\frac{k\rho}{2})^{n-1}\times \R)} 
    \end{equation*}
    \begin{equation*}
        \partial^{\|}Q^{\nu,k}_{\rho}(x):=\partial Q^{\nu,k}_{\rho}(x) \cap R_{\nu}\textstyle{(\R^{n-1} \times  (-\frac{\rho}{2},\frac{\rho}{2}))}.
    \end{equation*}

\item\label{e:corA} $\corA$ and $\corA_{\infty}$ denotes the collection of all bounded open sets and of all bounded open Lipschitz sets of $\R^n$ respectively; if $A,B \in \corA$, by $A \subset \subset B$ we mean that exists a compact set $K$ such that $A\subset K \subset B$. For every $C\in \corA$, we define $\corA(C):=\{A\in \corA \; | \; A\subseteq C\}$ and $\corA_{\infty}(C):=\{A\in \corA_{\infty} \; | \; A\subseteq C\}$.

\item\label{e:Borel} For every topological space $X$, $\Bor(X)$ denotes its Borel $\sigma$-algebra. For every integer $k\geq 1$, $\Bor^k$ is the Borel 
$\sigma$-algebra of $\R^k$, while $\Bor^n_S$ denotes the Borel $\sigma$-algebra of $\Sf^{n-1}$. 

\item\label{e:LkHk} $\calL^k$ and $\calH^{k-1}$ denote respectively the Lebesgue and the $(k-1)$-dimensional Hausdorff measure on $\R^k$.

\item\label{e:misure} Let $\mu$ and $\lambda$ two Radon measures on $A\in \corA$, with values in a finite dimensional real vector space $X$ and in $[0,+\infty]$, respectively; then $\frac{\dd \mu}{\dd \lambda}:=\frac{\dd \mu^a}{\dd \lambda} \in L^1_{\textup{loc}}(A,X)$, where $\mu^a \ll \lambda$, $\mu^a+\mu^s$ is the Radon-Nykodym decomposition of $\mu$ respect to $\lambda$ and $\mu^a(B)=\int_B \frac{\dd \mu^a}{\dd \lambda} \dd \lambda $ for every Borel set $B\subseteq A$.

\item\label{e:BV} For $u\in BV(A,\R^N)$, with $A\in \corA$, the jump of $u$ on the jump set $J_u$ is denoted by $[u]:=u^+-u^-$, while $\nu_u$ denotes the normal to $J_u$. The distributional gradient $Du$, is a $\R^{N\times n}$-valued Radon measure on $A$, whose absolutely continuous part with respect to $\calL^n$, denoted by $D^au$, has density $\nabla u \in L^1(A,\R^{N\times n})$ (which coincides with that approximate gradient of $u$), while the singular part $D^su$ can be decomposed as $D^su=D^ju+D^cu$ where the jump part $D^ju$ is given by $D^j u=[u]\otimes \nu_u \calH^{n-1}\res J_u $,
    and the Cantor part $D^cu$ is a $\R^{N\times n}$-valued Radon measure on $A$ which vanishes on all Borel sets $B\subseteq A$ with $\calH^{n-1}(B)<+\infty$.

We refer to the book \cite{AFP} for all the properties of $(G)BV$ and $(G)SBV$ functions, giving precise references.   
    
\item\label{e:funzione salto} For $x\in \R^n$, $\zeta\in \R^m$, $\nu \in \Sf^{n-1}$ and $\eps>0$ we define the function $u_{x,\zeta,\nu}, \overline{u}^{\eps}_{x,\zeta,\nu}:\R^{n}\to \R^N$ as
    \begin{equation*}
        u_{x,\zeta,\nu}(y):=
        \begin{cases}
            \zeta \; & \; \textup{if $(y-x)\cdot \nu \geq 0$} \\
            0 \; & \; \textup{if $(y-x)\cdot \nu < 0$},
        \end{cases}
        \quad \textup{and} \quad \overline{u}^{\eps}_{x,\zeta,\nu}(y):=\zeta \overline{\textup{u}}\textstyle{(\frac{1}{\eps}(y-x)\cdot \nu)}
    \end{equation*}
    where $\overline{\textup{u}}:\R\to [0,1]$ is a fixed smooth cut-off function such that $\overline{\textup{u}}\equiv 1$ on $[1/2,+\infty)$ and $\overline{\textup{u}}\equiv 0$ on $(-\infty,-1/2]$.

    We also use the shorthand notation $u_{\zeta,\nu}:=u_{0,\zeta,\nu}$, $\overline{u}_{x,\zeta,\nu}:=\overline{u}^{1}_{x,\zeta,\nu}$ and $\overline{u}_{\zeta,\nu}:=\overline{u}_{0,\zeta,\nu}$.

\item\label{i:truncation} We define the truncation functions $\mathcal{T}_k\in C^1_c(\R^N,\R^N)$ satisfying 
\begin{equation}\label{e:prop 1 Tk}
    \mathcal{T}_k(\zeta):=
    \begin{cases}
        \zeta \; & \; \textup{if $|\zeta|\leq a_k$}, \\
        0 \; & \; \textup{if $|\zeta|\geq a_{k+1}$},
    \end{cases}
\end{equation}
and 
\begin{equation}\label{e:prop 2 Tk}
\mathrm{Lip}(\mathcal{T}_k)\leq 1 \quad \textup{and} \quad |\mathcal{T}_k(\zeta)|\leq a_{k+1} \; \; \textup{for every $\zeta \in \R^n$},
\end{equation}
for some diverging and strictly increasing sequence of positive numbers $(a_k)$.

\item\label{e:funzione recessione} Given $h:\R^{N\times n}\to [0,+\infty]$ its recession function $h^{\infty}:\R^{N\times n}\to [0,+\infty]$ is defined as
    \begin{equation*}
        h^{\infty}(\xi):=\limsup_{t\to +\infty}\frac{h(t\xi)}{t}.
    \end{equation*}
\end{enumerate}

{{We will quote the book \cite{Dalmaso1993} for all the results on the abstract
theory of $\Gamma$-convergence needed in what follows.}}

\subsection{The subadditive ergodic Theorem}
In this subsection we recall a variant of the pointwise subaddtive ergodic Theorem of Ackoglu and Krengel \cite[Theorem~2.7]{Ackoglu} which is useful for our purposes (cf.\ \cite[Theorem~4.1]{Licht}).

Let $d \in \N$. Let $(\Omega,{{\mathcal{T}}},P)$ be a probability space and let $\tau:=(\tau_z)_{z\in\mathbb{Z}^d}$ denote a group of $P$-preserving transformations on $(\Omega,{{\mathcal{T}}},P)$, that is, $\tau$ is a family of measurable mappings $\tau_z:\Omega\to\Omega$ satisfying the following properties:
\begin{itemize}
	\item $\tau_z\tau_{z'}=\tau_{z+z'}$, $\tau_z^{-1}=\tau_{-z}$, for every $z,z'\in\mathbb{Z}^d$;
	\item $\tau$ preserves the probability measure $P$\ie $P(\tau_zE)=P(E)$, for every $z\in\mathbb{Z}^d$ and every $E\in{{\mathcal{T}}}$;
	\end{itemize}
If in addition every $\tau$-invariant set $E\in{{\mathcal{T}}}$, {{i.e. $\tau_zE=E$ for every $z\in\mathbb{Z}^d$},}
has either probability $0$ or $1$, then $\tau$ is called \textit{ergodic}.

For every $a,b\in \R^d$ with $a_i<b_i$ for $i=1,\dots,d$, we define 
\begin{equation*}
    [a,b)=\{x\in \R^d \; : \; a_i\leq x_i <b_i \; \; \textup{for $i=1,\dots,d$}\},
\end{equation*}
and we set
\begin{equation*}
  \calI_d:=\{[a,b) \; : \; a,b\in \R^d, \; a_i<b_i \; \; \textup{for $i=1,\dots,d$}\}.
\end{equation*}

\begin{definition}[Subadditive process]\label{d:subadditive process}  Let $\tau:=(\tau_z)_{z\in\mathbb{Z}^d}$ be a group of $P$-preserving transformations on $(\Omega,{{\mathcal{T}}},P)$.
A $d$-dimensional subadditive process is a function $\mu:\Omega \times \calI_d \to \R$ satisfying the following properties:
\begin{enumerate} [label= (\alph*)]
    \item for every $A\in \calI_d$ the map $\omega \mapsto \mu(\omega,A)$ is $\calT$-measurable;
    \item for every $\omega\in \Omega$, $A\in \calI_d$, and $z\in \Z^d$ we have $\mu(\omega,A+z)=\mu(\tau_z\omega,A)$;
    \item for every $A\in \calI_d$ and for every finite family $(A_i)_{i\in I}$ in $\calI_d$ of pairwise disjoint sets such that $\cup_{i\in I}A_i=A$, we have 
    \begin{equation*}
        \mu(\omega,A)\leq \sum_{i\in I}\mu(\omega,A_i),
    \end{equation*}
    for every $\omega\in \Omega$;
    \item there exists $c>0$ such that 
    \[
    0\leq \mu(\omega,A)\leq c\calL^d(A),
    \] 
    for every $\omega\in \Omega$ and every $A\in \calI_d$.
\end{enumerate}
\end{definition}

\begin{definition}[Regular family of sets]\label{d:regular family}
A family of sets $(A_t)_{t>0}$ in $\calI_d$ is called regular with constant $M \in (0,+\infty)$ if there exists another family of sets $(A'_t)_{t>0}$ in $\calI_d$ such that:
\begin{itemize}
    \item $A_t\subset A'_t$ for every $t>0$;
    \item $A'_s \subset A'_t$ whenever $0<s<t$;
    \item $0<\calL^d(A'_t)\leq M\calL^d(A_t)$ for every $t>0$;
    \item $\bigcup_{t>0}A'_t=\R^d$.
\end{itemize}
\end{definition}
\begin{theorem}[Subadditive Ergodic Theorem]\label{t:ergodic theorem}
Let $\tau=(\tau_z)_{z\in \Z^d}$ be a group of $P$-preserving transformations on $(\Omega,\calT,P)$. Let $\mu:\Omega\times \calI_d \to [0,+\infty)$ be a $d$-dimensional subadditive process. Then there exist a $\calT$-measurable function $\varphi:\Omega \to [0,+\infty)$ and a set $\Omega'\in \calT$ with $P(\Omega')=1$ such that 
\begin{equation*}
    \lim_{t\to +\infty}\frac{\mu(\omega,A_t)}{\calL^d(A_t)}=\varphi(\omega)
\end{equation*}
for every regular family of sets $(A_t)_{t>0}$ in $\calI_d$ and for every $\omega\in \Omega'$. If in addition $\tau$ is ergodic, then $\varphi$ is constant $P$-a.e.
\end{theorem}

\subsection{Assumptions} In this subsection we introduce the class of the admissible random integrands.  
\begin{definition}[Admissible integrand]\label{d:Admissible deterministic integrand}
Let $C\geq 1$ and $\alpha \in (0,1)$ be given, then $\mathcal{F}(C,\alpha)$ denotes the collection of all functions $f:\R^n\times \R^{N\times n}\to [0,+\infty)$ with the following properties: 
\begin{enumerate}[label=(f\arabic*)]
    \item\label{meas}\textup{(measurability)} $f$ is $\Bor^n\otimes \Bor^{N\times n}$-measurable;
    \item \label{e:crescita lineare}
 \textup{(linear growth)} for every $x\in\R^n$ and every $\xi \in \R^{N\times n}$
    \begin{equation*}        
    C^{-1}|\xi|\leq f(x,\xi)\leq C(|\xi|+1);
    \end{equation*}
    \item \label{e:semicontinuità f e finf} \textup{(continuity)} for every $x\in \R^n$ the maps
       $\xi\mapsto f(x,\xi)$ and $\xi \mapsto f^{\infty}(x,\xi)$
    are continuous;
    \item \label{e:limite funz rec} \textup{(recession function)} for every $x\in \R^n$ every $\xi\in \R^{N\times n}$ and every $t>0$
    \begin{equation*}
         \Big|f^{\infty}(x,\xi)-\frac{f(x,t\xi)}{t}\Big|<\frac{C}{t}(1+f(x,t\xi)^{1-\alpha}).
    \end{equation*}
 \end{enumerate}
\end{definition}

\begin{remark}\label{r:oss. funz. recessione}
Let $f \in \mathcal F(C,\alpha)$, then thanks to \ref{e:crescita lineare} and \ref{e:limite funz rec}, for every $x\in \R^n$ and every $\xi\in \R^{N\times n}$ we have that there exists the limit
\begin{equation*}
   \lim_{t\to+\infty}\frac{f(x,t\xi)}{t} = f^{\infty}(x,\xi),
\end{equation*}
and
\begin{equation}\label{e:crescita funz recess}
    C^{-1}|\xi|\leq f^{\infty}(x,\xi)\leq C|\xi|.
\end{equation}
Moreover, for every $L>0$ there exists $M>0$ such that for every $x\in \R^n$, $\xi\in \R^{N\times n}$ with $|\xi|=1$ and $t>L$ we have that
\begin{equation}\label{e:H4 BFM}
    \Big|f^{\infty}(x,\xi)-\frac{f(x,t\xi)}{t}\Big|\leq \frac{M}{t^{\alpha}}.
\end{equation}
\end{remark}
\begin{definition}[Random integrand]
A function $f:\Omega\times \R^n\times \R^{N\times n} \to [0,+\infty)$ is called a random integrand if
\begin{enumerate}[label=(s-f\arabic*)]
    \item[(s-f1)] $f$ is $\calT \otimes \Bor^n \otimes \Bor^{N\times n}$-measurable;
      \item[(s-f2)] $f(\omega,\cdot,\cdot)\in \mathcal{F}(C,\alpha)$ for every $\omega \in \Omega$, where $\mathcal{F}(C,\alpha)$ is as in Definition~\ref{d:Admissible deterministic integrand}.
\end{enumerate}
\end{definition}
If $f$ is a random integrand then $f^{\infty}:\Omega \times \R^n \times \R^{N\times n} \to [0,+\infty)$ is given by
\begin{equation*}
    f^{\infty}(\omega,x,\xi)=\lim_{t\to +\infty}\frac{f(\omega,x,t\xi)}{t},
\end{equation*}
where the existence of the limit above is ensured by the very definition of random integrand together with Remark~\ref{r:oss. funz. recessione}. 

\begin{definition}[Stationary random integrand]\label{d:def stationary integrand}
A random integrand $f$ is stationary if there exists $\tau=(\tau_{z})_{z\in \Z^n}$ $n$-dimensional group of $P$-preserving transformation on $(\Omega,\calT,P)$ such that
\begin{equation*}
    f(\omega,x+z,\xi)=f(\tau_z \omega,x,\xi)
\end{equation*}
for every $\omega\in \Omega$, $x\in \R^n$, $z\in {{\Z}^n}$, and $\xi\in \R^{N\times n}$.

If in addition $\tau$ is ergodic we call $f$ an ergodic random integrand. 
\end{definition}

\section{Statements of the main results}

Let $f$ be a given stationary random integrand. For $\eps>0$ we consider the phase-field functionals $\Functeps_{\eps}(\omega):L^1_{\textup{loc}}(\R^n,\R^{N+1})\times \corA\longrightarrow [0,+\infty]$ defined as 
\begin{equation}\label{e:Random Functionals}
    \Functeps_{\eps}(\omega)(u,v,A):=
    \begin{cases}
        \displaystyle \int_A \hspace{-1mm}\textstyle{(v^2 f(\omega,\frac{x}{\eps},\nabla u)+\frac{(1-v)^2}{\eps}+\eps|\nabla v|^2)}\dx, &  (u,v)\in W^{1,1}(A,\R^N)\times W^{1,2}(A) \\
        +\infty  & \textup{otherwise}.
    \end{cases}
\end{equation}
\begin{remark}
For $v\in W^{1,2}(A)$ set $\tilde v:= \min\{\max\{0,v\},1\}$. We notice that for every $\eps>0$, $\omega \in \Omega$ there holds
\[
 \Functeps_{\eps}(\omega)(u,\tilde v,A) \leq  \Functeps_{\eps}(\omega)(u,v,A),
\]
for every $(u,v)\in W^{1,1}(A,\R^N)\times W^{1,2}(A)$ and $A\in \corA$.
Therefore it is not restrictive to assume that the phase-field variable $v$ satisfies the pointwise bounds $0\leq v\leq 1$ for $\calL^n$-a.e. $x\in A$. 
\end{remark}

\begin{remark}[Equi-coercivity]
The coercivity assumption in \ref{e:crescita lineare} immediately gives that 
\begin{equation*}
 C^{-1}\int_{A}v^2{|\nabla u|}\dx +\int_A\Big(\textstyle{\frac{(1-v)^2}{\eps}}+\eps |\nabla v|^2 \Big)\dx   \leq \Functeps_{\eps}(\omega)(u,v,A)
\end{equation*}
where the functionals on the left-hand side are those studied in \cite{AlicBrShah} (see also \cite{AlFoc}). Hence, up to considering the perturbed functionals 
\[
\Functeps_{\eps}(\omega)(u,v,A)+ \|u\|_{L^q(A,\R^N)},
\]
for some $q>1$, we can appeal to \cite[Lemma 7.1]{AlFoc} to deduce that if $(u_\eps,v_\eps)\subset W^{1,1}(A,\R^N)\times W^{1,2}(A,[0,1])$ satisfies
\[
\sup_{\eps>0} \Big(\Functeps_{\eps}(\omega)(u_\eps,v_\eps,A)+ \|u_\eps\|_{L^q(A,\R^N)}\Big)<+\infty,
\]
then, up to subsequences, $(u_\eps,v_\eps) \to (u,1)$ strongly in $L^1(A,\R^{N+1})$ for some $u\in GBV(A,\R^N)$.
For this reason, in what follows we are going to study the $\Gamma$-convergence of $\Functeps_{\eps}$ with respect to the strong $L^1$-convergence. 
\end{remark}

Before stating our main results we need some additional notation. Let  $h:\R^n\times \R^{n\times N}\to [0,\infty)$ satisfy \ref{e:crescita lineare} and \ref{meas}. For $A\in \corA_{\infty}$ and $(u,v)\in W^{1,1}(A,\R^N) \times W^{1,2}(A,[0,1])$ consider the following auxiliary integral functionals  
\begin{equation}\label{e:FunzHomob}
 \Functminv^h(u,A):=\int_A h(x,\nabla u) \dx,
\end{equation}
and
\begin{equation}\label{e:FunzHomos}
\Functmins^h(u,v,A):=\int_A (v^2h(x,\nabla u) + (1-v)^2 + |\nabla v|^2) \dx.
\end{equation}
Moreover, let $w\in BV_{\textup{loc}}(\R^n,\R^N)$ and define the minimisation problems
\begin{equation}\label{e:def minprobv}
    \minprobv^h(w,A):=\inf\{E^h(u,A)\; : \;u\in W^{1,1}(A,\R^N), \; u=w \textup{ on } \partial A \}
\end{equation}
and
\begin{multline}\label{e:def minprobs}
    \minprobs^h(w,A):= \inf\{\Functmins^h(u,v,A) \colon  u\in  W^{1,1} (A,\R^N),  \; v\in W^{1,2}(A,[0,1]), \;(u,v)=(w,1) \textup{ on } \partial A\},
\end{multline}
where $u=w$ on $\partial A$ has to be intended in the sense of traces and inner traces for $u$ and $w$, respectively.
If $A\subseteq \R^n$ is a set such that $\textup{int} A \in \corA_{\infty}$ then we use the following convention $\minprobv^h(w,A):=\minprobv^h(w,\textup{int} A)$ and $\minprobs^h(w,A):=\minprobs^h(w,\textup{int} A)$.

The main result of this paper is contained in Theorem~\ref{t:Stochastic homogenisation}, below, and provides an \emph{almost sure} $\Gamma$-convergence result for the functionals $\Functeps_{\eps}$ defined in \eqref{e:Random Functionals}. 
In order to state this result we preliminarily need to state a theorem which guarantees the almost sure existence of the integrands of the $\Gamma$-limit. Namely, the next theorem establishes the existence and spatial homogeneity of the limits defining the asymptotic cell formulas appearing in Theorem~\ref{t:Stochastic homogenisation} below.

Throughout the paper we adopt the following shorthand notation.

\begin{equation*}
    \minprobv^{f_{\omega}}:=\minprobv^{f(\omega,\cdot,\cdot)} \quad \textup{and} \quad  \minprobs^{f^{\infty}_{\omega}}:=\minprobs^{f^{\infty}(\omega,\cdot,\cdot)}.
\end{equation*}
\begin{theorem}[Homogenisation formulas]\label{t:Homogenised Formulas}
  Let $f$ be a stationary random integrand. Then there exists $\Omega'\in \mathcal{T}$ with $P(\Omega')=1$, such that for every $\omega\in \Omega'$
  \begin{itemize}
      \item[(i)] every $x\in \R^n$, $\nu \in \Sf^{n-1}$, $k\in \N$, and $\xi \in \R^{N\times n}$, the limit
      \begin{equation*}
            \lim_{r\to +\infty}\frac{\minprobv^{f_{\omega}}(\ell_{\xi},Q^{\nu,k}_{r}(rx))}{k^{n-1}r^n}
      \end{equation*}
      exists and it is independent of $x,\nu$ and $k$;
      \item[(ii)] every $x\in \R^n$, $\zeta \in \R^N$, and $\nu \in \Sf^{n-1}$, the limit
      \begin{equation*}
          \lim_{r\to +\infty}\frac{\minprobs^{f^{\infty}_{\omega}}(u_{rx,\zeta,\nu},Q^{\nu}_{r}(rx))}{r^{n-1}}
      \end{equation*}
      exists and it is independent of $x$.
  \end{itemize}
  More precisely there exist a $\calT \otimes \Bor^{N\times n}$-measurable function $f\homm:\Omega\times \R^{N\times n}\to [0,\infty)$ and a $\calT \otimes \Bor^N \otimes \Bor_S^n$-measurable function $g\homm:\Omega \times \R^N \times \Sf^{n-1} \to [0,+\infty)$ such that for every $\omega \in \Omega'$, $x\in \R^n$, $\xi \in \R^{N\times n}$, $\zeta \in \R^N$, and $\nu \in \Sf^{n-1}$
  \begin{align*}
      f\homm(\omega,\xi)&=\lim_{r\to +\infty}\frac{\minprobv^{f_{\omega}}(\ell_{\xi},Q^{\nu,k}_{r}(rx))}{k^{n-1}r^n}
      \\
      &=\lim_{r\to +\infty}\frac{\minprobv^{f_{\omega}}(\ell_{\xi},Q_{r}(rx))}{r^n}=\lim_{r\to +\infty}\frac{\minprobv^{f_{\omega}}(\ell_{\xi},Q_{r})}{r^n},
  \end{align*}
  \begin{align*}
       f\homm^{\infty}(\omega,\xi)&=\lim_{r\to +\infty}\frac{\minprobv^{f^{\infty}_{\omega}}(\ell_{\xi},Q^{\nu,k}_{r}(rx))}{k^{n-1}r^n}
       \\
       &=\lim_{r\to +\infty}\frac{\minprobv^{f^{\infty}_{\omega}}(\ell_{\xi},Q_{r}(rx))}{r^n}=\lim_{r\to +\infty}\frac{\minprobv^{f^{\infty}_{\omega}}(\ell_{\xi},Q_{r})}{r^n},
  \end{align*}
  \begin{equation*}
      g\homm(\omega,\zeta,\nu)=\lim_{r\to +\infty}\frac{\minprobs^{f^{\infty}_{\omega}}(u_{rx,\zeta,\nu},Q^{\nu}_{r}(rx))}{r^{n-1}}=\lim_{r\to +\infty}\frac{\minprobs^{f^{\infty}_{\omega}}(u_{\zeta,\nu},Q^{\nu}_{r})}{r^{n-1}},
  \end{equation*}
where $f\homm^{\infty}$ denotes the recession function of $f\homm$.
 
If we additionally assume that $f$ is ergodic, then $f\homm$ and $g\homm$ are independent of $\omega$ and 
\begin{equation*}
    f\homm(\xi)=\lim_{r\to +\infty}\frac{1}{r^n}\int_{\Omega}\minprobv^{f_{\omega}}(\ell_{\xi},Q_{r}) \dd P(\omega),
\end{equation*}
\begin{equation*}
    f\homm^{\infty}(\xi)=\lim_{r\to +\infty}\frac{1}{r^n}\int_{\Omega}\minprobv^{f^{\infty}_{\omega}}(\ell_{\xi},Q_{r}) \dd P(\omega),
\end{equation*}
\begin{equation*}
    g\homm(\zeta,\nu)=\lim_{r\to +\infty}\frac{1}{r^n}\int_{\Omega}\minprobs^{f^{\infty}_{\omega}}(u_{\zeta,\nu},Q_{r}) \dd P(\omega).
\end{equation*}
\end{theorem}

We are now in a position to state the main result of this paper.

\begin{theorem}[Almost sure $\Gamma$-convergence]\label{t:Stochastic homogenisation}
Let $f$ be a stationary random integrand. For $\eps>0$ and $\omega \in \Omega$ let $\Functeps_{\eps}(\omega)$ be the functionals defined in \eqref{e:Random Functionals}. Then, there exists $\Omega'\in \mathcal{T}$ with $P(\Omega')=1$ such that 
for every $\omega \in \Omega'$, $A\in \corA$, and $(u,v)\in L^1_{\textup{loc}}(\R^n,\R^{N+1})$ we have 
 \begin{equation*}
    \Gamma(L^1_{\textup{loc}}(\R^n,\R^{N+1}))\text{-}\lim_{\eps \to 0}\Functeps_{\eps}(\omega)(u,v,A)=\Functeps\homm(\omega)(u,v,A),
\end{equation*}
where $\Functeps\homm(\omega):L^1_{\textup{loc}}(\R^n,\R^{N+1})\times \corA\longrightarrow [0,\infty]$ is defined as
 \begin{equation*}
     \Functeps\homm(\omega)(u,v,A):=
     \begin{cases}
     \displaystyle \int_A f\homm(\omega,\nabla u) \dx + \int_A f^{\infty}\homm\textstyle{(\omega,\frac{\dd D^cu}{\dd |D^cu|})}\dd |D^cu| + \displaystyle\int_{J_u\cap A}g\homm(\omega,[u],\nu_u) \dd \calH^{n-1} \\[4pt]
 \hspace{4.8cm}\text{if $u\in GBV(A,\R^N)$ and $v=1$ for $\calL^n$-a.e. $x\in A$}    
 \cr
 +\infty  \hspace{4.2cm} \text{otherwise}
     \end{cases}
\end{equation*}
with $f\homm$ and $g\homm$ as in Theorem~\ref{t:Homogenised Formulas}. 

If in addition $f$ is ergodic, then the functional $\Functeps\homm$ is deterministic.
\end{theorem}

The proof of Theorem~\ref{t:Stochastic homogenisation} will be carried out in a number of steps in the next sections. 
 
\section{Properties of the homogenized integrands}\label{s:homgenised densities}
In this section we prove a number of structural properties of the homogenized integrands $f\homm$ and $g\homm$. 

For later use, it is convenient to work in a deterministic framework where the dependence of $f\homm$ and $g\homm$ on $\omega$ is not taken into account. Then, as a consequence, we need to assume that the limits defining $f\homm$ and $g\homm$ exist and are spatially homogeneous. 

We start with $f\homm$. 

\begin{proposition}\label{p:prop_fhom}
Let $f\in \mathcal{F}(C,\alpha)$ and assume that for every $x\in \R^n$ and $\xi \in \R^{N\times n}$ the limit
\begin{equation}\label{e:fhom-ass}
 \lim_{r\to +\infty}\frac{\minprobv^f(\ell_{\xi},Q_{r}(rx))}{r^{n}}=:f\homm(\xi)
\end{equation}
exists (and is independent of $x$). 
Then, $f\homm$ satisfies the following properties:
\begin{enumerate}
\item $f\homm$ is quasi-convex;
    \item for every $\xi_1,\xi_2\in \R^{N\times n}$
    \begin{equation*}
        |f\homm(\xi_1)-f\homm(\xi_2)|\leq K|\xi_1-\xi_2|,
    \end{equation*}
    where $K$ is a constant that depends only on $n, N$ and $C$;
\item for every $\xi\in \R^{N\times n}$
    \begin{equation}\label{e:crescita lineare fhom}
        C^{-1}|\xi|\leq f\homm(\xi)\leq C(|\xi|+1).
    \end{equation}
\end{enumerate}
\end{proposition}

\begin{proof}
$(i)$ The quasi-convexity of $f\homm$ defined as in \eqref{e:fhom-ass} is shown in \cite[Proposition 5.5 Step 2]{RufZepp}. 

$(ii)$ Let $\xi_1,\xi_2 \in \R^{N\times n}$ and $r>0$ be fixed. For every $u\in BV(Q_r,\R^N)$ and  $A\in \corA(Q_r)$ consider the auxiliary functional defined as 
\begin{equation*}
    J(u,A):=
    \begin{cases}
      \displaystyle \int_{A}f(y,\nabla u) \dy \; & \textup{if $u\in W^{1,1}(Q_r,\R^N)$} 
      \cr
        +\infty \; & \textup{in} \; BV(Q_r,\R^N) \setminus W^{1,1}(Q_r,\R^N),
    \end{cases}
\end{equation*}
as well as $\overline J(\cdot,A):= sc^{-}(L^1)J(\cdot,A).$

By \ref{e:crescita lineare} we have that $\overline{J}(u,A)\leq C(|Du|(A)+\calL^n(A))$, therefore thanks to \cite[Lemma~3.1 and Lemma~4.1.2]{BouchitteFonsecaMascarenhas} we get 
\begin{align}\label{e:lip}
   |m_{\overline J}(\ell_{\xi_1},Q_r)&-m_{\overline J}(\ell_{\xi_2},Q_r)|  \leq C \|\ell_{\xi_1}-\ell_{\xi_2}\|_{L^1(\partial {{Q_r}})} 
   \nonumber \\
   & \leq C\widehat K|\xi_1-\xi_2|\int_{\partial Q_r}|y| \dd \calH^{n-1}(y)
   \leq \frac{C\widehat K\sqrt{n}}{2}\calH^{n-1}(\partial Q_1)|\xi_1-\xi_2|r^n
\end{align}
where $\widehat K$ depends only on $n$ and $N$, while
\begin{equation*}
    m_{\overline J}(\ell_{\xi},Q_r):=\inf \{\overline J(u,Q_r) \; : \; u\in BV(Q_r,\R^N) \;\; \textup{with $u=\ell_{\xi}$ on $\partial Q_r$} \}.
\end{equation*}
Appealing to \cite[Lemma~4.1.3]{BouchitteFonsecaMascarenhas} we deduce that 
\begin{equation*}
   m_{\overline J}(\ell_{\xi_1},Q_r)=\minprobv^{f}({{\ell_{\xi_1}}},Q_r) \quad \textup{and} \quad  m_{\overline J}(\ell_{\xi_2},Q_r)=\minprobv^{f}({{\ell_{\xi_2}}},Q_r)
\end{equation*}
therefore, combining \eqref{e:fhom-ass} and \eqref{e:lip} readily gives 
\begin{equation*}
    |f\homm(\xi_1)-f\homm(\xi_2)|\leq K|\xi_1-\xi_2|,
\end{equation*}
with $K:=\frac{C\widehat K\sqrt{n}}{2}\calH^{n-1}(\partial Q_1)$.

$(iii)$  Let $\xi\in \R^{N\times n}$, $r>0$, and $u\in W^{1,1}(Q_r,\R^N)$ with $u=\ell_{\xi}$ on $\partial Q_r$ be arbitrary and fixed. 
By \ref{e:crescita lineare} we have
\begin{equation*}
    C^{-1}|\xi|r^n = C^{-1}\Big|\int_{Q_r}\nabla u \;\dx \Big|\leq C^{-1}\int_{Q_r}|\nabla u|\dx \leq \int_{Q_r}f(x,\nabla u)\dx, 
\end{equation*}
therefore passing to the inf on $u$ we immediately get
\[
C^{-1}|\xi|r^n \leq \minprobv^f(\ell_{\xi},Q_r)
\]
for every $\xi \in \R^{N\times n}$ and $r>0$. Hence the {{first}} inequality in \eqref{e:crescita lineare fhom} follows by \eqref{e:fhom-ass}. 

The second inequality in \eqref{e:crescita lineare fhom} is a consequence of the trivial inequality
\begin{equation*}
    \minprobv^f(\ell_{\xi},Q_r)\leq \int_{Q_r}f(x,\nabla \ell_{\xi}) \dx \leq C(|\xi|+1)r^n
\end{equation*}
which, in turn, is implied by \ref{e:crescita lineare}.
\end{proof}

Below we prove that $f^\infty\homm$ can be equivalently expressed as the limit of suitable (scaled) minimisation problems.  To prove it we make use of the following lemma. 
\begin{lemma}\label{l:lemma funz recess integrale}
Let $g\in \mathcal{F}(C,\alpha)$, $A\in \corA$, $(u,v)\in W^{1,1}(A,\R^N)\times W^{1,2}(A,[0,1])$, then for every $t>0$ we have that 
\begin{equation*}
    \int_A \Big|v^2g^{\infty}(y,\nabla u) -v^2\frac{g(y,t\nabla u)}{t}\Big| \dy \leq \frac{K}{t}\calL^n(A)+ \frac{K}{t^{\alpha}}\calL^n(A)^{\alpha}\Big(\int_A v^2|\nabla u|\dy\Big)^{1-\alpha},
    \end{equation*}
\end{lemma}
where $K$ is a positive constant depending only on $C$ and $\alpha$.
\begin{proof}
Thanks to \ref{e:limite funz rec}, $0\leq v\leq 1$ and $\alpha\in (0,1)$ we have that
\begin{equation*}
      \int_A \Big|v^2g^{\infty}(y,\nabla u) -v^2\frac{g(y,t\nabla u)}{t}\Big| \dy \leq \frac{C}{t}\calL^n(A)
      + \frac{C}{t}\int_A v^{2(1-\alpha)}g(y,t\nabla u)^{1-\alpha} \dy\,,
\end{equation*}
thus by Jensen's Inequality we deduce that
\begin{equation*}
     \frac{C}{t}\int_A v^{2(1-\alpha)}g(y,t\nabla u)^{1-\alpha} \dy \leq \frac{C}{t}\calL^n(A)^{\alpha}\Big(\int_A v^2g(y,t\nabla u)\dy\Big)^{1-\alpha}\,.
\end{equation*}
Eventually, we conclude by \ref{e:crescita lineare}.
\end{proof}
\begin{proposition}\label{p:finfhom prop}
Let $f\in \mathcal{F}(C,\alpha)$ and assume that for every $x\in \R^n$, $\nu \in \Sf^{n-1}$, $k\in \N$, and $\xi \in \R^{N\times n}$
\begin{equation}\label{e:fhom-ass-ii}
 \lim_{r\to +\infty}\frac{\minprobv^f(\ell_{\xi},Q^{\nu,k}_{r}(rx))}{k^{n-1}r^{n}}=f\homm(\xi)
\end{equation}
where $f\homm$ is as in \eqref{e:fhom-ass}. 
Let $f^\infty\homm$ be the recession 
function of $f\homm$, then for every $x\in \R^n$, $\xi\in \R^{n\times N}$, $\nu \in \Sf^{n-1}$ and $k\in \N$ we have
\[
   f^{\infty}\homm(\xi)=\lim_{r\to +\infty} \frac{\minprobv^{f^{\infty}}(\ell_{\xi},Q^{\nu,k}_{r}(rx))}{k^{n-1}r^n},
\]
hence, in particular, $f^{\infty}\homm=(f^\infty)\homm$.
\end{proposition}

\begin{proof}
Let $x\in \R^n$, $\xi\in \R^{n\times N}$, $\nu \in \Sf^{n-1}$, $k\in \N$ and $\eta\in (0,1)$ be fixed. By \eqref{e:def minprobv}, for every $r>0$ there exists $u_r\in W^{1,1}(Q^{\nu,k}_r(rx))$ with $u_r=\ell_{\xi}$ on $\partial Q^{\nu,k}_r(rx)$, such that
\begin{equation}\label{e:stima 1 prop finfhom}
    \Functminv^{f^{\infty}}(u_r,Q^{\nu,k}_r(rx)) \leq \minprobv^{f^{\infty}}(\ell_{\xi},Q^{\nu,k}_r(rx))+\eta k^{n-1}r^n,
\end{equation}
and
\begin{equation*}
    C^{-1}\int_{Q^{\nu,k}_r(rx)}|\nabla u_r|\dy \leq \minprobv^{f^{\infty}}(\ell_{\xi},Q^{\nu,k}_r(rx))+\eta k^{n-1}r^n \leq C(|\xi|+1)k^{n-1}r^n.
\end{equation*}
In particular, by Lemma~\ref{l:lemma funz recess integrale}, for every $t\geq 1$ we obtain
\begin{align*}
  \int_{Q^{\nu,k}_{r}(rx)}&|f^{\infty}(y,\nabla u_r)-\frac{1}{t}f(y,t\nabla u_r)|\dy \\
  &\leq \frac{K}{t}k^{n-1}r^n+ \frac{K}{t^{\alpha}}(k^{n-1}r^n)^{\alpha}\Big(\int_{Q^{\nu,k}_r(rx)} |\nabla u_r|\dy\Big)^{1-\alpha} \leq \frac{1}{t^{\alpha}}\Hat K k^{n-1}r^n\,,
\end{align*}
where $\Hat K$ depends only on $C$, $\alpha$ and $\xi$. Hence, for $t\geq 1$,
\begin{equation*}
     \Functminv^{f_t}(u_r,Q^{\nu,k}_r(rx)) \leq  \Functminv^{f^{\infty}}(u_r,Q^{\nu,k}_r(rx)) +\frac{1}{t^{\alpha} }\Hat Kk^{n-1}r^n,
\end{equation*}
where $f_t(y,\xi):=\frac{f(y,t\xi)}{t}$ and consequently, by \eqref{e:stima 1 prop finfhom},
\begin{equation}\label{e:stima 2 prop finfhom}
    \frac{\minprobv^{f_t}(\ell_{\xi},Q^{\nu,k}_r(rx))}{k^{n-1}r^n}\leq \frac{\minprobv^{f^{\infty}}(\ell_{\xi},Q^{\nu,k}_r(rx))}{k^{n-1}r^n}+\eta+\frac{\Hat K}{t^{\alpha}}.
\end{equation}
Observing that $\minprobv^{f_t}(\ell_{\xi},Q^{\nu,k}_r(rx))=\frac{1}{t}\minprobv^{f}(\ell_{t\xi},Q^{\nu,k}_r(rx))$, thanks to the linearity of $\xi \mapsto \ell_{\xi}$, we get
\begin{equation}\label{e:equaz 1 prop finfhom}
 \lim_{r\to \infty}\frac{\minprobv^{f_t}(\ell_{\xi},Q^{\nu,k}_r(rx))}{k^{n-1}r^n}=  \lim_{r\to \infty}\frac{\minprobv^{f}(\ell_{t\xi},Q^{\nu,k}_r(rx))}{tk^{n-1}r^n} =\frac{f\homm(t\xi)}{t},
\end{equation}
by \eqref{e:fhom-ass-ii}. Hence, from \eqref{e:stima 2 prop finfhom} and \eqref{e:equaz 1 prop finfhom}, letting $\eta\to 0$, we have
\begin{equation*}
    \limsup_{t\to +\infty}\frac{f\homm(t\xi)}{t}\leq \liminf_{r\to +\infty}\frac{\minprobv^{f^{\infty}}(\ell_{\xi},Q^{\nu,k}_r(rx))}{k^{n-1}r^n}.
\end{equation*}
Exchanging the role of $f_t$ and $f^{\infty}$ and arguing analogously, we obtain
\begin{equation*}
    \limsup_{r\to +\infty}\frac{\minprobv^{f^{\infty}}(\ell_{\xi},Q^{\nu,k}_r(rx))}{k^{n-1}r^n} \leq \liminf_{t\to +\infty}\frac{f\homm(t\xi)}{t}.\qedhere
\end{equation*}
\end{proof}

To prove the properties satisfied by $g\homm$, we first establish some technical results in the spirit of \cite[Section~3]{BouchitteFonsecaMascarenhas}.
\begin{lemma}\label{l:bordi regolarizzati}
 Let $x\in \R^n$, $r>1$, $\nu \in \Sf^{n-1}$, and $w_1,w_2\in BV_{\textup{loc}}(\R^n,\R^N)$, then we have that
 \begin{equation*}
     |\minprobs^{f^{\infty}}(w_1,Q^{\nu}_r(x))-\minprobs^{f^{\infty}}(w_2,Q^{\nu}_{r}(x))|\leq C\int_{\partial Q^{\nu}_r(x)}|w_1-w_2|\calH^{n-1}.
 \end{equation*}
\end{lemma}
\begin{proof}
First we observe that for every $w\in BV_{\textup{loc}}(\R^n,\R^N)$, $x\in \R^n$, 
$\nu \in \Sf^{n-1}$, and $r>1$ there holds
$\minprobs^{f^{\infty}}(w,Q^{\nu}_r(rx))=\minprobs^{f^{\infty},\ast}(w,Q^{\nu}_r(rx))$,
where
\begin{align}\label{e:prob minimo coercivi}
    \minprobs^{f^{\infty},\ast}(w,Q^{\nu}_r(rx)):= \inf\{\Functmins^{f^{\infty}}&(u,v,Q^{\nu}_r(rx)) \; : \;  u\in  W^{1,1} (Q^{\nu}_r(rx),\R^N),  \; v\in W^{1,2}(Q^{\nu}_r(rx),[0,1]), \nonumber \\ & (u,v)=(w,1) \textup{ on } \partial Q^{\nu}_r(rx), \; \; v\geq \eta \; \textup{for some $\eta \in (0,1)$} \},
\end{align}
with $\Functmins^{f^{\infty}}$ defined in \eqref{e:FunzHomos}.
In fact, given $\eta \in (0,1)$, $u\in W^{1,1}(Q^{\nu}_r(x),\R^N)$ and $v\in W^{1,2}(Q^{\nu}_r(x),[0,1])$,  $v_{\eta} :=v\lor \eta\in W^{1,2}(Q^{\nu}_r(x),[0,1])$ with $v_{\eta}=1$ on $\partial Q^{\nu}_r(x)$, and
\begin{equation*}
    \int_{Q^{\nu}_r(x)}(1-v_{\eta})^2 + |\nabla v_{\eta}|^2 \dy \leq \int_{Q^{\nu}_r(x)}(1-v)^2 + |\nabla v|^2 \dy,
\end{equation*}
 and
\begin{equation*}
    \lim_{\eta \to 0^+} \int_{Q^{\nu}_r(x)}v^2_{\eta}f^{\infty}(y,\nabla u) \dy =  \int_{Q^{\nu}_r(x)}v^2f^{\infty}(y,\nabla u)\dy.
\end{equation*}
Let $v\in W^{1,2}(Q^{\nu}_r(x),[0,1])$ with $v=1$ on $\partial Q^{\nu}_r(x)$ and 
$v\geq \eta$ for some $\eta \in (0,1)$. Define the functional $\mathcal{F}_v:BV(Q^{\nu}_r(x),\R^N)\times \corA(Q^{\nu}_r(x))\longrightarrow [0,+\infty]$ as
\begin{equation*}
    \mathcal{F}_v(u,B):=
    \begin{cases}
      \displaystyle{\int_{B}v^2 f^{\infty}(y,\nabla u) \dy} \; & \textup{if $u\in W^{1,1}(Q^{\nu}_r(x),\R^N)$} \\
        +\infty \; & \textup{otherwise}.
    \end{cases}
\end{equation*}
Consider its relaxation $\overline{\mathcal{F}}_v:=sc^-(L^1)\mathcal{F}_v:BV(Q^{\nu}_r(x),\R^N)\times \corA(Q^{\nu}_r(x))\to [0,+\infty]$.
\cite[Lemma~3.1 and Lemma~4.1.2]{BouchitteFonsecaMascarenhas} and $\overline{\mathcal{F}}_v(u,B)\leq C|Du|(B)$ imply that
\begin{equation}\label{e:stima lip dirichlet surface}
    |m_{\overline{\mathcal{F}}_v}(w_1,Q^{\nu}_r(x))-m_{\overline{\mathcal{F}}_v}(w_2,Q^{\nu}_r(x))|\leq C\int_{\partial Q^{\nu}_r(x)}|w_1-w_2|\calH^{n-1},
\end{equation}
{{where for every $w\in BV(Q^{\nu}_r(x),\R^N)$
\begin{equation*}
  m_{\overline{\mathcal{F}}_v}(w,Q^{\nu}_r(x)):=\inf\{\overline{\mathcal{F}}_v(u,Q^{\nu}_r(x)) \; : \; u\in BV(Q^{\nu}_r(x),\R^N), \; u=w \; \textup{on $\partial Q^{\nu}_r(x)$}\}.
\end{equation*}}}
In addition, $m_{\overline{\mathcal{F}}_v}(w_i,Q^{\nu}_r(x))=m_{\mathcal{F}_v}(w_i,Q^{\nu}_r(x))$ for $i\in\{1,2\}$ by \cite[Lemma~4.1.3]{BouchitteFonsecaMascarenhas}, where {$m_{\mathcal{F}_v}(\cdot,Q^{\nu}_r(x))$ is defined as $m_{\overline{\mathcal{F}}_v}(\cdot,Q^{\nu}_r(x))$ with the functional $\overline{\mathcal{F}}_v$ replaced
by ${\mathcal{F}}_v$}.
Therefore, using \eqref{e:prob minimo coercivi} we can rewrite 
$\minprobs^{f^{\infty}}(w,Q^{\nu}_r(rx))$ as
\begin{align*}
\minprobs^{f^{\infty}}(w,Q^{\nu}_r(rx))=\inf \{&m_{\overline{\mathcal{F}}_v}(w,Q^{\nu}_r(rx))+ \textstyle\int_{Q^{\nu}_r(rx)}((1-v)^2+|\nabla v|^2) \dy:
\nonumber\\ &  
\; v\in W^{1,2}(Q^{\nu}_r(rx),[0,1])\; v\geq \eta\, 
\textup{ for some $\eta \in (0,1)$}, \;
{{v=1 \; \textup{on $\partial Q^{\nu}_r(x)$}}} \}\,,
\end{align*}
and thus we can conclude thanks to \eqref{e:stima lip dirichlet surface}.
\end{proof}
The following result readily follows from Lemma \ref{l:bordi regolarizzati}.

\begin{corollary}\label{c:comparison minimum problems}
     Let $x\in \R^n$, $r>1$, $\nu \in \Sf^{n-1}$, and $\zeta \in \R^N$ then we have that
 \begin{equation*}
     |\minprobs^{f^{\infty}}(u_{x,\zeta,\nu},Q^{\nu}_r(x))-\minprobs^{f^{\infty}}(\overline{u}_{x,\zeta,\nu},Q^{\nu}_{r}(x))|\leq 2C|\zeta| r^{n-2},
 \end{equation*}
where $u_{x,\zeta,\nu}$ 
and $\overline{u}_{x,\zeta,\nu}$ are defined in \ref{e:funzione salto}.

\end{corollary}

The next lemma will be widely used {in what follows (see
Proposition\ref{p:la ghom} and Section \ref{s:stochastic}).}
\begin{lemma}\label{l:lemma prob di minimo superficie su cubi diversi}
Let $x,z\in \R^n$, $\nu_1,\nu_2\in \Sf^{n-1}$, and $r_2>r_1>2R\geq 1$ be such that $ Q^{\nu_1}_{r_1}(x)\subset \subset Q^{\nu_2}_{r_2}(z) $ and $|(y-x)\cdot \nu_1|\leq \frac{1}{2}$ imply
\begin{equation}\label{e:condizione di contenimento}
    |(y-z)\cdot \nu_2|\leq R \quad \quad \textup{for every $y\in  Q^{\nu_2}_{r_2}(z)$}.
\end{equation}
Then, for every $\zeta \in \R^N$ and $\eta>0$ the following statements hold true:
for every $r_3\geq r_2$ such that $Q^{\nu_2}_{r_2}(z)\subset \subset Q^{\nu_1}_{r_3}(x)$, then 
    \begin{align*}
   \minprobs^{f^{\infty}}(\overline{u}_{z,\zeta,\nu_2},Q^{\nu_2}_{r_2}(z))- \minprobs^{f^{\infty}}(\overline{u}_{x,\zeta,\nu_1},Q^{\nu_1}_{r_1}(x)) \leq \eta r_1^{n-1}+  \tilde K{{|\zeta|}}((r_{3}-r_1)r_{3}^{n-2} +Rr_2^{n-2});
\end{align*}
%
with $\tilde K$ depending only on $n$ and $C$. 
\end{lemma}

\begin{proof}
Fix $\zeta \in \R^N$, $\eta>0$ and let $(u,v)\in W^{1,1}(Q^{\nu_1}_{r_1}(x),\R^N)\times W^{1,2}(Q^{\nu_1}_{r_1}(x),[0,1])$, with $(u,v)=(\overline{u}_{x,\zeta,{{\nu_1}}},1)$ on $\partial Q^{\nu_1}_{r_1}(x)$, such that 
\begin{equation}\label{e:prima equaz lemma dei prob di minimo}
    S^{f^\infty}(u,v,Q^{\nu_1}_{r_1}(x))=
     \int_{Q^{\nu_1}_{r_1}(x)}v^2 (f^{\infty}(y,\nabla u) + (1-v)^2 + |\nabla v|^2)\dy \leq \minprobs^{f^{\infty}}(\overline{u}_{x,\zeta,{{\nu_1}}},Q^{\nu_1}_{r_1}(x)) +\eta r_1^{n-1}.
\end{equation}
Define $(\hat u, \hat v)$ by 
 \begin{equation*}
     (\hat u(y),\hat v(y)):=
     \begin{cases}
         (u(y),v(y)) \; & \; \textup{if $y\in Q^{\nu_1}_{r_1}(x)$} \\
         (\overline{u}_{x,\zeta,{{\nu_1}}},1) \; & \; \textup{if $y\in Q^{\nu_2}_{r_2}(z) \setminus Q^{\nu_1}_{r_1}(x)$}\,,
     \end{cases}
 \end{equation*}
 and note that $(\hat u,\hat v)\in W^{1,1}(Q^{\nu_2}_{r_2}(z),\R^N)\times W^{1,2}(Q^{\nu_2}_{r_2}(z),[0,1]) $ with $(\hat u, \hat v)=(\overline{u}_{x,\zeta,{{\nu_1}}},1)$ on $\partial Q^{\nu_2}_{r_2}(z)$.

By \eqref{e:condizione di contenimento} it follows that $\overline{u}_{z,\zeta,\nu_2}(y)=\overline{u}_{x,\zeta,\nu_1}(y)$ for every $y\in Q^{\nu_2}_{r_2}(z)$ such that $|(y-z)\cdot \nu_2|> R$; in particular Lemma \ref{l:bordi regolarizzati} yields that
 \begin{align}\label{e:seconda equaz lemmino problemi di minimo}
    & |\minprobs^{f^{\infty}}(\overline{u}_{z,\zeta,\nu_2},Q^{\nu_2}_{r_2}(z))-\minprobs^{f^{\infty}}(\overline{u}_{x,\zeta,\nu_1},Q^{\nu_2}_{r_2}(z))|\leq \int_{\partial Q^{\nu_2}_{r_2}(z)}|\overline{u}_{z,\zeta,\nu_2}-\overline{u}_{x,\zeta,\nu_1}|\dd \mathcal{H}^{n-1} \nonumber \\ & =  \int_{\partial Q^{\nu_2}_{r_2}(z)\cap {\Sigma_{\nu_2,R}}}|\overline{u}_{z,\zeta,\nu_2}-\overline{u}_{x,\zeta,\nu_1}|\dd \mathcal{H}^{n-1}\leq 8R(n-1)|\zeta|r_2^{n-2},
 \end{align}
 where $\Sigma_{\nu_2,R}:=\{|(y-z)\cdot \nu_2| \leq R\}$. Furthermore, setting $\Sigma_{\nu_1,\sfrac12}:=\{|(y-x)\cdot \nu_1|\leq \frac{1}{2}\}$, from \ref{e:funzione salto}, \eqref{e:crescita funz recess} and \eqref{e:prima equaz lemma dei prob di minimo} we get
 \begin{align}\label{e:terza equaz lemmino problemi di minimo}
    & \minprobs^{f^{\infty}}(\overline{u}_{x,\zeta,\nu_1},Q^{\nu_2}_{r_2}(z))\leq \Functmins^{f^{\infty}}(\hat u,\hat v,Q^{\nu_2}_{r_2}(z)) \nonumber
    \\ & \leq 
    S^{f^\infty}(u,v,Q^{\nu_1}_{r_1}(x))+ \int_{Q^{\nu_2}_{r_2}(z)\setminus Q^{\nu_1}_{r_1}(x)}f^{\infty}({{y}},\nabla \hat u) \dy \nonumber 
    \\ & \leq \minprobs^{f^{\infty}}(\overline{u}_{x,\zeta,\nu_1},Q^{\nu_1}_{r_1}(x)) +\eta r_1^{n-1}+ C {{|\zeta|}}\|\overline{\textup{u}}'\|_{L^{\infty}(\R)}\calL^{n}((Q^{\nu_2}_{r_2}(z)\setminus Q^{\nu_1}_{r_1}(x))\cap {\Sigma_{\nu_1,\sfrac12}} )\nonumber
    \\ & \leq \minprobs^{f^{\infty}}(\overline{u}_{x,\zeta,\nu_1},Q^{\nu_1}_{r_1}(x)) +\eta r_1^{n-1}+  C{{|\zeta|}}\|\overline{\textup{u}}'\|_{L^{\infty}(\R)}\calL^{n}((Q^{\nu_1}_{r_3}(z)\setminus Q^{\nu_1}_{r_1}(x))\cap {\Sigma_{\nu_1,\sfrac12}} )\nonumber
      \\ & \leq \minprobs^{f^{\infty}}(\overline{u}_{x,\zeta,\nu_1},Q^{\nu_1}_{r_1}(x)) +\eta r_1^{n-1}+ C{{|\zeta|}}\|\overline{\textup{u}}'\|_{L^{\infty}(\R)}(r_3^{n-1}-r_1^{n-1})\nonumber 
       \\ & \leq \minprobs^{f^{\infty}}(\overline{u}_{x,\zeta,\nu_1},Q^{\nu_1}_{r_1}(x)) +\eta r_1^{n-1}+ C(n-1){{|\zeta|}}\|\overline{\textup{u}}'\|_{L^{\infty}(\R)}(r_{3}-r_1)r_{3}^{n-2}.
 \end{align}
 Therefore, recollecting \eqref{e:seconda equaz lemmino problemi di minimo} and \eqref{e:terza equaz lemmino problemi di minimo}, we deduce
\begin{align*}
 &  \minprobs^{f^{\infty}}(\overline{u}_{z,\zeta,\nu_2},Q^{\nu_2}_{r_2}(z))\leq \minprobs^{f^{\infty}}(\overline{u}_{x,\zeta,\nu_1},Q^{\nu_2}_{r_2}(z))+ 8R(n-1)|\zeta|r_2^{n-2} \\ & \leq \minprobs^{f^{\infty}}(\overline{u}_{x,\zeta,\nu_1},Q^{\nu_1}_{r_1}(x)) +\eta r_1^{n-1}+ C(n-1){{|\zeta|}}\|\overline{\textup{u}}'\|_{L^{\infty}(\R)}(r_{3}-r_1)r_{3}^{n-2} + 8R(n-1)|\zeta|r_2^{n-2}
\end{align*}
and therefore the claim.
 \end{proof}

We are now in a position to establish some properties satisfied by $g\homm$.
\begin{proposition}\label{p:la ghom}
Let $f\in \mathcal{F}(C,\alpha)$ and assume that for every $x\in \R^n$, $\zeta \in \R^N$, and $\nu \in \Sf^{n-1}$ the limit
\begin{equation}\label{e:ghom-ass}
      \lim_{r\to +\infty}\frac{\minprobs^{f^{\infty}}(u_{rx,\zeta,\nu},Q^{\nu}_{r}(rx))}{r^{n-1}}=:g\homm(\zeta,\nu)
\end{equation} 
exists (and is independent of $x$). 
Then, $g\homm$ satisfies the following properties:
 \begin{itemize}
    \item[(i)] for every $\zeta_1,\zeta_2\in \R^N$ and every $\nu \in \Sf^{n-1}$ 
    \begin{equation}\label{e:la ghom è lipschitz}
        |g\homm(\zeta_1,\nu)-g\homm(\zeta_2,\nu)|\leq C\calH^{n-1}(\partial Q_1)|\zeta_1-\zeta_2|;
    \end{equation}
    \item[(ii)] $g\homm:\R^{N}\times \Hat \Sf^{n-1}_{\pm}\to [0,+\infty)$ is continuous;
    \item[(iii)] for every $\zeta\in \R^N$ and every $\nu \in \Sf^{n-1}$
    \begin{equation}\label{e:crescita ghom}
        \frac{2|\zeta|}{C(|\zeta|+2)}\leq g\homm(\zeta,\nu)\leq \frac{2C|\zeta|}{|\zeta|+2};
    \end{equation}
    \item[(iv)] for every $\zeta \in \R^N$ and every $\nu \in \Sf^{n-1}$
    \begin{equation*}
        g\homm(\zeta,\nu)=g\homm(-\zeta,-\nu).
    \end{equation*}
    \end{itemize}
\end{proposition}

\begin{proof}
To prove (i) fix $\nu \in \Sf^{n-1}$ and $\zeta_1,\zeta_2 \in \R^{N\times n}$.  Thanks to Lemma~\ref{l:bordi regolarizzati} we get
\begin{equation*}
   |\minprobs^{f^{\infty}}(\overline{u}_{\zeta_1,\nu},Q^{\nu}_r)-\minprobs^{f^{\infty}}(\overline{u}_{\zeta_2,\nu},Q^{\nu}_r)|\leq C\int_{\partial Q^{\nu}_r}|\overline{u}_{\zeta_1,\nu}-\overline{u}_{\zeta_2,\nu}|\dd \calH^{n-1}. 
\end{equation*}
By the definition of $\overline{u}_{\zeta,\nu}$ we have 
\begin{equation*}
  \int_{\partial Q^{\nu}_r}|\overline{u}_{\zeta_1,\nu}-\overline{u}_{\zeta_2,\nu}|\dd \calH^{n-1}= \int_{\partial Q^{\nu}_r}|\zeta_1-\zeta_2|\overline{\textup{u}}\left(y\cdot \nu\right) \dd \calH^{n-1}(y)\leq \calH^{n-1}(\partial Q_1)|\zeta_1-\zeta_2|r^{n-1}.
\end{equation*}
Then we conclude by \eqref{e:ghom-ass} also noticing that by Corollary~\ref{c:comparison minimum problems} we have
\begin{equation*}
 g\homm(\zeta,\nu)=\lim_{r\to +\infty}\frac{\minprobs^{f^{\infty}}(\overline{u}_{rx,\zeta,\nu},Q^{\nu}_{r}(rx))}{r^{n-1}},
\end{equation*}
for every $x\in \R^n$, $\zeta \in \R^N$, and $\nu \in \Sf^{n-1}$.

\smallskip

To prove (ii) we preliminarily show that $g\homm(\zeta,\cdot):\Hat \Sf^{n-1}_{\pm}\to [0,+\infty)$ is continuous for every $\zeta \in \R^N$. Fix $\zeta \in \R^N$, $\nu \in \Hat \Sf^{n-1}_{\pm}$ and a sequence $(\nu_j)_{j\in\N}$ in $\Hat \Sf^{n-1}_{\pm}$ such that $\nu_j\to \nu$ as $j\to +\infty$. For every $\delta\in (0,1/2)$, by the continuity of the map $\nu \mapsto R_{\nu}$ on $\Hat \Sf^{n-1}_{\pm}$
(cf. \ref{e:Rnu} of the notation list), there exists $j_{\delta}$ such that for every $r>0$ and every $j\geq j_{\delta}$
\begin{equation}\label{e:cubi compattamente contenuti prop ghom}
    Q^{\nu}_r \subset \subset Q^{\nu_j}_{(1+\delta)r} \subset \subset Q^{\nu}_{(1+2\delta)r}\,.
\end{equation}
Setting $\kappa_j:=\max \{|R_{\nu_j}(e_i)\cdot \nu| \; : \; i=1,\dots,n-1\}$, we have that $\kappa_j \to 0$ as $j\to +\infty$, by the continuity of the map $\nu \mapsto R_{\nu}$ on $\Hat \Sf^{n-1}_{\pm}$. Letting $y\in \overline{Q^{\nu_j}_{r(1+\delta)}}$, then $y=y'+(y\cdot \nu_j)\nu_j$ where 
 \begin{equation*}
     y'\in R_{\nu_j}\Big(\big[-\frac{r}{2}(1+\delta),\frac{r}{2}(1+\delta)\big]^{n-1}\times \{0\}\Big).
 \end{equation*}
 In particular $(y\cdot \nu_j)(\nu \cdot \nu_j)=y\cdot \nu-y'\cdot \nu$ and thus, if $|y\cdot \nu|\leq \frac{1}{2}$ and $j$ is large enough, we get 
 \begin{equation}\label{e:stima cornicione transizione}
   \textstyle  |y\cdot \nu_j|\leq \frac{|y'\cdot \nu|}{|\nu_j\cdot \nu|}+\frac{1}{2\nu_j \cdot \nu}\leq\frac{(n-1)\kappa_jr (1+\delta)+1}{2(1-\delta)}=K(\delta)r\kappa_j+1,
 \end{equation}
 where $K(\delta):=\frac{(n-1) (1+\delta)}{2(1-\delta)}$. Applying Lemma \ref{l:lemma prob di minimo superficie su cubi diversi} with $R=K(\delta)r\kappa_j+1$, we deduce
 \begin{align*}
    & \minprobs^{f^{\infty}}(\overline{u}_{\zeta,\nu_j},Q^{\nu_j}_{r(1+\delta)})\leq  \minprobs^{f^{\infty}}(\overline{u}_{\zeta,\nu},Q^{\nu}_{r})+\eta r^{n-1} \\ & + \tilde K(2\delta (1+2\delta)^{n-2}r^{n-1}+(K(\delta)r\kappa_j+1)|\zeta|(1+\delta)^{n-2}r^{n-2}),
 \end{align*}
 where $\tilde K$ depends only on $n$ and $C$. Consequently, letting the $r\to +\infty$, appealing to \eqref{e:ghom-ass} and to {{Corollary~\ref{c:comparison minimum problems}}}, we get
\begin{equation*}
 (1+\delta)^{n-1}g\homm(\zeta,\nu_j)  \leq  g\homm(\zeta,\nu)+\eta + \tilde K(1+2\delta)^{n-2}(2\delta+K(\delta)\kappa_j|\zeta|).
 \end{equation*}
 Taking the  $\limsup$ for $j\to +\infty$ we have
 \begin{equation*}
    (1+\delta)^{n-1} \limsup_{j\to +\infty}g\homm(\zeta,\nu_j)\leq g\homm(\zeta,\nu)+\eta+2\tilde K\delta(1+2\delta)^{n-2}
 \end{equation*}
 thus letting $\eta,\delta \to 0$ we obtain
 \begin{equation*}
    \limsup_{j\to +\infty}g\homm(\zeta,\nu_j)\leq g\homm(\zeta,\nu). 
 \end{equation*}
 An analogous argument, using the cube $Q^{\nu_j}_{(1-\delta)r}$, shows that
\begin{equation*}
    g\homm(\zeta,\nu)\leq \liminf_{j\to +\infty} g\homm(\zeta,\nu_j)
\end{equation*}
and hence the claim. 

To establish the continuity with respect to both variables, consider a sequence $(\zeta_j)_{j\in\N}$ in $\R^{N\times n}$ such that $\zeta_j \to \zeta$. 
Thanks to \eqref{e:la ghom è lipschitz}, we have that
\begin{align*}
   |g\homm(\zeta,\nu)-g\homm(\zeta_j,\nu_j)|&\leq |g\homm(\zeta,\nu)-g\homm(\zeta,\nu_j)|+ |g\homm(\zeta,\nu_j)-g\homm(\zeta_j,\nu_j)|
   \\ & \leq |g\homm(\zeta,\nu)-g\homm(\zeta,\nu_j)|+ C\calH^{n-1}(\partial Q_1)|\zeta-\zeta_j|
\end{align*}
and therefore we get (ii).
\smallskip 

To prove (iii) fix $\zeta \in \R^N$ and $\nu \in \Sf^{n-1}$, and recall that by \eqref{e:ghom-ass} and the spatial homogeneity of $g\homm$
we have that
\begin{equation}\label{e:ghom sup term la prop ghom}
    g\homm(\zeta,\nu)=\lim_{r\to +\infty}\frac{\minprobs^{f^{\infty}}(u_{\zeta,\nu},Q^{\nu}_r)}{r^{n-1}}\,.
\end{equation}
We notice that for every $r>0$ and $M\in\N$ we have 
\[
\frac{\minprobs^{f^{\infty}}(u_{\zeta,\nu},Q^{\nu}_{Mr})}{{{(Mr)}}^{n-1}}
\leq\frac{\minprobs^{f^{\infty}}(u_{\zeta,\nu},Q^{\nu}_r)}{r^{n-1}}\,.
\]
Indeed, assume for simplicity $\nu=e_n$, then if $(u,v)$ is a competitor for $\minprobs^{f^{\infty}}(u_{\zeta,e_n},Q^{e_n}_r)$ then
$(u_M,v_M)$ defined by $(u,v)(x-r{\tt i})$ for $x\in r{\tt i}+Q^{e_n}_r$, for ${\tt i}\in\Z^{n-1}\times\{0\}$ with components in $[-M+1,M-1]$, 
and equal to $u_{\zeta,e_n}$ otherwise on $Q^{e_n}_{Mr}$ is a competitor for  $\minprobs^{f^{\infty}}(u_{\zeta,e_n},Q^{e_n}_{Mr})$ 
with 
$\Functmins^{f^{\infty}}(u_M,v_M,Q^{e_n}_{Mr})=M^{n-1}\Functmins^{f^{\infty}}(u,v,Q^{e_n}_{r})$. 
Thus, we infer that
\begin{equation}\label{e:lim=inf}
\lim_{r\to +\infty}\frac{\minprobs^{f^{\infty}}(u_{\zeta,\nu},Q^{\nu}_r)}{r^{n-1}}
=\inf_{r>0}\frac{\minprobs^{f^{\infty}}(u_{\zeta,\nu},Q^{\nu}_r)}{r^{n-1}}\,.
\end{equation}
Moreover, by \ref{e:crescita lineare} and $C\geq 1$,
{recalling the definiton in \eqref{e:FunzHomos},
we have that

\[
C^{-1} S^{|\cdot|}(u,v,Q^{\nu}_r)\leq S^{f^\infty}(u,v,Q^{\nu}_r)
\leq C S^{|\cdot|}(u,v,Q^{\nu}_r)
\]
Therefore, by means of \eqref{e:lim=inf} we conclude that}
\[
C^{-1}\inf_{r>0}\frac{\minprobs^{|\cdot|}(u_{\zeta,\nu},Q^{\nu}_r)}{r^{n-1}}\leq
g\homm(\zeta,\nu)\leq C\inf_{r>0}\frac{\minprobs^{|\cdot|}(u_{\zeta,\nu},Q^{\nu}_r)}{r^{n-1}}\,.
\]
Finally, \cite[Lemma~3.8 and Remark~3.9]{AlFoc} yield that
\[
\inf_{r>0}\frac{\minprobs^{|\cdot|}(u_{\zeta,\nu},Q^{\nu}_r)}{r^{n-1}}=g({{|\zeta|}})\,, 
\]
where using the same notation as in \cite{AlFoc}
\begin{equation}\label{e:g esplicita AlFoc}
   g(s):=\min_{[0,1]}\textstyle\{t^2s+ 4\int_{t}^1(1-\lambda)\dd \lambda \}=\displaystyle\frac{2s}{s+2}\,. 
\end{equation}
Eventually, (iv) is a direct consequence of the identity $R_{\nu}(Q_1)=R_{-\nu}(Q_1)$ and of the fact that
$u=u_{-\zeta,-\nu}$ on $\partial Q^{\nu}_r$ if and only if $u+\zeta=u_{\zeta,\nu}$ on $\partial Q^{\nu}_r$ for every $\zeta \in \R^N$, $\nu \in \Sf^{n-1}$, $r>0$, and $u\in W^{1,1}(Q^{\nu}_r,\R^N)$.
\end{proof}

\section{Deterministic homogenisation}

To prove the stochastic homogenisation result in Theorem~\ref{t:Stochastic homogenisation} we follow the same proof strategy as in \cite{CDMSZStochom, CDMSZGlobal}. 
To this end, we preliminarily work in a deterministic framework (where $\omega\in \Omega$ is regarded as fixed) and prove a homogenisation result without assuming any periodicity of the integrand. Then, in Section \ref{s:stochastic}, the deterministic homogenisation result at fixed $\omega$ will be used in combination with the Subadditive Ergodic Theorem, Theorem \ref{t:ergodic theorem}, to derive an almost sure $\Gamma$-convergence result for the random functionals $F_\eps(\omega)$. 

The main result of this section is stated in the following theorem. 

\begin{theorem}[Deterministic homogenisation]\label{t:Deterministic Gamma-conv}
Let $f\in \mathcal{F}(C,\alpha)$ and consider the phase-field functionals  
 $\Functeps_{\eps}:L^1_{\textup{loc}}(\R^n,\R^{N+1})\times \corA\longrightarrow [0,+\infty]$ given by
\begin{equation}\label{e:funzionali approssimanti}
\Functeps_{\eps}(u,v,A):=\begin{cases}
        \displaystyle \int_A \hspace{-1mm}\textstyle{(v^2 f(\frac{x}{\eps},\nabla u)+\frac{(1-v)^2}{\eps}+\eps|\nabla v|^2)}\dx, &  (u,v)\in W^{1,1}(A,\R^N)\times W^{1,2}(A,[0,1]) \\
        +\infty  & \textup{otherwise}.
    \end{cases}\end{equation}
Assume that
\begin{enumerate}
\item for every $x\in \R^n$, $\xi \in \R^{N\times n}$, $\nu \in \Sf^{n-1}$ and $k\in \N$ the limit
\begin{equation}\label{e:fhom}
 \lim_{r\to +\infty}\frac{\minprobv^f(\ell_{\xi},Q^{\nu,k}_{r}(rx))}{k^{n-1}r^{n}}=:f\homm(\xi)
\end{equation}
exists and is independent of $x,\nu$ and $k$; \\
\item for every $x\in \R^n$, $\zeta\in \R^N$ and $\nu\in \Sf^{n-1}$ the limit
\begin{equation}\label{e:ghom}
   \lim_{r\to +\infty}\frac{\minprobs^{f^{\infty}}(u_{rx,\zeta,\nu},Q^{\nu}_{r}(rx))}{r^{n-1}}=:g\homm(\zeta,\nu)
\end{equation}
exists and is independent of $x$.
\end{enumerate}
Let, moreover, $f^{\infty}\homm$ be the recession function of $f\homm$. Then, for every $A\in \corA$ and every $(u,v)\in L^{1}_{\textup{loc}}(\R^n,\R^{N+1})$ we have
\begin{equation*}
 \Gamma(L^1_{\textup{loc}}(\R^n,\R^{N+1}))\text{-}\lim_{\eps \to 0}\Functeps_{\eps}(u,v,A)=\Functeps\homm(u,v,A),   
\end{equation*}
where $\Functeps\homm:L_{\textup{loc}}^1(\R^n,\R^{N+1})\times \corA\longrightarrow [0,+\infty]$ is the functional defined by
\begin{equation}\label{e:det-G-lim}
    \Functeps\homm(u,v,A):=\int_A f\homm(\nabla u) \dx + \int_A f^{\infty}\homm\textstyle{(\frac{\dd D^cu}{\dd |D^cu|})}\dd |D^cu| + \displaystyle\int_{J_u\cap A}g\homm([u],\nu_u) \dd \calH^{n-1},
\end{equation}
if $u\in GBV(A,\R^N)$ and $v=1\; \calL^n$-a.e in $A$, $\Functeps\homm(u,v,A)=+\infty$, otherwise. 
\end{theorem}

To prove Theorem~\ref{t:Deterministic Gamma-conv} we use a standard approach in homogenisation theory based on the compactness of $\Gamma$-convergence and on the so-called localization method (cf.\ \cite{BrDeF, Dalmaso1993}). Namely, we first show that for every 
infinitesimal sequence $(\eps_j)_{j\in\N}$, up to a subsequence, the functionals $\Functeps_{\eps_j}$, defined in \eqref{e:funzionali approssimanti}, $\Gamma$-converge to some abstract functional~$\Fsucc$. Then, we prove
that $\Fsucc$ admits an integral representation as in \eqref{e:det-G-lim} on $BV(A,\R^N)$, for every $A\in\corA$. Eventually, thanks to \eqref{e:fhom} and \eqref{e:ghom} we deduce that
$\Fsucc$ does not depend on the extracted subsequence, and hence the homogenisation result for $(F_\eps)$ follows by the {{Urysohn}} property of $\Gamma$-convergence.

\medskip

We start by proving the following abstract $\Gamma$-convergence result. 

\begin{theorem}[$\Gamma$-convergence and properties of the $\Gamma$-limit]\label{t:Sottosucc gamma-conv.}
Let $f \in \mathcal{F}(C,\alpha)$ and $\Functeps_{\eps}$ be as in \eqref{e:funzionali approssimanti}, then there exists a subsequence $(\eps_j)_{j\in\N}$ and a functional $\Fsucc:L^1_{\textup{loc}}(\R^n,\R^{N+1})\times \corA\longrightarrow [0,+\infty]$ such that, for every $A\in \corA$ and every $u\in L^1_{\textup{loc}}(\R^n,\R^N)$ with $u\in BV(A,\R^N)$
\begin{equation}\label{e:Gamma-conv sottosucc}
     \Gamma(L^1_{\textup{loc}}(\R^n,\R^{N+1}))\text{-}\lim_{j \to +\infty}\Functeps_{\eps_j}(u,1,A)=\Fsucc(u,1,A).
\end{equation}
Moreover $\Fsucc$ satisfies the following properties:
\begin{enumerate}
    \item (\textup{locality}) $\Fsucc(u_1,1,A)=\Fsucc(u_2,1,A)$ for every $A\in \corA$ and every $u_1, u_2 \in L^1_{\textup{loc}}(\R^n,\R^{N})$ such that $u_1=u_2 \;\calL^n$-a.e in $A$;
    \item (\textup{semicontinuity}) for every $A\in \corA$ the functional $\Fsucc(\cdot,1,A):L^1_{\textup{loc}}(\R^n,\R^{N})\longrightarrow [0,+\infty]$ is lower semicontinuous;
    \item (\textup{upper bound}) for every $A\in \corA$ and every $u\in L^1_{\textup{loc}}(\R^n,\R^N)$ with $u\in BV(A,\R^N)$ there holds 
    \begin{equation}\label{e:upper bound Fsucc}
        \Fsucc(u,1,A) \leq C(\calL^n(A)+|Du|(A));
    \end{equation}
    \item (\textup{lower bound}) for every $M>0$ there exists $C_M>0$ such that for every $A\in \corA$ and every $u\in L^1_{\textup{loc}}(\R^n,\R^N)$ with $u\in BV(A,\R^N)$ and $\|u\|_{L^{\infty}(A,\R^N)}\leq M$ we have 
    \begin{equation}\label{e:lower bound Fsucc}
       C_M |Du|(A) \leq \Fsucc(u,1,A) ;
    \end{equation}
    \item (\textup{measure property}) for every $A\in \corA$, every $u\in L^1_{\textup{loc}}(\R^n,\R^N)$ such that $u\in BV(A,\R^N)$, the set function $\Fsucc(u,1,\cdot):\corA(A) \to [0,+\infty]$ is the restriction of a finite Radon measure on $A$;
    \item (\textup{translation invariance in $u$}) for every $A\in \corA$ and every $u\in L^1_{\textup{loc}}(\R^n,\R^{N})$ we have
    \begin{equation*}
        \Fsucc(u+z,1,A)=\Fsucc(u,1,A),
    \end{equation*}
    for every $z\in \R^N$.
    \end{enumerate}
\end{theorem}

\begin{proof}
Given any sequence of positive real numbers decreasing to zero \cite[Theorem~16.9]{Dalmaso1993} provides us with a subsequence $(\eps_j)$ such that
\begin{equation*}
    \overline{\Gamma}\textup{-}\lim_{j\to +\infty}\Functeps_{\eps_j}=\Fsucc,
\end{equation*}
where $\Fsucc:L^1_{\textup{loc}}(\R^n,\R^{N+1})\times \corA\longrightarrow [0,+\infty]$ is increasing, inner regular, and superadditive as a set function and lower semicontinuous in $L^1_{\textup{loc}}(\R^n,\R^{N+1})$ as a functional. By definition of $\overline{\Gamma}$-convergence, we have
\begin{equation}\label{e:gamma bar conv uguaglianza}
    \Fsucc'_{-}=\Fsucc=\Fsucc''_{-},
\end{equation}
where
\begin{equation*}
    \Fsucc'(\cdot,A)=\Gamma\text{-}\liminf_{j\to \infty}\Functeps_{\eps_j}(\cdot,A),  \quad\Fsucc''(\cdot,A)=\Gamma\text{-}\limsup_{j\to \infty}\Functeps_{\eps_j}(\cdot,A),
\end{equation*}
and
\begin{equation*}
   \Fsucc'_{-}(\cdot,A):=\sup_{A' \subset \subset A}\Fsucc'(\cdot,A'), \quad   \Fsucc''_{-}(\cdot,A):=\sup_{A' \subset \subset A}\Fsucc''(\cdot,A').
\end{equation*}
The locality property and the translation invariance of $\Fsucc$ are direct consequences of \eqref{e:gamma bar conv uguaglianza}, and of the locality and translation invariance of $\Fsucc'$ and $\Fsucc''$. 

Arguing exactly as in \cite[Lemma~5.1]{AlFoc}, for every $A\in \corA$, every $u\in L^1_{\textup{loc}}(\R^n,\R^N)$ such that $u\in BV(A,\R^N)$, every $A',A''\in \corA(A)$ and every $B'\subset \subset A'$, with $B'\in \corA(A)$ we can obtain that
\begin{equation*}
    \Fsucc''(u,1,B'\cup A'') \leq \Fsucc''(u,1,A')+ \Fsucc''(u,1,A''),
\end{equation*}
from which we can easily deduce that the inner regular envelope $\Fsucc''_{-}(u,1,\cdot)$ is subadditive on $\corA(A)$. Therefore,  thanks to the De Giorgi-Letta Criterion, we infer that the set function $ \Fsucc(u,1,\cdot):\corA(A) \to [0,+\infty] $ is the restriction to the open sets of a Borel measure on $A$. 

For every $A\in \corA$, and every $u\in L^1_{\textup{loc}}(\R^n,\R^N)$ with $u\in BV(A,\R^N),$ in view of \ref{e:crescita lineare}, we obtain 
\begin{equation*}
    \Fsucc''(u,1,A) \leq C(|Du|(A)+\calL^n(A)).
\end{equation*}
Hence, in particular
\begin{equation*}
   \Fsucc(u,1,A)\leq C(|Du|(A)+\calL^n(A)), 
\end{equation*}
so that using \cite[Proposition~18.6]{Dalmaso1993} we get
\begin{equation*}
    \Fsucc'(u,1,A)=\Fsucc(u,1,A)=\Fsucc''(u,1,A).
\end{equation*}
The latter eventually provides the $\Gamma$-convergence statement in \eqref{e:Gamma-conv sottosucc}.

{{To prove the lower bound inequality in item (iv) we argue as follows.
By \ref{e:crescita lineare}, a comparison argument and \cite[Proposition~4.1]{AlFoc}
(see \cite[Remark~3.5]{AlFoc}) we infer that
\[
\Fsucc(u,1,A)\geq C^{-1}|Du|(A\setminus J_u)+C^{-1}\int_{J_u\cap A}g(|u^+-u^-|)|\dd\calH^{n-1}
\]
where $g(s)=\frac{2s}{s+2}$ for all $s\geq 0$ (cf. \eqref{e:g esplicita AlFoc}).
Finally, note that for every $M>0$ there is $C_M>0$ such that
$g(s)\geq C_Ms$ for every $s\in[0,2M]$, and the conclusion follows at once.
}}
\end{proof}
The next three subsections are devoted to the proof of 
Theorem~\ref{t:Deterministic Gamma-conv}. Namely, in subsections~\ref{s:volume} - \ref{s:cantor} we identify, respectively, the three measure derivatives 
\[
\frac{\dd \Fsucc(u,1,\cdot)}{\dd \calL^n}, \quad  \frac{\dd \Fsucc(u,1,\cdot)}{\dd \calH^{n-1} \res J_u}, \quad \text{and} \quad \frac{\dd \Fsucc (u,1,\cdot)}{\dd |D^c u|}.
\]
In fact, we will prove that under the assumptions of Theorem~\ref{t:Deterministic Gamma-conv}, for every $A\in \corA$ and 
$u\in BV(A,\R^N)$ the following three equalities hold:
\begin{align*}
    \frac{\dd \Fsucc(u,1,\cdot)}{\dd \calL^n}(x)&=f\homm(\nabla u(x)) \quad \textup{for $\calL^n$-a.e. $x\in A$},  \\   \frac{\dd \Fsucc(u,1,\cdot)}{\dd \mathcal{H}^{n-1}}(x)&= g\homm([u](x),\nu_{u}(x)) \quad \textup{for $\mathcal{H}^{n-1}$-a.e. $x\in J_u\cap A$}, \\  \frac{\dd \Fsucc(u,1,\cdot)}{\dd |D^cu|}(x) &= f^{\infty}\homm({{\textstyle\frac{\dd D^c u}{\dd |D^c u|}(x)}}) \quad \textup{for $|D^cu|$-a.e. $x\in A$} .
\end{align*}
Since in the equalities above the right-hand sides do not depend on the subsequence $(\eps_j)_{j\in\N}$, we will be able to conclude that $\widehat F$ is subsequence independent and therefore the $\Gamma$-convergence result holds for the whole sequence $(\Functeps_{\eps})$ (cf.\ Theorem~\ref{t:Deterministic Gamma-conv}).

The strategy to prove the identities above uses, on one hand, the global method for relaxation in $BV$ \cite{BouchitteFonsecaMascarenhas} and, on the other hand, a direct (although involved) comparison argument. 

\medskip

For later use it is useful to recall the following notation: 
let $U\in \corA_{\infty}$ and let $G:BV(U,\R^N)\times \corA(U)\longrightarrow [0,\infty)$; 
for every $(w,A)\in BV(U,\R^N)\times\corA_\infty(U)$ set
\begin{equation}\label{e:equaz minfunz}
  m_G(w,A):=\inf\{G(u,A) \; : \; u\in L^1_{\textup{loc}}(\R^n,\R^{N}), \; u\in BV(A,\R^N), \; u=w \; \textup{on $\partial A$}\}.
\end{equation}
In addition, we use the notation $sc^-(L^1)G$ for the relaxation of $G$ with respect to the $L^1$ convergence, namely
$sc^-(L^1)G(u,A):=\Gamma(L^1)\hbox{-}\lim_jG(u;A)$ (cf.\ \cite{Dalmaso1993}).

In what follows we will use in several instances a truncation lemma that follows from  De~Giorgi's slicing and averaging argument on the codomain 
(see for instance \cite[Proposition~6.2]{AlFoc} and \cite[Proposition~3.2]{ContiFocardiIurlano2024}). We give here a detailed proof of it since the statement is slightly
different from the standard one. In particular, in Propositions~\ref{p:Homogenised volume integrand} and \ref{p:homogenised cantor integrand} we
choose $v\equiv 1$, while in Proposition~\ref{p:homogenised surface integrand} it is important that the constant $\gamma$ in 
the growth condition below equals $0$. We recall the notation $\mathcal{T}_k$ for the smooth truncation operators and $a_k$ for the related sequence
introduced in \ref{i:truncation}.
\begin{lemma}\label{l:truncation lemma}
Let $A\in \corA$ and $\mathcal{G}: GBV(A,\R^N)\times L^1(A,[0,1])\to [0,\infty]$ be the functional defined by
\begin{equation*}
    \mathcal{G}(u,v):=\int_A v^2(x)g(x,\nabla u(x)) \dx
\end{equation*}
where $g:\R^n\times \R^{N\times n}\to [0,\infty)$ is a Borel function for which there exists $\gamma\in [0,\infty)$ such that 
\begin{equation}\label{e:lemma di troncamento 2 equaz 1}
    c^{-1}|\xi|\leq g(x,\xi)\leq c(|\xi|+\gamma)
\end{equation}
for every $(x,\xi)\in \R^n\times \R^{N\times n}$, and for some $c>0$. 

Then for every $M\in \N$ and $(u,v)\in GBV(A,\R^N)\times L^1(A,[0,1])$ there exists $k\in \{M+1,\dots,2M\}$ such that $\mathcal{T}_k(u)\in BV\cap L^\infty(A,\R^N)$ with $\|\mathcal{T}_k(u)\|_{L^\infty}\leq a_{k+1}$, $\mathcal{H}^{n-1}(J_{\mathcal{T}_k(u)}\cap A)\leq \mathcal{H}^{n-1}(J_u\cap A)$ and
\begin{equation*}
\mathcal{G}(\mathcal{T}_k(u),v)\leq    \Big(1+\frac{c^2}{M}\Big)\mathcal{G}(u,v)+ \gamma c\calL^n(\{|u|>a_{M}\})\,.
\end{equation*}
\end{lemma}

\begin{proof}
Let us fix $M\in \N$ and $(u,v)\in GBV(A,\R^N)\times L^1(A,[0,1])$ with $\mathcal{G}(u,v)<\infty$, otherwise the claim follows trivially. 
By averaging, there exists $k\in \{M+1,\dots,2M\}$ such that 
\begin{equation}\label{e:lemma di troncamento 2 equaz 2}
    \int_{A\cap \{a_k\leq |u|<a_{k+1}\}} v^2(x)g(x,\nabla u(x))\dx \leq \frac{1}{M}\mathcal{G}(u,v).
\end{equation}
By the properties of $GBV$ functions and the very definition of $\mathcal{T}_k$, we have that $\mathcal{T}_k(u)$ belongs to $BV(A,\R^N)\cap L^\infty(A,\R^N)$ with $\|\mathcal{T}_k(u)\|_{L^\infty}\leq a_{k+1}$, $J_{\mathcal{T}_k(u)}\cap A\subseteq  J_u\cap A$ and $\nabla (\mathcal{T}_k(u))(x)=\nabla \mathcal{T}_k(u(x))\nabla u(x)$ for $\calL^n$-a.e. $x\in A$. Furthermore, being ${\rm Lip}(\mathcal{T}_k)\leq 1$, we can check for every $y,v\in \R^N$ with $|v|=1$ that $|(\nabla \mathcal{T}_k(y))v|=|\partial_v \mathcal{T}_k(y)|\leq 1$ that provides $\|\nabla \mathcal{T}_k(y)\|_2\leq 1$ for every $y\in \R^N$ (here $\|\cdot\|_2$ stands for the matrix norm on $\R^{N\times N}$ induced by $|\cdot|$ on $\R^N$) and consequently 
\begin{equation}\label{e:lemma di troncamento 2 equaz 3}
   |\nabla (\mathcal{T}_k(u))(x)|\leq |\nabla u(x)| \quad \quad \textup{for $\calL^n$-a.e. $x\in A$}. 
\end{equation}
In particular, in virtue of $\mathcal{T}_k(y)={{y}}$ on $\{|y|< a_k\}$ and $\mathcal{T}_k(y)=0$ on $\{|y|\geq a_{k+1}\}$, we obtain
\begin{align*}
   & \mathcal{G}(\mathcal{T}_k(u),v)=\int_{A\cap \{|u|< a_k\}}v^2(x)g(x,\nabla u) \dx + \int_{A\cap \{a_k\leq |u|<a_{k+1} \}}v^2(x)g(x,\nabla (\mathcal{T}_k(u))) \dx + \\ & + \int_{A\cap \{|u|\geq a_{k+1}\}}v^2(x)g(x,0) \dx \leq \mathcal{G}(u,v)+ c \int_{A\cap \{a_k\leq |u|<a_{k+1} \}}v^2(x)|\nabla (\mathcal{T}_k(u))| \dx \\ & + c\gamma  \calL^n({|u| \geq a_{k}}) \leq \mathcal{G}(u,v)+ c\int_{A\cap \{a_k\leq |u|<a_{k+1} \}}v^2(x)|\nabla (u)| \dx + c\gamma  \calL^n({|u| > a_{M}}) \\ & \leq \Big(1+\frac{c^2}{M}\Big)\mathcal{G}(u,v)+c\gamma \calL^n({|u| > a_{M}}),
\end{align*}
where in the first inequality we used \eqref{e:lemma di troncamento 2 equaz 1}, in the second one \eqref{e:lemma di troncamento 2 equaz 3}, and finally in the last one \eqref{e:lemma di troncamento 2 equaz 1} and \eqref{e:lemma di troncamento 2 equaz 2}. 
\end{proof}
 
\subsection{Identification of the volume term}\label{s:volume}

This section is devoted to identify the measure derivative  
$\frac{\dd \Fsucc(u,1,\cdot)}{\dd \calL^n}$ with $f\homm$.
\begin{proposition}[Homogenised volume integrand]\label{p:Homogenised volume integrand}
Let $f\in \mathcal F(C,\alpha)$ satisfy \eqref{e:fhom-ass}. 
Let $\Fsucc$ be as in \eqref{e:Gamma-conv sottosucc}. 
Then,
for every $A\in \corA$ and every $u\in L^1_{\textup{loc}}(\R^n,\R^N)$, with $u\in BV(A,\R^N)\cap L^{\infty}(A,\R^N)$ there holds
\begin{equation*}
    \frac{\dd \Fsucc(u,1,\cdot)}{\dd \calL^n}(x)=f\homm(\nabla u(x)) \quad \textup{for $\calL^n$-a.e. $x\in A$},
\end{equation*}
where $f\homm$ is as in \eqref{e:fhom-ass}. 
\end{proposition}
To prove Proposition \ref{p:Homogenised volume integrand}, we need the two following technical lemmas.  
\begin{lemma}\label{l:lemma del rilassato}
Let $g\in \mathcal{F}(C,\alpha)$ be given and define $\hat g:\R^n\times \R^{N\times n}\to [0,\infty)$ as
\begin{equation}\label{e:limsup del teo bfm}
      \hat g(x,\xi):=\limsup_{\eta \to 0} \inf\left\{\int_{Q}g(x+\eta z,\nabla w) \dd z \; : \; w-\ell_{\xi}\in W^{1,1}_0(Q,\R^N)\right\},
\end{equation}
Let $A\in \corA_{\infty}$ and let $E^g(\cdot, A)$ and $E^{\hat g}(\cdot, A)$ be defined as in 
\eqref{e:FunzHomob} with $h$ replaced by $g$ and $\hat g$, respectively. Moreover, consider the functionals
$F^g,F^{\hat g}:L^1(A,\R^N)\longrightarrow [0,\infty]$ given by
\begin{equation*}
   F^g(u):= \begin{cases}
      E^{g}(u,A) \; & \textup{if $u\in W^{1,1}(A,\R^N)$} \\
      +\infty \; & \textup{otherwise}
    \end{cases}
    , \quad  F^{\hat g}(u):= \begin{cases}
      E^{\hat g}(u,A) \; & \textup{if $u\in W^{1,1}(A,\R^N)$} \\
      +\infty \; & \textup{otherwise}.
    \end{cases}
\end{equation*}
Then the following statements hold:
\begin{enumerate}
 \item if $g$ is $1$-homogeneous in $\xi$, then the same holds for $\hat g$; 
 \item there exists $H\subseteq \R^n$ with $\calL^n(H)=0$ such that for every $x\in \R^n\setminus H$ and every $\xi \in \R^{N\times n}$
    \begin{equation*}
        \hat g(x,\xi)\leq g(x,\xi);
    \end{equation*}
       \item  for every $u\in W^{1,1}(A,\R^N)$
    \begin{equation*}
        sc^{-}(L^1)F^g(u)=F^{\hat g}(u);
    \end{equation*}
    \item for every $u\in BV(A,\R^N)$
    \begin{equation*}
       \Big|  sc^{-}(L^1)F^g(u)-\int_A \hat g(x,\nabla u)\dx\Big|\leq C|D^su|(A);
    \end{equation*}
    \item for every $u\in L^1(A,\R^N)$
    \begin{equation*}
                sc^{-}(L^1)F^g(u)=sc^{-}(L^1)F^{\hat g}(u);
    \end{equation*}
    \item for every $\xi\in \R^{N\times n}$
    \begin{equation*}
        \minprobv^g(\ell_{\xi},A)=\minprobv^{\hat g}(\ell_{\xi},A).
    \end{equation*}
\end{enumerate}
\end{lemma}

\begin{proof}
Property $(i)$ readily follows from the definition of $\hat g$.
Instead, $(iii)$ and $(iv)$ are a direct consequence of \cite[Theorem~4.1.4]{BouchitteFonsecaMascarenhas}.

To prove $(ii)$ let $\xi \in \Q^{N\times n}$ be fixed, by definition we get 
\begin{equation*}
    \hat g(x,\xi)\leq \limsup_{\eta \to 0} \int_{Q}g(x+\eta z,\xi) \dd z = \limsup_{\eta \to 0} \frac{1}{\eta^n}\int_{Q_{\eta}(x)}g(z,\xi) \dd z 
\end{equation*}
for every $x\in \R^n$. Then, the Lebesgue Differentiation Theorem provides us with a set $H_\xi \subset \R^n$ such that $\mathcal L^n(H_\xi)=0$ and $\hat g(x,\xi)\leq g(x,\xi)$ for every $x \in \R^n \setminus H_\xi$. 
Therefore, we conclude by setting 
\[
H:=\bigcup_{\xi \in \Q^{N\times n}}H_\xi\,,
\]
and invoking the continuity of $g$, \ref{e:semicontinuità f e finf}, and the lower 
semicontinuity of $\hat g$ as the bulk energy density of the functional 
$sc^{-}(L^1)F^{\hat g}$.

The proof of $(v)$ follows straightforwardly from $(ii)$ and $(iii)$. 

To conclude the proof, we are left to show $(vi)$. We start noticing that in view of $(ii)$ we only need to prove that
\begin{equation*}
    \minprobv^g(\ell_{\xi},A)\leq \minprobv^{\hat g}(\ell_{\xi},A).
\end{equation*}
for every $\xi \in \R^{N\times n}$. 
To prove the inequality above, fix $\xi \in \R^{N\times n}$ and let $u\in W^{1,1}(A,\R^N)$ satisfy $u=\ell_{\xi}$ on $\partial A$. By $(iii)$ we can infer the existence of a sequence $(u_j)_{j\in\N} \subset W^{1,1}(A,\R^N)$ such that $u_j\to u$ in $L^1(A,\R^N)$ as $j\to \infty$ and 
\begin{equation*}
    \lim_{j\to \infty} E^g(u_j,A)=E^{\hat g}(u,A).
\end{equation*}
By \cite[Lemma 2.6 and Remark 2.7]{BouchitteFonsecaMascarenhas} we can find a sequence $(w_j)_{j\in\N} \subset W^{1,1}(A,\R^N)$ satisfying $w_j=\ell_{\xi}$ on $\partial A$ such that $w_j\to u$ in $L^1(A,\R^N)$ as $j\to \infty$ and
\begin{equation*}
    \limsup_{j\to \infty}E^{g}(w_j,A)\leq \liminf_{j\to \infty}E^{g}(u_j,A)=E^{\hat g}(u,A),
\end{equation*}
therefore the claim follows by the arbitrariness of $u$.
\end{proof}

Using a classical argument of Ambrosio, in the following lemma we prove a truncation result in the same spirit as in \cite[Lemma~4.4]{ContiFocardiIurlano2024}.
\begin{lemma}\label{l:Lemma di troncamento}
Let $\Functeps_{\eps}$ be the functionals defined in \eqref{e:funzionali approssimanti}. Then, for every $\delta\in (0,1)$, $A\in \corA$, and $(u,v)\in L^1_{\textup{loc}}(\R^n,\R^{N+1})$ with  $u\in W^{1,1}(A,\R^N) \cap L^{\infty}(A,\R^N)$ and ${{v\in}}$ $ W^{1,2}(A,[0,1])$, there exists $u^{\delta} \in L^1_{\textup{loc}}(\R^n,\R^N) \cap SBV(A,\R^N)$ (also depending on $A$) such that for every $\eps>0$ 
\begin{equation}\label{e:bound funz tronct.}
    H^{\delta}_{\eps}(u^{\delta},A) \leq \Functeps_{\eps}(u,v,A)+C\calL^n(\{v \leq \delta\}\cap A),
\end{equation}
where $H^{\delta}_{\eps}:L^1_{\textup{loc}}(\R^n,\R^N)\times \corA\longrightarrow [0,+\infty]$ is the functional given by
\begin{equation*}
    H^{\delta}_{\eps}(w,A):=
    \begin{cases}
        \displaystyle \alpha_{\delta}\int_A f\textstyle{(\frac{x}{\eps},\nabla w)}\dx + \beta_{\delta}\mathcal{H}^{n-1}(J_{w}\cap A) \; & \; \textup{if $w\in SBV(A,\R^N)$} \\
        +\infty \; & \; \textup{otherwise},
    \end{cases}
\end{equation*}
with $\trunca_{\delta}, \truncb_{\delta}>0$ such that
\begin{equation*}
    \lim_{\delta \to 1} \trunca_{\delta}= 1 \quad \textup{and} \quad \lim_{\delta\to 1}\truncb_{\delta}=0. 
\end{equation*}
Moreover, if $(u_{\eps},v_{\eps})\to (u,1)$ in $L^1(A,\R^{N+1})$ as $\eps \to 0$, then the corresponding $(u^\delta_\eps)$ satisfies 
\begin{equation}\label{e:convergenza troncate}
  u^{\delta}_{\eps}\to u \quad \textup{in $L^1(A,\R^N)$} \quad \textit{as $\eps\to 0$}.
\end{equation}
\end{lemma}

\begin{proof}
Let $\delta\in (0,1)$, $\eps>0$, $A\in \corA$, and $(u,v)\in W^{1,1}(A,\R^N)\times W^{1,2}(A,[0,1])$ be given. We have 
\begin{equation}\label{e:prima stima lemma tronc.}
    \Functeps_{\eps}(u,v,A) \geq \int_{\{v\geq \delta^2\}}\alpha_{\delta}f\textstyle{(\frac{x}{\eps},\nabla u)}\dx +\displaystyle \int_A (1-v)|\nabla v| \dx
\end{equation}
where $\alpha_\delta:=\displaystyle\min_{t\in [\delta^2,1]} t^2=\delta^4$. Set 
\[
\Phi(t):=\int_0^t (1-s)\,ds= t-\frac{t^{2}}{2} \quad  \text{and} \quad \Phi_{v}:=\Phi \circ v\in W^{1,2}(A). 
\]
By the Coarea Formula we can infer that
\begin{equation*}
   \int_A (1-v)|\nabla v| \dx=  \int_{A}|\nabla \Phi_{v}| \dx \geq \int_{\Phi(\delta^2)}^{\Phi(\delta)}\mathcal{H}^{n-1}(A\cap \partial ^* ( \{\Phi_{v}>t\}))\dd t,
\end{equation*}
therefore, there exists $t^{\delta}\in (\Phi(\delta^2),\Phi(\delta))$ such that
\begin{equation}\label{e:seconda stima lemma tronc.}
      \int_A (1-v)|\nabla v| \dx \geq (\Phi(\delta)-\Phi(\delta^2))\mathcal{H}^{n-1}(A\cap \partial ^* ( \{\Phi_{v}>t^{\delta}\})).
\end{equation}
Set $u^{\delta}:=u\chi_{\{v>\Phi^{-1}(t^{\delta})\}}$; we notice that $u^{\delta} \in L^1_{\textup{loc}}(\R^n,\R^N)\cap SBV(A,\R^N)$ since $\{v>\Phi^{-1}(t^{\delta})\}$ is a set of finite perimeter in $A$ and $u\in L^{\infty}(A,\R^N)$. 
Since by definition $J_{u^{\delta}}\cap A \subseteq \partial ^* ( \{\Phi_{v}>t^{\delta}\}) \cap A$, \eqref{e:seconda stima lemma tronc.} becomes 
\begin{equation}\label{e:seconda stima lemma tronc ref}
\int_A (1-v)|\nabla v| \dx \geq \beta_{\delta} \mathcal H^{n-1}(J_{u^{\delta}}\cap A),
\end{equation}
where $\beta_{\delta}:=\Phi(\delta)-\Phi(\delta^2)$.

Moreover, by the strict monotonicity of $\Phi$ on $[0,1]$, we get 
\begin{equation*}
  \int_{\{v\geq \delta^2\}}f\textstyle{(\frac{x}{\eps},\nabla u)}\dx \geq \displaystyle \int_{\{v> \Phi^{-1}(t^{\delta})\}}f\textstyle{(\frac{x}{\eps},\nabla u^{\delta})}\dx,
\end{equation*}
so that thanks to \ref{e:crescita lineare}, we obtain 
\begin{equation}\label{e:terza stima lemma tronc.}
    \alpha_{\delta}\int_{\{v\geq \delta^2\}}f\textstyle{(\frac{x}{\eps},\nabla u)}\dx+ C\calL^n(\{v \leq \delta\}\cap A) \geq \displaystyle \alpha_{\delta} \int_{A}f\textstyle{(\frac{x}{\eps},\nabla u^{\delta})}\dx.
\end{equation}    
Eventually, \eqref{e:bound funz tronct.} follows by gathering \eqref{e:prima stima lemma tronc.}, \eqref{e:seconda stima lemma tronc ref}, and \eqref{e:terza stima lemma tronc.}. 

Now let $(u_{\eps},v_{\eps})\to (u,1)$ in $L^1(A,\R^{N+1})$ as $\eps \to 0$ and consider  
\[
u_\eps^{\delta}:=u_\eps\chi_{\{v_\eps>\Phi^{-1}(t^{\delta})\}}.
\]
We observe that
\begin{equation}\label{e:ultima equaz lemma tronc.}
    \|u_{\eps}-u^{\delta}_{\eps}\|_{L^1(A)}\leq \|u_{\eps}\|_{L^1(\{v_{\eps}\leq \delta\}\cap A)}.
\end{equation}
Therefore \eqref{e:convergenza troncate} follows by \eqref{e:ultima equaz lemma tronc.} in view of the equi-integrability of $(u_{\eps})$ and the convergence in measure of $(v_{\eps})$ to $1$.
\end{proof}

We are now ready to identify the Radon-Nikodym derivative of $\Fsucc$ with respect to the Lebesgue measure with $f\homm$.
\begin{proof}[Proof of Proposition~\ref{p:Homogenised volume integrand}]
Fix $A\in \corA$ and $u\in L^1_{\textup{loc}}(\R^n,\R^N)$ with $u\in BV(A,\R^N)\cap L^{\infty}(A,\R^N)$. We divide the proof into two steps.

\textit{Step 1:} We claim that
\begin{equation*}
     \frac{\dd \Fsucc(u,1,\cdot)}{\dd \calL^n}(x)\leq f\homm(\nabla u(x)) \quad \textup{for $\calL^n$-a.e. $x\in A$}.
\end{equation*}
{In particular, we claim the previous equality to be true for every point
$x\in A$ for which the conditions $\rho^{-n}|Du|(Q_{\rho}(x))\to|\nabla u(x)|$ and
$\rho^{-n}|D^su|(Q_{\rho}(x))\to0$ as $\rho\to0$, and the Calder\'on-Zygmund Theorem holds for $x$ (cf. \cite[Theorem~3.83]{AFP}).}
For every $A\in\corA$ and $q\in \Q\cap (0,1)$ set 
$\Fsucc_q(u,A):=\Fsucc(u,1,A)+q|Du|(A)$.
Thanks to \cite[Lemma~3.5]{BouchitteFonsecaMascarenhas} we obtain that 
for $\calL^n$-a.e. $x\in A$
\begin{equation}\label{e:BFM 3.20}
    \frac{\dd \Fsucc(u,1,\cdot)}{\dd \calL^n}(x)+q |\nabla u(x)|=\lim_{\rho \to 0}\frac{m_{\Fsucc_q}(\ell_{\nabla u(x)},Q_{\rho}(x))}{\rho^n}
\end{equation}
where 
where $m_{\Fsucc_q}$ is as in \eqref{e:equaz minfunz}.
Let $x\in A$ be that \eqref{e:BFM 3.20} holds, and set $\xi:=\nabla u(x)$. 
In view of \eqref{e:fhom-ass}, for every $\rho>0$ we have
\begin{equation}\label{e:omogenizzata prop volum term}
    f\homm(\xi)=
    \lim_{r\to +\infty}\frac{\minprobv^f(\ell_{\xi},Q_{r}(\textstyle{\frac{r}{\rho}}x)}{r^n}.
\end{equation}
Fix $\eta\in (0,1)$. By \eqref{e:def minprobv}, for every $\rho, r >0$ there exists $w^{\rho}_r\in W^{1,1}(Q_{r}(\textstyle{\frac{r}{\rho}}x),\R^N)$ with $w^{\rho}_r=\ell_{\xi}$ on $\partial Q_{r}(\textstyle{\frac{r}{\rho}}x)$, such that
\begin{align}\label{e:def_mb}
    \int_{Q_{r}(\frac{r}{\rho}x)}f(y,\nabla w^{\rho}_r)\dd y& \leq \minprobv^f(\ell_{\xi},Q_{r}(\textstyle{\frac{r}{\rho}}x)) +\eta r^n\,. 
\end{align}
Thus, for every $\rho>0$, \eqref{e:omogenizzata prop volum term} and \eqref{e:def_mb} yield
\begin{equation*}
 \limsup_{r\to +\infty}\frac{1}{r^n}\int_{Q_{r}(\frac{r}{\rho}x)}f(y,\nabla w^{\rho}_r)\dy \leq f\homm(\xi)+\eta.  
\end{equation*}
Now, let $(\eps_j)_{j\in\N}$ be as in \eqref{e:Gamma-conv sottosucc} and set $r=\frac{\rho}{\eps_j}$. Define $u^{\rho}_{\eps_j}:\R^n\to \R^N$ as 
\begin{equation*}
    u^{\rho}_{\eps_j}(y):=
    \begin{cases}
       \eps_j w^{\rho}_r\left(\frac{y}{\eps_j}\right) \; & \; \textup{if $y\in Q_{\rho}(x)$} \\
        \ell_\xi(y) \; & \; \textup{if $y\in \R^n\setminus Q_{\rho}(x)$},
    \end{cases}
\end{equation*}
therefore $u^{\rho}_{\eps_j}\in W^{1,1}_{\textup{loc}}(\R^n,\R^N)$ with $u^{\rho}_{\eps_j}=\ell_{\xi}$ on $\R^n\setminus Q_{\rho}(x)$. Changing variables and again invoking \eqref{e:def_mb}, for every $\rho>0$ we get
\begin{equation}\label{e:stima gradienti lowerbound volume q}
    \limsup_{j \to +\infty} \frac{1}{\rho^n}\int_{Q_{\rho(1+\eta)}(x)}f\textstyle(\frac{y}{\eps_j},\nabla u^{\rho}_{\eps_j})\dy \leq  f\homm(\xi)+\eta+ C(|\xi|+1)((1+\eta)^n-1)
\end{equation}
where we also used \ref{e:crescita lineare}, the fact that $\nabla u^{\rho}_{\eps_j}=\xi$ on $Q_{\rho(1+\eta)}(x)\setminus \overline{Q_{\rho}(x)}$, and \eqref{e:Gamma-conv sottosucc}.

Appealing to \eqref{e:stima gradienti lowerbound volume q},  \ref{e:crescita lineare}, and the Poincaré Inequality, for every $\rho$ we can find  a subsequence of $(\eps_j)_{j\in\N}$ (not relabeled) such that $u^{\rho}_{\eps_j}$ converges in $L^1_{\textup{loc}}(\R^n,\R^N)$ to some $u^{\rho}\in L^1_{\textup{loc}}(\R^n,\R^N)\cap BV(Q_{\rho(1+\eta)}(x),\R^N)$ with $u^{\rho}=\ell_{\xi}$ on $\partial Q_{\rho(1+\eta)}(x)$. Moreover, by \eqref{e:Gamma-conv sottosucc}, \eqref{e:stima gradienti lowerbound volume q}, and \ref{e:crescita lineare}, for every $\rho >0$, we have that
\begin{align*}
    \frac{m_{\Fsucc_{{{q}}}}(\ell_{\xi},Q_{\rho(1+\eta)}(x))}{\rho^n}&\leq \frac{\Fsucc(u^{\rho},1,Q_{\rho(1+\eta)}(x))+q|Du^{\rho}|(Q_{\rho(1+\eta)}(x))}{\rho^n} \\ & \leq \textstyle\liminf_{j\to +\infty}\Big(\frac{\Functeps_{\eps_j}(u^{\rho}_{\eps_j},1,Q_{\rho(1+\eta)}(x))}{\rho^n}+\frac{q}{\rho^n} \int_{Q_{\rho(1+\eta)}(x)} |\nabla u^{\rho}_{\eps_j}|\dy \Big)  \nonumber \\ &  \leq    \big( f\homm(\xi)+\eta+ C(|\xi|+1)((1+\eta)^n-1))\big)(1+qC).  \\
\end{align*}
Eventually, by \eqref{e:BFM 3.20} and taking the limit as $\rho \to 0$ we get
 \begin{align*}
   (1+\eta)^n \Big(\frac{\dd \Fsucc(u,1,\cdot)}{\dd \calL^n}(x)+q{{|\xi|}}\Big)& =\lim_{\rho \to 0}\frac{m_{\Fsucc_{{{q}}}}(\ell_{\xi},Q_{\rho(1+\eta)}(x))}{\rho^n} \\ & \leq \big(f\homm(\xi)+\eta+ C(|\xi|+1)((1+\eta)^n-1)\big)(1+qC),
 \end{align*}
 hence the claim follows by letting $\eta,q \to 0$.
 
 \textit{Step 2:} We claim that
 \begin{equation*}
     \frac{\dd \Fsucc(u,1,\cdot)}{\dd \calL^n}(x)\geq f\homm(\nabla u(x)) \quad \textup{for $\calL^n$-a.e. $x\in A$}.
\end{equation*}
Let $A'\in\corA(A)$, by Theorem~\ref{t:Sottosucc gamma-conv.}, we can find a sequence $(u_j,v_j)_{j\in\N}\in L^1_{\textup{loc}}(\R^n,\R^N)$ such that $(u_j,v_j)\in W^{1,1}(A',\R^N)\times W^{1,2}(A',[0,1])$, $(u_{j},v_j) \to (u,1)$ in $L^1_{\textup{loc}}(\R^n,\R^{N+1})$, $v_j(x)\to 1$ for $\calL^n$-a.e. $x\in A'$ as $j\to +\infty$ and
\begin{equation}\label{e:sup finito energie prop volume part}
     \lim_{j\to +\infty}\Functeps_{\eps_j}(u_j,v_j,A')=\Fsucc(u,1,A'). 
\end{equation}
{Without loss of generality we may assume $u_j\in L^\infty(A';\R^n)$.
Indeed, it suffices to apply Lemma~\ref{l:truncation lemma} for every $j\in\N$ to $u_j$
with $M_j\to\infty$ and note that by construction
$\mathcal{T}_{k_j}(u_{j}) \to u$ in $L^1_{\textup{loc}}(\R^n,\R^N)$,
and by \eqref{e:sup finito energie prop volume part}
\[
\Fsucc(u,1,A')\leq  \lim_{j\to +\infty}\Functeps_{\eps_j}(\mathcal{T}_{k_j}(u_j),v_j,A')\leq
 \lim_{j\to +\infty}\Functeps_{\eps_j}(u_j,v_j,A')=\Fsucc(u,1,A').
\]
.}
Let $\delta\in (0,1)$ be fixed; {thus} by Lemma~\ref{l:Lemma di troncamento} we have
\begin{equation*}
    H^{\delta}_{\eps_j}(u^{\delta}_j,A') \leq \Functeps_{\eps_j}(u_j,v_j,A')+C\calL^n(\{v_j\leq \delta\}\cap A'),
\end{equation*}
where $(u^\delta_j)\subset SBV(A',\R^N)$ with $u^{\delta}_j\to u$ in $L^1(A',\R^N)$. Therefore by \eqref{e:sup finito energie prop volume part} we get
\begin{equation}\label{e:He_bdd}
    \liminf_{j\to +\infty} H^{\delta}_{\eps_j}(u^{\delta}_j,A')\leq \Fsucc(u,1,A'),
\end{equation}
since $(v_j)_{j\in\N}$ converges in measure to $1$ on $A'$. 

We now consider the measures $\mu^{\delta}_{j}$ defined  on $A'$ as follows
\begin{equation*}
    \mu^{\delta}_j:=\alpha_{\delta}f\textstyle{(\frac{x}{\eps_j},\nabla {u}^{\delta}_{j})} \mathcal{L}^n \res A' + \beta_{\delta} \mathcal{H}^{n-1} \res (J_{{u}^{\delta}_{j}}\cap A')\,.
\end{equation*}
Note that by \eqref{e:He_bdd}, 
there is a subsequence (not relabeled) and a finite Radon measure $\mu^{\delta}$ on $A'$ such that $\mu^{\delta}_j\stackrel{*}{\weakto} \mu^{\delta}$ 
as $j\to +\infty$.

Now let $x_0\in A'$ be a 
point of approximate differentiability of $u$, and additionally assume that 
\begin{equation}\label{e:uno_der} 
\lim_{\rho\to 0}\frac{\mu^{\delta}(Q_{\rho}(x_0))}{\rho^n}=\frac{\dd \mu^{\delta}}{\dd \calL^n}(x_0)\,.
\end{equation}
Such conditions determine a subset of full measure in $A'$. Then, consider the rescaled function $u^{\rho}:Q_1\to \R^N$ given by
\begin{equation*}
    u^{\rho}(y):=\frac{u(x_0+\rho y)-u(x_0)}{\rho},
\end{equation*}
thanks to \cite[Remark~3.72]{AFP} we have $u^{\rho}\to \ell_{\xi}$ in $L^1(Q_1,\R^N)$, where $\xi:=\nabla u(x_0)$. 

By the weak$^*$-convergence of $\mu^{\delta}_j$ towards $\mu^\delta$ we have 
\begin{align}  \nonumber
     \frac{\dd \mu^{\delta}}{\dd \calL^n}(x_0)&=\lim_{\rho \to 0}\frac{\mu^{\delta}(Q_{\rho}(x_0))}{\rho^n}=\lim_{\rho \to 0, \; \rho\in I(x_0)} \lim_{j\to +\infty} \frac{\mu^{\delta}_j(Q_{\rho}(x_0))}{\rho^n} \\\label{e:prima eq lower bound volume new} &=\lim_{\rho \to 0, \; \rho\in I(x_0)} \lim_{j\to +\infty} \rho^{-n}\Big(\alpha_{\delta}\int_{Q_{\rho}(x_0)}f\textstyle{(\frac{x}{\eps_j},\nabla {u}^{\delta}_{j})} \dx + \beta_{\delta} \mathcal{H}^{n-1}(J_{{u}^{\delta}_{j}}\cap Q_{\rho}(x_0))\Big) 
     \end{align}
where $I(x_0):=\{\rho \in (0,\frac{2}{\sqrt{n}}\textup{dist}(x_0,\partial A')) \colon \mu^{\delta}(\partial Q_{\rho}(x_0))=0\}$. 

For every $\rho$ and $j$, define the rescalings $u^{\rho}_j\in SBV(A',\R^N)$ by
\begin{equation*}
    u^{\rho}_j(y):=\frac{u^{\delta}_j(x_0+\rho y)-u(x_0)}{\rho}\,,
\end{equation*}
then $u^{\rho}_j\to u^{\rho}$ in $L^1(Q_1,\R^N)$ as $j\to +\infty$. Furthermore, thanks to \eqref{e:prima eq lower bound volume new} we get 
\begin{equation}\label{e:blow-up in x0 volume}
   \frac{\dd \mu^{\delta}}{\dd \calL^n}(x_0)=\lim_{\rho \to 0, \; \rho\in I(x_0)} \lim_{j\to +\infty}  \Big(\alpha_{\delta}\int_{Q_1}f\textstyle{(\frac{x_0+\rho y}{\eps_j},\nabla {u}^{\rho}_{j})} \dy + \displaystyle\frac{\beta_{\delta}}\rho 
   \mathcal{H}^{n-1}(J_{{u}^{\rho}_{j}}\cap Q_1)\Big). 
\end{equation}
Fix $M\in \N$, for every $\rho$ and $j$, we apply Lemma~\ref{l:truncation lemma} with $v\equiv 1$ so that there is $k_{\rho,j}\in \{M+1,\dots, 2M\}$ such that 
$\hat u^{\rho}_j:=\mathcal{T}_{k_{\rho,j}}(u^{\rho}_j)\in SBV(Q_1,\R^N)$, 
\begin{align}\label{e:media 2 volume}
    \int_{Q_1}\textstyle{f(\frac{x_0+\rho_i y}{\eps_i},\nabla \hat u^{\rho}_j)} \dd y \leq \displaystyle\Big(1+\frac{C^2}{M}\Big)\int_{Q_1}f\textstyle{(\frac{x_0+\rho_i y}{\eps_i},\nabla u^{\rho}_j)} \dd y + C\mathcal{L}^n(\{|u^{\rho}_j|\geq a_M\})\,.
\end{align}

Up to subsequences (not relabeled) we can assume that $k_{\rho,j}\in \{M+1,\dots,2M\}$ actually depends only on $\rho$. If we choose $a_M>\sup_{y\in Q_1}\ell_{\xi}(y)$ we get that 
\[
\lim_{\rho\to 0}\lim_{j\to +\infty}\hat u^{\rho}_j=\ell_{\xi} \quad \text{in $L^1(Q_1,\R^N)$}
\]
and
\begin{equation}\label{e:convergenza in misura}
    \lim_{\rho \to 0, \; \rho\in I(x_0)} \limsup_{j\to +\infty} \mathcal{L}^n(\{|u^{\rho}_j|\geq a_M\})=0,
\end{equation}
since we also have that $\displaystyle\lim_{\rho\to 0}\lim_{j\to +\infty}u^{\rho}_j=\ell_{\xi}$ in $L^1(Q_1,\R^N)$. 
In particular, for $M$ is large enough, by combining \eqref{e:blow-up in x0 volume}, \eqref{e:media 2 volume}, and \eqref{e:convergenza in misura} we can infer
\begin{equation}\label{e:blow-up x0 volume 2}
    \Big(1+\frac{C^2}{M}\Big) \frac{\dd \mu^{\delta}}{\dd \calL^n}(x_0)\geq \limsup_{\rho \to 0, \; \rho\in I(x_0)} \limsup_{j\to +\infty} \alpha_{\delta}\int_{Q_1}f\textstyle{(\frac{x_0+\rho y}{\eps_j},\nabla \hat{u}^{\rho}_{j})} \dy,
\end{equation}
and
\begin{equation}\label{e:misura salti sta andando a 0 volume}
\lim_{\rho \to 0, \rho\in I(x_0)}\limsup_{j\to +\infty}\mathcal{H}^{n-1}(J_{\hat u^{\rho}_j}\cap Q_1)=0,
\end{equation}
since $\mathcal{T}_{k_{\rho,j}}\in C^1(\R^N,\R^N)$, $\frac{\dd \mu^{\delta}}{\dd \calL^n}(x_0)$ is finite, and $\beta_{\delta}>0$. Set
\begin{equation*}
    \tau_{\rho,j}
:= \|\hat u^{\rho}_j-\ell_{\xi}\|_{L^1(Q_1,\R^N)}+\frac{\rho}{j};
    \end{equation*}
then, $\displaystyle\lim_{\rho \to 0}\lim_{j\to +\infty} \tau_{\rho,j}=0$. Thus, for every $\rho>0$ small and every $j$ large (depending on $\rho$) we have $\tau_{\rho,j}\in (0,1)$. 
Therefore, thanks to the Coarea formula and to the properties of the traces of $BV$ functions on rectifiable sets (see \cite[Theorem 3.77]{AFP}), there exists $\hat r_{\rho,j}\in (1-\tau_{\rho,j}^{1/2},1)$ such that
\begin{equation}\label{e:dato al bordo oculato}
   \int_{\partial Q_{\hat r_{\rho,j}}} \textstyle{|(\hat u^{\rho}_j)^--\ell_{\xi}| \dd \calH^{n-1} \leq \displaystyle \tau_{\rho,j}^{-1/2}\|\hat u^{\rho}_j-\ell_{\xi}\|_{L^1(Q_1,\R^N)}\leq \tau_{\rho,j}^{1/2}},
\end{equation}
where $(\hat u^{\rho}_j)^-$ is the inner trace of $\hat u^{\rho}_j$ on $\partial Q_{\hat r_{\rho,j}}$.
Therefore, defining the functions $w^{\rho}_{j}\in SBV(Q_1,\R^N)$ as  
\begin{equation*}
w^{\rho}_{j}(y):=
\begin{cases}
\hat u^{\rho}_j(y) \; & \textup{if $y\in Q_{\hat r_{\rho,j}}$} \\
 \ell_{\xi}(y) \; & \textup{if $y\in Q_1\setminus  Q_{\hat r_{\rho,j}}$},  
\end{cases}
\end{equation*}
thanks to \ref{e:crescita lineare} we have that 
\begin{equation*}
    \alpha_{\delta}\int_{Q_1}f\textstyle{(\frac{x_0+\rho y}{\eps_j},\nabla \hat{u}^{\rho}_{j})} \dy + \alpha_{\delta}(C|\xi|+1)\calL^n(Q_1\setminus Q_{\hat r_{\rho,j}}) \geq  \displaystyle\alpha_{\delta}\int_{Q_1}f\textstyle(\frac{x_0+\rho y}{\eps_j},\nabla w^{\rho}_{j}) \dy
\end{equation*}
and, since $\displaystyle\lim_{\rho\to 0}\lim_{j\to +\infty} \hat r_{\rho,j}=1$, from \eqref{e:blow-up x0 volume 2} we obtain
\begin{equation}\label{e:stima wrhoj}
    \Big(1+\frac{C^2}{M}\Big) \frac{\dd \mu^{\delta}}{\dd \calL^n}(x_0)\geq \limsup_{\rho \to 0, \; \rho\in I(x_0)} \limsup_{j\to +\infty} \alpha_{\delta}\int_{Q_1}f\textstyle{(\frac{x_0+\rho y}{\eps_j},\nabla w^{\rho}_{j})} \dy.
\end{equation}
Furthermore, thanks to \eqref{e:dato al bordo oculato}, \eqref{e:prop 2 Tk}, and  to the definition of $\hat u^{\rho}_j$ we can estimate the singular part of $Dw^{\rho}_j$ as follows 
\begin{equation}\label{e:stima gradiente singolare volume part 1}
 |D^sw^{\rho}_j|(Q_1)\leq    \int_{J_{w^{\rho}_j}\cap Q_1}|[w^{\rho}_j]|\dd \calH^{n-1} \leq \tau_{\rho,j}^{1/2}+2a_{2M+1}\mathcal{H}^{n-1}(J_{\hat{u}^{\rho}_{j}}\cap Q_1).
\end{equation}
 Now, for every $\rho>$ and $j\in \N$, consider functional $F_{\rho,j}:L^1(Q_1,\R^N)\longrightarrow [0,\infty]$ given by
 \begin{equation*}
 F_{\rho,j}(w):=
 \begin{cases}
     \displaystyle \int_{Q_1}f\textstyle{(\frac{x_0+\rho y}{\eps_j},\nabla w)} \dy \; & \textup{if $w\in W^{1,1}(Q_1,\R^N)$} \\
     \infty \; & \textup{otherwise}.
 \end{cases}
 \end{equation*}
 In view of Lemma~\ref{l:lemma del rilassato} (iv), for every $w\in SBV(Q_1,\R^N)$ we have
 \begin{equation}\label{e:teo rilassamento BFM in volume part 1}
   \Big|sc^-(L^1)F_{\rho,j}(w)-\int_{Q_1}f_{\rho,j}(y,\nabla w) \dy\Big| \leq C\int_{J_w\cap Q_1}|[w]| \dd \calH^{n-1}, \end{equation}
where $f_{\rho,j}:=\hat{g}$, with  $g(x,\xi):=f\left(\frac{x_0+\rho x}{\eps_j},\xi\right)$ for every $(x,\xi)\in \R^n\times\R^{N\times n}$ 
(cf. \eqref{e:limsup del teo bfm}). 
By {the definition of relaxed functional and correction of the boundary datum via} \cite[Lemma 2.6]{BouchitteFonsecaMascarenhas}, for every $\rho$ and $j$ we can find $\hat w^{\rho}_j\in W^{1,1}(Q_1,\R^N)$ with $\hat w^{\rho}_j=\ell_{\xi}$ on $\partial Q_1$ and such that
\begin{equation}\label{e:approx al rilassato volume part 1}
    \Big|sc^-(L^1)F_{\rho,j}(w^{\rho}_j)-  \int_{Q_1}f_{\rho,j}(y,\nabla \hat w^{\rho}_j) \dy \Big| \leq \frac{\rho}{j}.
\end{equation}
In particular, from \eqref{e:stima gradiente singolare volume part 1}, \eqref{e:teo rilassamento BFM in volume part 1} and \eqref{e:approx al rilassato volume part 1} 
and the equality $f_{\rho,j}(y,\xi)=\hat{f}\textstyle{(\frac{x_0+\rho y}{\eps_j},\xi)}$
which follows from formula \eqref{e:limsup del teo bfm}, we infer
\begin{align}\label{e:stima frho j}
 \int_{Q_1}f&\textstyle{(\frac{x_0+\rho y}{\eps_j},\nabla w^{\rho}_{j}) \dy} \geq \displaystyle
    \int_{Q_1}f_{\rho,j}(y,\nabla w^{\rho}_{j}) \dy
  \geq  sc^-(L^1)F_{\rho,j}(w^{\rho}_j)-C\int_{J_{w^{\rho}_j}\cap Q_1}|[w^{\rho}_j]| \dd \calH^{n-1} \nonumber\\ 
  & \geq \int_{Q_1}f_{\rho,j}(y,\nabla \hat w^{\rho}_j) \dy -\frac{\rho}{j}-C(\tau^{1/2}_{\rho,j}+2a_{2M+1}\calH^{n-1}(J_{\hat u^{\rho}_j}\cap Q_1))
  \nonumber\\ 
  &= \int_{Q_1}\hat f\textstyle{(\frac{x_0+\rho y}{\eps_j},\nabla \hat w^{\rho}_{j})} \dy-\frac{\rho}{j}-C(\tau^{1/2}_{\rho,j}+2a_{2M+1}\calH^{n-1}(J_{\hat u^{\rho}_j}\cap Q_1))\,.
\end{align}
Setting 
\[
r_{\rho,j}=\frac{\rho}{\eps_j} \quad \text{and} \quad\overline{w}^{\rho}_j(x):=r_{\rho,j}\hat w^{\rho}_j\textstyle{(\frac{x}{r_{\rho,j}}-\frac{x_0}{\rho})}
\]
we have $\overline{w}^{\rho}_j\in W^{1,1}(Q_{r_{\rho,j}}\textstyle{(\frac{r_{\rho,j}}{\rho}x_0)},\R^N)$ with $\overline{w}^{\rho}_j=\ell_{\xi}-\frac{1}{\eps_j}x_0$ on $\partial Q_{r_{\rho,j}}\textstyle{(\frac{r_{\rho,j}}{\rho}x_0)}$ and 
\begin{equation*}
  \int_{Q_1}\hat   f\textstyle{(\frac{x_0+\rho y}{\eps_j},\nabla \hat w^{\rho}_{j})} \dy=\displaystyle\frac1{r_{\rho,j}^n}\int_{Q_{r_{\rho,j}}(\frac{r_{\rho,j}}{\rho}x_0)}\hat f(x,\nabla \overline{w}^{\rho}_j)\dx. 
\end{equation*}
In particular, Lemma~\ref{l:lemma del rilassato} (vi) gives
\begin{align*}
\int_{Q_{r_{\rho,j}}(\frac{r_{\rho,j}}{\rho}x_0)}\hat f(x,\nabla \overline{w}^{\rho}_j)\dx \geq 
\minprobv^{\hat f}\textstyle(\ell_{\xi},Q_{r_{\rho,j}}(\frac{r_{\rho,j}}{\rho}x_0))
=\minprobv^{f}\textstyle(\ell_{\xi},Q_{r_{\rho,j}}(\frac{r_{\rho,j}}{\rho}x_0))\,.
\end{align*}
Therefore, \eqref{e:stima wrhoj}, \eqref{e:stima frho j} and \eqref{e:fhom-ass} yield
\begin{align*}
     \Big(1+\frac{C^2}{M}\Big) \frac{\dd \mu^{\delta}}{\dd \calL^n}(x_0) & \geq \limsup_{\rho \to 0, \; \rho\in I(x_0)} \limsup_{j\to +\infty} \alpha_{\delta}\frac{\minprobv^{f}\textstyle(\ell_{\xi},Q_{r_{\rho,j}}(\frac{r_{\rho,j}}{\rho}x_0))}{r_{\rho,j}^n}\\
     &=\alpha_{\delta}f\homm(\xi)=\alpha_{\delta}f\homm(\nabla u(x_0))\,
\end{align*}
and thus 
\begin{equation*}
    \frac{\dd \mu^{\delta}}{\dd \calL^n}(x_0)\geq\alpha_{\delta}f\homm(\nabla u(x_0))
\end{equation*}
by letting $M\to\infty$. Hence, recalling \eqref{e:He_bdd}, we deduce that 
\begin{equation*}
   \Fsucc(u,1,A')\geq  \liminf_{j\to \infty}H^{\delta}_{\eps_j}(u^{\delta}_j,A')= \liminf_{j\to \infty}\mu^{\delta}_j(A')\geq \mu^{\delta}(A')
   \geq\alpha_{\delta}\int_{A'}f\homm(\nabla u)\dx\,.
\end{equation*}
Eventually, the claim follows by letting $\delta \to 0$ and by the arbitrariness of $A' \in \corA(A)$. 
\end{proof}

\subsection{Identification of the surface term}\label{s:surface}

In this subsection we show that the Radon Nikodym derivative of $\Fsucc$ with respect to $\calH^{n-1}\res J_u$ equals to $g\homm$ for every $u\in BV$. 
\begin{proposition}[Homogenised surface integrand]\label{p:homogenised surface integrand}
Let $f\in \mathcal F(C,\alpha)$ satisfy \eqref{e:ghom-ass}. 
Let $\Fsucc$ be as in \eqref{e:Gamma-conv sottosucc}. 
Then,
for every $A\in \corA$ and every $u\in L^1_{\textup{loc}}(\R^n,\R^N)$, with $u\in BV(A,\R^N)\cap L^{\infty}(A,\R^N)$ there holds
\begin{equation*}
    \frac{\dd \Fsucc(u,1,\cdot)}{\dd \calH^{n-1} \res J_u}(x)=g\homm([u](x),\nu_u(x)) \quad \textup{for $\calH^{n-1}$-a.e. $x\in  J_u\cap A$},
\end{equation*}
where $g\homm$ is as in \eqref{e:ghom-ass}.
\end{proposition}

To prove Proposition \ref{p:homogenised surface integrand} we need a preliminary lemma which is an extension to our setting of some results contained in \cite{BouchitteFonsecaMascarenhas}. 
\begin{lemma}\label{l:Lemma bfm per espressione blow-up parte di superficie}
Let $U\in \corA$ be fixed and let $G:BV(U,\R^N)\times \corA(U)\longrightarrow [0,\infty)$ be such that 
\begin{enumerate}[label=(\arabic*)]
    \item for every $u\in BV(U,\R^N)$ the set function $G(u,\cdot)$ is the restriction to $\corA(U)$ of a finite Radon measure on $U$;
    \item for every $A\in \corA(U)$ the functional $G(\cdot,A)$ is $L^1(A,\R^n)$-lower semicontinuous;
    \item there exists $K\in (0,\infty)$ such that 
    \begin{equation*}
        G(u,A)\leq K(\calL^n(A)+|Du|(A))
    \end{equation*}
    for every $u\in BV(U,\R^N)$ and every $A\in \corA(U)$.
    \item For every $M\in (0,\infty)$ there exists $K_M\in (0,\infty)$ such that 
    \begin{equation*}
        K_M|Du|(A) \leq G(u,A)
    \end{equation*}
    for every $u\in BV(U,\R^N)$ with $\|u\|_{L^{\infty}(U)}\leq M$ and every $A\in \corA(U)$.
\end{enumerate}
Then, if $w\in BV(U,\R^N)$ is such that $2\|w\|_{L^{\infty}(U)}\leq M$ we have that for $\calH^{n-1}$-a.e. $x\in J_w$
\begin{equation}\label{e:BFM jump}
    \frac{\dd G(w,\cdot)}{\dd \mathcal{H}^{n-1}\res J_w}(x)=\lim_{\rho \to 0}\frac{m_G^M(w,Q_{\rho}^{\nu_w(x)}(x))}{\rho^{n-1}}=\lim_{\rho \to 0}\frac{m_G^M(u_{x,[w](x),\nu_w(x)},Q_{\rho}^{\nu_w(x)}(x))}{\rho^{n-1}}
\end{equation}
where 
\begin{align}\label{e:mMG}
    m^{M}_G(w,A):= \inf \{G(v,A):\, v\in BV(A,\R^N),\,v=w \textup{ on $\partial A$},\,\|v\|_{L^{\infty}(A,\R^N)}\leq M\}.
\end{align}
\end{lemma}
\begin{proof}
The proof follows by combining a number of arguments from \cite[Section~3]{BouchitteFonsecaMascarenhas} which we briefly summarize. 
Appealing to \cite[Lemma~3.5 and formula (3.17) in Theorem~3.7]{BouchitteFonsecaMascarenhas} the equality in \eqref{e:BFM jump} can be established for functionals $G$ satisfying assumptions (1)-(3) above, 
and the stronger growth condition
\begin{equation}\label{e:G stronger growth condition}
        C|Du|(A) \leq G(u,A),
\end{equation}
for every $u\in BV(A,\R^N)$.

In their turn, \cite[Lemma~3.5 and formula (3.17) in Theorem~3.7]{BouchitteFonsecaMascarenhas} are a consequence of \cite[Lemmata~3.1 and 3.3]{BouchitteFonsecaMascarenhas}.
Namely, \cite[Lemmata~3.1]{BouchitteFonsecaMascarenhas} establishes the Lipschitz continuity of $m_G$ as in \eqref{e:equaz minfunz}, with respect to the traces
and is stated under the sole positivity of $G$. It is easy to check that 
an analogous result holds true for $m_G^M$ as in \eqref{e:mMG}.

Moreover, \eqref{e:G stronger growth condition} is used in \cite[Lemma~3.3]{BouchitteFonsecaMascarenhas} 
to prove the equality $G(u,A)=\sup_{\delta>0}m_{G,\delta}(u,A)$, where 
\begin{align*}
m_{G,\delta}(u,A):=\inf\Big\{\sum_{i\in\N}m_G(u,Q_{r_i}^{\nu_i}(x_i)):&\,Q_{r_i}^{\nu_i}(x_i)\subset A,
Q_{r_i}^{\nu_i}(x_i)\cap Q_{r_j}^{\nu_j}(x_j)=\emptyset, i\neq j,\\
&\mathrm{diam}Q_{r_i}^{\nu_i}(x_i)<\delta,\,\mu(A\setminus\cup_{i\in\N}Q_{r_i}^{\nu_i}(x_i))=0\Big\}\,,
\end{align*}
with $\mu:=\calL^n+|D^su|$.

Then to conclude we notice that the same identity holds true for $m_G^M$ under the assumptions (1)-(4). In fact, one inequality is trivial, while the other
can be obtained by exhibiting a competitor with the same $L^\infty$ bound.
\end{proof}

We are now ready to show Proposition~\ref{p:homogenised surface integrand}. 
\begin{proof}[Proof of Proposition~\ref{p:homogenised surface integrand}]
Let us fix $A\in \corA$ and $u\in L^1_{\textup{loc}}(\R^n,\R^N)$ with $u\in BV(A,\R^N)\cap L^{\infty}(A,\R^N)$. We divide the proof into two steps.

\textit{Step 1:} We claim that 
\begin{equation*}
      \frac{\dd \Fsucc(u,1,\cdot)}{\dd \calH^{n-1} \res J_u}(x)\leq g\homm([u](x),\nu_u(x)) \quad \textup{for $\calH^{n-1}$-a.e. $x\in J_u\cap A$}.
\end{equation*}
Thanks to Theorem~\ref{t:Sottosucc gamma-conv.} and Lemma~\ref{l:Lemma bfm per espressione blow-up parte di superficie}, we obtain for $\calH^{n-1}$-a.e. $x\in J_u\cap A$ that
\begin{equation}\label{e:BFM 3.25}
    \frac{\dd \Fsucc(u,1,\cdot)}{\dd \calH^{n-1}\res J_u}(x)=\lim_{\rho \to 0}\frac{m^L_{\Fsucc(\cdot,1,\cdot)}(u,Q^{\nu_u(x)}_{\rho}(x))}{\rho^{n-1}}=\lim_{\rho \to 0}\frac{m^L_{\Fsucc(\cdot,1,\cdot)}(u_{x,[u](x),\nu_u(x)},Q^{\nu_u(x)}_{\rho}(x))}{\rho^{n-1}},
\end{equation}
for every $L>0$ with $ 2\|u\|_{L^{\infty}(A)}\leq L$. 

Fix $x\in A\cap J_u$ such that \eqref{e:BFM 3.25} holds and set $\zeta=[u](x)$ and $\nu_u(x)=\nu$. By combining \eqref{e:ghom-ass} and Corollary~\ref{c:comparison minimum problems}, for every $\rho>0$ we get
\begin{equation}\label{e:ghom sup term}
    g\homm(\zeta,\nu)=\lim_{r\to +\infty}\frac{\minprobs^{f^{\infty}}(\overline{u}_{\frac{r}{\rho}x,\zeta,\nu},Q^{\nu}_r(\textstyle{\frac{r}{\rho}}x))}{r^{n-1}}.
\end{equation}
Let $\eta\in (0,1)$ be fixed. By \eqref{e:ghom sup term}, for every $\rho$ and every $r$ large enough, we have that 
\begin{equation*}
   \minprobs^{f^{\infty}}(\overline{u}_{\frac{r}{\rho}x,\zeta,\nu},Q^{\nu}_r(\textstyle{\frac{r}{\rho}}x)) 
   \leq g\homm(\zeta,\nu)r^{n-1}+\eta r^{n-1}.
\end{equation*}
Therefore there exist $w^{\rho}_{r}\in W^{1,1}(Q^{\nu}_{r}(\textstyle{\frac{r}{\rho}}x),\R^N)$ and $v^{\rho}_{r}\in W^{1,2}(Q^{\nu}_{r}(\textstyle{\frac{r}{\rho}}x),[0,1])$ with $(w^{\rho}_{r},v^{\rho}_{r})=(\overline{u}_{\frac{r}{\rho}x,\zeta,\nu},1)$ on $\partial Q^{\nu}_{r}(\textstyle{\frac{r}{\rho}}x)$, such that
\begin{align}\label{e:stima 1 sup term}
  &\Functmins^{f^{\infty}}\textstyle(w^{\rho}_{r}, v^{\rho}_{r},Q^{\nu}_r(\textstyle{\frac{r}{\rho}}x)) =\displaystyle
  \int_{Q^{\nu}_r(\frac{r}{\rho}x)}\left((v^{\rho}_{r})^2f^{\infty}(y,\nabla w^{\rho}_{r}) + (1- v^{\rho}_{r})^2+|\nabla v^{\rho}_{r}|^2 \right)\dy \leq %
  \nonumber\\ & 
  \leq  
  \minprobs^{f^{\infty}}\textstyle(\overline{u}_{\frac{r}{\rho}x,\zeta,\nu},Q^{\nu}_r(\textstyle{\frac{r}{\rho}}x)) +\eta r^{n-1}\leq g\homm(\zeta,\nu)r^{n-1}+2\eta r^{n-1}. 
\end{align}
Next apply Lemma~\ref{l:truncation lemma} with $\gamma=0$ to infer that
$\mathcal{T}_{k^{\rho}_{r}}(w^{\rho}_{r})=:\hat w^{\rho}_{r}\in W^{1,1}(Q^{\nu}_r(\textstyle{\frac{r}{\rho}}x),\R^N)$ satisfies 
\begin{align}\label{e:troncamento linfinito in limsup superficie}
    \Functmins^{f^{\infty}}\textstyle(\hat w^{\rho}_{r},v^{\rho}_{r},Q^{\nu}_r(\textstyle{\frac{r}{\rho}}x)) &
    \leq \displaystyle \Big(1+\frac{C^2}{M}\Big)\Functmins^{f^{\infty}}\textstyle(w^{\rho}_{r}, v^{\rho}_{r},Q^{\nu}_r(\textstyle{\frac{r}{\rho}}x)) \nonumber 
    \\ & \leq \Big(1+\frac{C^2}{M}\Big)(g\homm(\zeta,\nu)r^{n-1}+2\eta r^{n-1})\,.
\end{align}
Moreover, if $M\in \N$ is such that $(2|\zeta|+1)+ 2\|u\|_{L^{\infty}(A)}\leq a_{2M}$ then 
$\hat w^{\rho}_{r}=\overline{u}_{\frac{r}{\rho}x,\zeta,\nu}$ on $\partial Q^{\nu}_r(\textstyle{\frac{r}{\rho}}x)$ and 
$\|\hat w^{\rho}_{r}\|_{L^\infty(Q^{\nu}_r(\frac{r}{\rho}x))}\leq a_{2M+1}$.

In particular, by 
\eqref{e:troncamento linfinito in limsup superficie} we have that
\begin{equation}\label{e:stima L^1 gradiente limsup parte di superficie}
    \frac{1}{r^{n-1}}\int_{Q^{\nu}_r(\frac{r}{\rho}x)} (v^{\rho}_{r})^2|\nabla \hat w^{\rho}_{r}| \dy \leq C\Big(1+\frac{C^2}{M}\Big)( g\homm(\zeta,\nu)+2\eta).
\end{equation}
By Lemma~\ref{l:lemma funz recess integrale} and \eqref{e:stima L^1 gradiente limsup parte di superficie}, given $\rho>0$, for every $r$ large enough we have 
\begin{align}\label{e:stima 2 sup term}
   & \frac{1}{r^{n-1}}\int_{Q^{\nu}_r(\frac{r}{\rho}x)}\textstyle| (v^{\rho}_{r})^2f^{\infty}(y,\nabla \hat w^{\rho}_{r})-  (v^{\rho}_{r})^2\frac{\rho}{r}f(y,\frac{r}{\rho}\nabla \hat w^{\rho}_{r})|\dy \nonumber
   \\ & \leq K\rho+ \frac{K\rho^{\alpha}}{r^{(n-1)(1-\alpha)}} \Big(\int_{Q^{\nu}_{r}(\frac{r}{\rho}x)}(v^{\rho}_{r})^2|\nabla \hat w^{\rho}_{r}| \dy \Big)^{1-\alpha}\nonumber 
   \\ & \leq K\rho + K\rho^{\alpha}C^{1-\alpha}\Big(1+\frac{C^2}{M}\Big)^{1-\alpha}(g\homm(\zeta,\nu)+2\eta)^{1-\alpha}\,.
\end{align}
Therefore, recollecting \eqref{e:ghom sup term}, \eqref{e:stima 1 sup term} and \eqref{e:stima 2 sup term}, for every $\rho>0$ we get the following
\begin{align*}
     \limsup_{r\to +\infty} \frac{1}{r^{n-1}}&\int_{Q^{\nu}_{r}(\frac{r}{\rho}x)}\textstyle(  (v^{\rho}_{r})^2\frac{\rho}{r}f(y,\frac{r}{\rho}\nabla \hat w^{\rho}_{r})+(1-v^{\rho}_{r})^2+|\nabla v^{\rho}_{r}|^2 ) \dy  \\ 
&   \leq \Big(1+\frac{C^2}{M}\Big)(g\homm(\zeta,\nu) + 2\eta) +\tilde K(M,\rho,\eta),
\end{align*}
where $\tilde K(M,\rho,\eta):=K\rho + K\rho^{\alpha}C^{1-\alpha}\big(1+\frac{C^2}{M}\big)^{1-\alpha}(g\homm(\zeta,\nu)+2\eta)^{1-\alpha}$.

Given $\eps>0$ and $\rho>0$, we define $(\hat u^{\rho}_{\eps},\hat v^{\rho}_{\eps}):\R^n\to \R^{N+1}$ as follows
\begin{equation*}
\hat u^{\rho}_{\eps}(y):=
    \begin{cases}
    \hat w^{\rho}_{r}(\frac{ry}{\rho}) \; & \; \textup{if $y\in Q^{\nu}_{\rho}(x)$}\\
    \overline{u}^{\eps}_{x,\zeta,\nu} \; & \; \textup{if $y\in \R^n \setminus Q^{\nu}_{\rho}(x)$}\\
    \end{cases}
\quad \quad     \hat v^{\rho}_{\eps}(y):=
    \begin{cases}
       \hat v^{\rho}_{r}(\frac{ry}{\rho}) \; & \; \textup{if $y\in Q^{\nu}_{\rho}(x)$}\\
    1 \; & \; \textup{if $y\in \R^n \setminus Q^{\nu}_{\rho}(x)$}\\
    \end{cases}
\end{equation*}
with $r=\frac{\rho}{\eps}$. Thereby $\hat u^{\rho}_{\eps}\in W^{1,1}_{\textup{loc}}(\R^n,\R^{N})$  with $\|\hat u^{\rho}_{\eps}\|_{L^{\infty}(\R^n)}\leq a_{2M+1}$ and $\hat v^{\rho}_{\eps}\in W^{1,2}_{\textup{loc}}(\R^n,[0,1])$. Changing variables it is immediate to get
\begin{align*}
   &  \limsup_{j \to +\infty}\Big( \frac{1}{\rho^{n-1}}\int_{Q^{\nu}_{\rho}(x)} ((\hat v^{\rho}_{\eps_j})^2 \textstyle f(\frac{y}{\eps_j},\nabla \hat u^{\rho}_{\eps_j}) + \frac{(1-\hat v^{\rho}_{\eps_j})^2}{\eps_j}+\eps_j |\nabla \hat v^{\rho}_{\eps_j}|^2) \dy \Big)
   \\ & \leq \Big(1+\frac{C^2}{M}\Big)(g\homm(\zeta,\nu) + 2\eta) +\tilde K(M,\rho,\eta),
    \end{align*}
where $(\eps_j)_{j\in\N}$ is the sequence in Theorem~\ref{t:Sottosucc gamma-conv.} along which the $\Gamma$-convergence of $(\Functeps_{\eps})_{\eps>0}$ holds. 
Moreover, we observe that 
\begin{align*}
  &  \frac{1}{\rho^{n-1}}\int_{Q^{\nu}_{\rho(1+\eta)}(x) \setminus Q^{\nu}_{\rho}(x)}\textstyle( (\hat v^{\rho}_{\eps_j})^2 f(\frac{y}{\eps_j},\nabla \hat u^{\rho}_{\eps_j}) + \frac{(1-\hat v^{\rho}_{\eps_j})^2}{\eps_j}+\eps_j |\nabla \hat v^{\rho}_{\eps_j}|^2) \dy 
  \\ & \leq \frac{C}{\rho^{n-1}}\int_{Q^{\nu}_{\rho(1+\eta)}(x) \setminus Q^{\nu}_{\rho}(x)} (1+|\nabla \overline{u}^{\eps_j}_{x,\zeta,\nu}|)\dy 
  \\ & \leq C\rho\left((1+\eta)^n-1\right)+ \frac{C}{\rho^{n-1}}\int_{(Q^{\nu}_{\rho(1+\eta)}(x) \setminus Q^{\nu}_{\rho}(x))\cap \{|(y-x)\cdot \nu|\leq \eps_j/2\}} |\nabla \overline{u}^{\eps_j}_{x,\zeta,\nu}|\dy
  \\ &  \leq C\rho\left((1+\eta)^n-1\right)+ C|\zeta|\|\overline{\textup{u}}'\|_{L^\infty(\R)} \left((1+\eta)^{n-1}-1\right)\leq C((1+\eta)^n-1)(\rho+|\zeta|\|\overline{\textup{u}}'\|_{L^\infty(\R)}).
\end{align*}
 Therefore, for every $\rho$, 
 we have that
\begin{equation*}
   \sup_{j\in \N} \int_{Q^{\nu}_{\rho(1+\eta)}(x)}\textstyle ((\hat v^{\rho}_{\eps_j})^2 |\nabla \hat u^{\rho}_{\eps_j}| + \frac{(1-\hat v^{\rho}_{\eps_j})^2}{\eps_j}+\eps_j |\nabla \hat v^{\rho}_{\eps_j}|^2 \dy )<+\infty.
\end{equation*}
From $\|\hat u^{\rho}_{\eps_j}\|_{L^{\infty}(Q^{\nu}_{\rho}(x))}\leq a_{2M+1}$ and \cite[Lemma~7.1]{AlFoc} there exists a subsequence (not relabeled) 
of $(\eps_j)_{j\in\N}$ and $u^{\rho}\in L^1_{\textup{loc}}(\R^n,\R^N)$ such that $(\hat u^{\rho}_{\eps_j},\hat v^{\rho}_{\eps_j})\to (u^{\rho},1)$ in $L^1_{\textup{loc}}(\R^n,\R^{N+1})$, $u^{\rho}\in BV(Q^{\nu}_{\rho(1+\eta)}(x),\R^N)$, $u^{\rho}(y)=u_{x,\zeta,\nu}(y)$ for $\calL^n$-a.e. $y\in \R^n\setminus Q^{\nu}_{\rho}(x)$. By assumption (ii) {{in}} Theorem~\ref{t:Deterministic Gamma-conv} (cf. formula \eqref{e:Gamma-conv sottosucc}), it follows that for every $\rho$ 
\begin{align*}
   & \frac{m^{a_{2M+1}}_{\Fsucc(\cdot,1,\cdot)}(u_{x,\zeta,\nu},Q^{\nu}_{(1+\eta)\rho}(x))}{\rho^{n-1}}\leq \frac{\Fsucc(u^{\rho},1,Q^{\nu}_{(1+\eta)\rho}(x))}{\rho^{n-1}}\leq \liminf_{j\to +\infty}\frac{\Functeps_{\eps_j}(\hat u^{\rho}_{\eps_j},\hat v^{\rho}_{\eps_j},Q^{\nu}_{(1+\eta)\rho}(x))}{\rho^{n-1}} 
   \\ & \leq  \Big(1+\frac{C^2}{M}\Big)(g\homm(\zeta,\nu) + 2\eta) +\tilde K(M,\rho,\eta)+C((1+\eta)^n-1)(\rho+|\zeta|\|\overline{\textup{u}}'\|_{L^\infty(\R)}).
\end{align*}
In particular, thanks to \eqref{e:BFM 3.25} we have that
\begin{align*}
    & (1+\eta)^{n-1} \frac{\dd \Fsucc(u,1,\cdot)}{\dd \calH^{n-1}\res J_u}(x)=\lim_{\rho \to 0}\frac{{{m^{a_{2M+1}}_{\Fsucc(\cdot,1,\cdot)}}}(u_{x,\zeta,\nu},Q^{\nu}_{(1+\eta)\rho}(x))}{\rho^{n-1}}
    \\ & \leq  \Big(1+\frac{C^2}{M}\Big)(g\homm(\zeta,\nu) + 2\eta) + C|\zeta|\|\overline{\textup{u}}'\|_{L^\infty(\R)} \left((1+\eta)^{n}-1\right)
\end{align*}
and thereby, letting $M\to +\infty$ and $\eta\to 0$, we can conclude.

\textit{Step 2:} We claim that
\begin{equation*}
    \frac{\dd \Fsucc(u,1,\cdot)}{\dd \calH^{n-1}\res J_u}(x)\geq g\homm([u](x),\nu_{u}(x)) \quad \textup{for $\calH^{n-1}$-a.e. $x\in J_u\cap A$.} 
\end{equation*}
By Theorem~\ref{t:Sottosucc gamma-conv.} there exists $(u_j,v_j)\in L^1_{\textup{loc}}(\R^n,\R^N)$ with $u_{j}\in W^{1,1}(A,\R^N)$ and $v_{j}\in W^{1,2}(A,[0,1])$, such that $v_j(x)\to 1$ for $\calL^n$-a.e. $x\in A$ as $j\to +\infty$,
\begin{equation*}
    (u_j,v_j)\to (u,1) \quad \textup{in $L^1_{\textup{loc}}(\R^n,\R^{N+1})$} \quad \textup{and} \quad \lim_{j\to +\infty}\Functeps_{\eps_j}(u_j,v_j,A)=\Fsucc(u,1,A).
\end{equation*}
For ${{\calH^{n-1}}}$-a.e. $x\in J_u\cap A$ (cf. \cite[Theorem 3.77 and Proposition 3.92]{AFP}) we have 
\begin{equation}\label{e:densità n-1 dim punto di superficie}
    \lim_{\rho\to 0}\frac{1}{\rho^{n-1}}|Du|(Q^{\nu_u(x)}_{\rho}(x))=|[u](x)|\neq 0,
\end{equation}
\begin{equation}\label{e:densità n dim punto di superficie}
    \lim_{\rho \to 0}\frac{1}{\rho^n}\int_{Q^{\nu_u(x)}_{\rho}(x)}|u(y)-u_{x,[u](x),\nu_u(x)}(y)| \dy =0,
\end{equation}
\begin{equation}\label{e:radon-nyk parte di superficie}
    \lim_{\rho \to 0}\frac{\Fsucc(u,1,Q^{\nu_u(x)}_{\rho}(x))}{\rho^{n-1}}=\frac{\dd \Fsucc(u,1,\cdot)}{\dd \calH^{n-1} \res J_u}(x)<+\infty.
\end{equation}
Let us fix $x\in J_u\cap A$ such that \eqref{e:densità n-1 dim punto di superficie}-\eqref{e:radon-nyk parte di superficie} are satisfied, 
and set $\zeta:=[u](x)$ and $\nu:=\nu_u(x)$.
Using the lower bound inequality in the $\Gamma$-convergence of $(\Functeps_\eps)_{\eps>0}$ on $Q^{\nu_u(x)}_{\rho}(x)$ and  
$A\setminus Q^{\nu_u(x)}_{\rho}(x)$, the super-additivity of the inferior limit operator implies that 
$\displaystyle{\lim_{j\to +\infty}\Functeps_{\eps_j}(u_j,v_j,Q^{\nu}_{\rho}(x))=\Fsucc(u,1,Q^{\nu}_{\rho}(x))}$ 
for every $\rho \in I(x):=\{\rho \in (0,\frac{2}{\sqrt{n}}\textup{dist}(x,\partial A)) \; : \; \Fsucc(u,1,\partial Q^{\nu}_{\rho}(x))=0\}$.
Hence, we deduce that
\begin{equation}\label{e:radon-nyk. parte di superficie lower bound}
    \frac{\dd \Fsucc(u,1{{,}}\cdot)}{\dd \calH^{n-1}\res J_u}(x)=\lim_{I(x) \ni \rho \to 0}\lim_{j\to +\infty}\frac{1}{\rho^{n-1}}
    \int_{Q_{\rho}^{\nu}(x)}\textstyle(v_j^2f(\frac{y}{\eps_j},\nabla u_{j}) + \frac{(1-v_j)^2}{\eps_j}+\eps_j |\nabla v_{j}|^2) \dy.
\end{equation}
Now, we consider the rescalings $(u^{\rho}_{j},v^{\rho}_{j}), (u^{\rho},v^{\rho}):Q_1^{\nu}\to \R^{N+1}$ given by 
\begin{equation*}
    (u^{\rho}_{j}(y),v^{\rho}_{j}(y)):=(u_j(x+\rho y),v_j(x+\rho y)) \quad \textup{and} \quad (u^{\rho}(y),v^{\rho}(y)):=(u(x+\rho y),1)
\end{equation*}
Then $u^{\rho}_{j}\in W^{1,1}(Q_1^{\nu},\R^N)$, $v^{\rho}_{j}\in W^{1,2}(Q_1^{\nu},[0,1])$, $u^{\rho}\in BV(Q_1^{\nu},\R^N)$, and
$(u^{\rho}_{j},v^{\rho}_{j})\to(u^{\rho},1)$ in $L^1(Q_1^{\nu},\R^{N+1})$,
$u^{\rho}\to u_{\zeta,\nu}$ in $L^1(Q^{\nu},\R^N)$ by \eqref{e:densità n dim punto di superficie},
and $v^{\rho}_j\to 1$ in $L^2(Q^{\nu})$ for every $\rho$ by\eqref{e:radon-nyk. parte di superficie lower bound}.
Changing variables,  formula \eqref{e:radon-nyk. parte di superficie lower bound} rewrites as
\begin{equation*}
     \frac{\dd \Fsucc(u,1\cdot)}{\dd \calH^{n-1}\res J_u}(x)=\lim_{I(x) \ni \rho \to 0}\lim_{j\to +\infty}\int_{Q_1^{\nu}}\textstyle
     (\rho (v^{\rho}_{j})^2f(\frac{x+\rho y}{\eps_j},\frac{1}{\rho}\nabla u^{\rho}_{j}) 
     + \frac\rho{\eps_j}(1-v^{\rho}_{j})^2+\frac{\eps_j}\rho |\nabla v^{\rho}_{j}|^2 )\dy\,,
\end{equation*}
thus  by \ref{e:crescita lineare} we infer that
\begin{equation}\label{e:upperbound norma l1 gradienti prop lowerbound di superficie}
    \frac{\dd \Fsucc(u,1\cdot)}{\dd \calH^{n-1}\res J_u}(x)\geq C\limsup_{I(x) \ni \rho \to 0}\limsup_{j\to +\infty} \int_{Q_1^{\nu}} (v^{\rho}_{j})^2|\nabla u^{\rho}_{j}|\dy.
\end{equation}
Lemma~\ref{l:lemma funz recess integrale} implies that
\begin{equation}\label{e:cambio con funz di recessione in prop lowerbound superficie}
    \rho \int_{Q_1^{\nu}}(v^{\rho}_{j})^2\textstyle f(\frac{x+\rho y}{\eps_j},\frac{1}{\rho}\nabla u^{\rho}_{j}) \dy \displaystyle \geq \int_{Q_1^{\nu}}(v^{\rho}_{j})^2 f^{\infty}\textstyle(\frac{x+\rho y}{\eps_j},\nabla u^{\rho}_{j}) \dy - \displaystyle K\rho-K\rho^{\alpha}\Big(\int_{Q_1^{\nu}} (v^{\rho}_{j})^2|\nabla u^{\rho}_{j}|\dy\Big)^{1-\alpha}
\end{equation}
where $K$ is a constant that depends only on $C$ and $\alpha$. Thanks to \eqref{e:radon-nyk parte di superficie}, 
\eqref{e:upperbound norma l1 gradienti prop lowerbound di superficie} and \eqref{e:cambio con funz di recessione in prop lowerbound superficie} we get
\begin{equation}\label{e:radon-nyk col functeps infinito}
 \frac{\dd \Fsucc(u,1\cdot)}{\dd \calH^{n-1}\res J_u}(x)\geq \limsup_{I(x) \ni \rho \to 0}\limsup_{j\to +\infty} \Functeps^{\infty,x}_{\rho,\eps_j}(u^{\rho}_{j},v^{\rho}_{j},Q_1^{\nu}),
\end{equation}
where 
 \begin{equation*}
      \Functeps_{\rho,\eps}^{\infty,x}(u,v,A):=
      \int_A \textstyle(v^2f^{\infty}(\frac{x+\rho y}{\eps},\nabla u)+\frac{\rho}{\eps}(1-v)^2+\frac{\eps}{\rho}|\nabla v|^2) \dy.
   \end{equation*}
Now, for every $\rho$ and $j$ we consider the sequences $(a^{\rho}_{j})_{j\in\N},\, (b^{\rho}_{j})_{j\in\N}$ and $(s^{\rho}_{j})_{j\in\N}$ given by
\begin{equation}\label{e:scelta scala}
  a^{\rho}_{j}:=\rho +\|u^{\rho}_{j}-w^{\rho}\|^{\frac{1}{2}}_{L^1(Q_1^{\nu})}+\|v^{\rho}_{j}-1\|^{\frac{1}{2}}_{L^2(Q_1^{\nu})}, \quad b^{\rho}_{j}:=\left\lfloor \frac{\rho}{\eps_j} \right\rfloor \quad \textup{and} \quad s^{\rho}_{j}:=\frac{a^{\rho}_{j}}{b^{\rho}_{j}},
\end{equation}
where $(w^{\rho})$ is a sequence in $W^{1,1}(Q^{\nu}_1,\R^N)$ such that $w^{\rho}=u_{\zeta,\nu}$ on $\partial Q_1^{\nu}$ for every $\rho$,
\begin{equation}\label{e:proprieta wrho}
 |Dw^{\rho}|(Q_1^{\nu})\to |Du_{\zeta,\nu}|(Q_1^{\nu}) \quad \textup{and} \quad  w^{\rho}\to u_{\zeta,\nu} \; \; \textup{in $L^1(Q_1^{\nu},\R^N)$ as $\rho \to 0$},
\end{equation}
(see \cite[Lemma~2.5]{BouchitteFonsecaMascarenhas}), where $\lfloor s\rfloor$ denotes the integer part of $s\in\R$.
Fix $\rho$ small enough, such that $\rho+\|u^{\rho}-w^{\rho}\|^{\frac{1}{2}}_{L^1(Q_1^{\nu})}<\frac{1}{4}$ and then fix $j$ large enough such that $0<a^{\rho}_{j}<\frac{1}{2}$ and $2<b^{\rho}_{j}$. For every $i=0,\dots,b^{\rho}_{j}$ we define $Q^{\nu}_{\rho,j,i}$ as 
\begin{equation*}
    Q^{\nu}_{\rho,j,i}:=(1-a^{\rho}_{j}+is^{\rho}_{j})Q_1^{\nu},
\end{equation*}
while for every $i=1,..,b^{\rho}_{j}$ we consider the cut-off function $\phi^{\rho}_{j,i}\in C^{\infty}_{c}(Q^{\nu}_{\rho,j,i})$ such that $0\leq \phi^{\rho}_{j,i}\leq 1$, $\phi^{\rho}_{j,i}\equiv 1$ on $Q^{\nu}_{\rho,j,i-1}$ and $\|\nabla \phi^{\rho}_{j,i}\|_{L^\infty(\R^n)}\leq 2(s^{\rho}_{j})^{-1}$. Set for $i=1,..,b^{\rho}_{j}$
\begin{equation*}
    u^{\rho}_{j,i}:=\phi^{\rho}_{j,i-1}u^{\rho}_{j}+(1-\phi^{\rho}_{j,i-1})w^{\rho} \quad \textup{and} \quad v^{\rho}_{j,i}:=\phi^{\rho}_{j,i}v^{\rho}_{j}+(1-\phi^{\rho}_{j,i}).
\end{equation*} 
Then $u^{\rho}_{j,i}\in W^{1,1}(Q_1^{\nu},\R^N)$, $v^{\rho}_{j,i}\in W^{1,2}(Q_1^{\nu},[0,1])$ with $(u^{\rho}_{j,i},v^{\rho}_{j,i})=(u_{\zeta,\nu},1)$ on $\partial Q_1^{\nu}$. 
Moreover, for every $i=2,\dots,{{b^\rho_j}}$ we have the following
\begin{align*}
 & \Functeps^{\infty,x}_{\rho,\eps_j}(u^{\rho}_{j,i},v^{\rho}_{j,i},Q_1^{\nu}) \leq \Functeps^{\infty,x}_{\rho,\eps_j}(u^{\rho}_{j},v^{\rho}_{j},Q^{\nu}_{\rho,j,i-2}) 
 \\ &+\int_{Q_{\rho,j,i-1}^{\nu}\setminus Q^{\nu}_{\rho,j,i-2}}\textstyle(v^{\rho}_{j})^2f^{\infty}(\frac{x+\rho y}{\eps_j},\nabla u^{\rho}_{j,i}) \dy  + \displaystyle\int_{Q^{\nu}_{\rho,j,i-1}\setminus Q^{\nu}_{\rho,j,i-2}}\textstyle{}( \frac{\rho}{\eps_j}(1-v^{\rho}_{j})^2+\frac{\eps_j}{\rho} |\nabla v^{\rho}_{j}|^2)\dy \\ & 
 + \int_{Q_1^{\nu}\setminus Q^{\nu}_{\rho,j,i-1}}\textstyle(v^{\rho}_{j,i})^2f^{\infty}(\frac{x+\rho y}{\eps_j},\nabla w^{\rho}) \dy 
 +\displaystyle \int_{Q_{\rho,j,i}^{\nu}\setminus Q^{\nu}_{\rho,j,i-1}}\textstyle(\frac{\rho}{\eps_j}(1-v^{\rho}_{j,i})^2+\frac{\eps_j}{\rho} |\nabla v^{\rho}_{j,i}|^2 )\dy.
   \end{align*}
 We estimate separately the terms appearing above. We start with
\begin{equation*}
  \Functeps^{\infty,x}_{\rho,\eps_j}(u^{\rho}_{j},v^{\rho}_{j},Q^{\nu}_{\rho,j,i-2})  
  + \int_{Q^{\nu}_{\rho,j,i-1}\setminus Q^{\nu}_{\rho,j,i-2}}\textstyle{}( \frac{\rho}{\eps_j}(1-v^{\rho}_{j})^2+\frac{\eps_j}{\rho} |\nabla v^{\rho}_{j}|^2)\dy \displaystyle
  \leq \Functeps^{\infty,x}_{\rho,\eps_j}(u^{\rho}_{j},v^{\rho}_{j},Q_{{{1}}}^{\nu}).
\end{equation*}
Moreover, since 
$\nabla u^{\rho}_{j,i}=\phi^{\rho}_{j,i-1}\nabla u^{\rho}_{j} + (1-\phi^{\rho}_{j,i-1})\nabla w^{\rho} + \nabla \phi^{\rho}_{j,i-1} \otimes (u^{\rho}_{j}-w^{\rho})$,
we have that
\begin{align*}
& \int_{Q_{\rho,j,i-1}^{\nu}\setminus Q^{\nu}_{\rho,j,i-2}}\textstyle(v^{\rho}_{j})^2f^{\infty}(\frac{x+\rho y}{\eps_j},\nabla u^{\rho}_{j,i}) \dy \\ 
& \leq C \int_{Q_{\rho,j,i-1}^{\nu}\setminus Q^{\nu}_{\rho,j,i-2}}(v^{\rho}_{j})^2(|\nabla u^{\rho}_{j}|+|\nabla w^{\rho}|+|\nabla \phi^{\rho}_{j,i-1}||u^{\rho}_{j}-w^{\rho}|) \dy \\ 
&  \leq C^2 \int_{Q_{\rho,j,i-1}^{\nu}\setminus Q^{\nu}_{\rho,j,i-2}}\textstyle(v^{\rho}_{j})^2 f^{\infty}(\frac{x+\rho y}{\eps_j},\nabla u^{\rho}_{j}) \dy + \displaystyle C \int_{Q_{\rho,j,i-1}^{\nu}\setminus Q^{\nu}_{\rho,j,i-2}} |\nabla w^{\rho}| \dy \\ 
& + \frac{2C}{s^{\rho}_{j}}\int_{Q_{\rho,j,i-1}^{\nu}\setminus Q^{\nu}_{\rho,j,i-2}} |u^{\rho}_{j}-w^{\rho}| \dy. 
\end{align*}
Analogously, we obtain
\begin{equation*}
  \int_{Q_1^{\nu}\setminus Q^{\nu}_{\rho,j,i-1}}\textstyle(v^{\rho}_{j,i})^2f^{\infty}(\frac{x+\rho y}{\eps_j},\nabla w^{\rho}) \dy \displaystyle\leq C   \int_{Q_1^{\nu}\setminus Q^{\nu}_{\rho,j,i-1}}|\nabla w^{\rho}| \dy. 
\end{equation*}
Since $\nabla v^{\rho}_{j,i}=\phi^{\rho}_{j,i}\nabla v^{\rho}_{j}+(v^{\rho}_{j}-1)\nabla \phi^{\rho}_{j,i}$, 
we have that 
\begin{align*}
 & \int_{Q^{\nu}_{\rho,j,i}\setminus Q^{\nu}_{\rho,j,i-1}} \textstyle{}(\frac{\rho}{\eps_j}(1-v^{\rho}_{j,i})^2+\frac{\eps_j}{\rho} |\nabla v^{\rho}_{j,i}|^2 )\dy \\ 
 & \leq \frac{\rho}{\eps_j}\calL^n(Q^{\nu}_{\rho,j,i}\setminus Q^{\nu}_{\rho,j,i-1})+ 2 \int_{Q^{\nu}_{\rho,j,i}\setminus Q^{\nu}_{\rho,j,i-1}}  \frac{\eps_j}{\rho} |\nabla v^{\rho}_{j}|^2 \dy +  \frac{8\eps_j}{\rho (s^{\rho}_{j})^2}\int_{Q^{\nu}_{\rho,j,i}\setminus Q^{\nu}_{\rho,j,i-1}}|v^{\rho}_{j}-1|^2 \dy\,. 
  \end{align*}
In particular, thanks to the previous calculations and and recalling the definition of $s^{\rho}_{j}$, 
there exists $i^{\rho}_{j}\in \{2,\ldots,b^{\rho}_{j}\}$ such that 
\begin{align*}
 & \Functeps^{\infty,x}_{\rho,\eps_j}(u^{\rho}_{j,i^{\rho}_{j}},v^{\rho}_{j,i^{\rho}_{j}},Q_1^{\nu})\leq \frac{1}{b^{\rho}_{j}-1}\sum_{i=2}^{b^{\rho}_{j}}  \Functeps^{\infty,x}_{\rho,\eps_j}(u^{\rho}_{j,i},v^{\rho}_{j,i},Q_1^{\nu})\\ 
 & \leq \Functeps^{\infty,x}_{\rho,\eps_j}(u^{\rho}_{j},v^{\rho}_{j},Q_1^{\nu})+\frac{{{C^2+}}2}{b^{\rho}_{j}-1} \Functeps^{\infty,x}_{\rho,\eps_j}(u^{\rho}_{j},v^{\rho}_{j},Q_1^{\nu}) + C\int_{Q_1^{\nu}\setminus Q^{\nu}_{0,\rho,j}} |\nabla w^{\rho}| \dy \\
 & + \frac{2C b^{\rho}_{j}}{(b^{\rho}_{j}-1)a^{\rho}_{j}}\int_{Q_1^{\nu}\setminus Q^{\nu}_{\rho,j,0}}|u^{\rho}_{j}-w^{\rho}|\dy 
 +\frac{\rho}{\eps_j (b^{\rho}_{j}-1)}\calL^n(Q_1^{\nu}\setminus Q^{\nu}_{\rho,j,0}) \\ 
 & + \frac{8\eps_j(b^{\rho}_{j})^2}{\rho (a^{\rho}_{j})^2(b^{\rho}_{j}-1)}\int_{Q_1^{\nu}\setminus Q^{\nu}_{\rho,j,0}}|v^{\rho}_{j}-1|^2 \dy .
\end{align*}
  Hence, by the definition of $a^{\rho}_{j}$ and $b^{\rho}_{j}$ in \eqref{e:scelta scala} we deduce that
  \begin{align*}
   & \Functeps^{\infty,x}_{\rho,\eps_j}(u^{\rho}_{j,i^{\rho}_j},v^{\rho}_{j,i^{\rho}_j},Q_{{{1}}}^{\nu}) \leq 
   \Big(1+\frac{{{C^2+}}2}{b^{\rho}_{j}-1}\Big)\Functeps^{\infty,x}_{\rho,\eps_j}(u^{\rho}_{j},v^{\rho}_{j},Q_{{{1}}}^{\nu}) 
   \\ & +C \int_{Q^{\nu}\setminus Q^{\nu}_{\rho,j,0}} |\nabla w^{\rho}| \dy +4C a^{\rho}_{j} +3 \calL^n(Q_{{{1}}}^{\nu}\setminus Q^{\nu}_{\rho,j,0})
   +16(a^{\rho}_{j})^2.
  \end{align*}
Thus, setting $\kappa_{\rho}=1-\|u^{\rho}-w^{\rho}\|_{L^1(Q_{{{1}}}^{\nu})}$, from \eqref{e:radon-nyk col functeps infinito} and \eqref{e:scelta scala} we obtain
  \begin{equation*}
      \limsup_{I(x) \ni \rho \to 0}\limsup_{j\to +\infty}\Functeps^{\infty,x}_{\rho,\eps_j}(u^{\rho}_{j,i^{\rho}_j},v^{\rho}_{j,i^{\rho}_j},Q_{{{1}}}^{\nu}) \leq 
\frac{\dd \Fsucc(u,1\cdot)}{\dd \calH^{n-1}\res J_u}(x)
     + C\limsup_{I(x)\ni \rho \to 0}|Dw^{\rho}|(Q_{{{1}}}^{\nu}\setminus Q^{\nu}_{\kappa_\rho}),
  \end{equation*}
As $w^{\rho}\to u_{\zeta,\nu}$ strictly in $BV$ (cf. \eqref{e:proprieta wrho}), we have that $
|Dw^{\rho}|(Q_{{{1}}}^{\nu}\setminus Q^{\nu}_{\kappa_\rho})\to 0$ as $\rho \to 0$, and thus
  \begin{equation}\label{e:ultima stima lower bound superficie}
      \frac{\dd \Fsucc(u,1\cdot)}{\dd \calH^{n-1}\res J_u}(x) \geq  \limsup_{I(x) \ni \rho \to 0}\limsup_{j\to +\infty}\Functeps^{\infty,x}_{\rho,\eps_j}(u^{\rho}_{j,i^{\rho}_j},v^{\rho}_{j,i^{\rho}_j},Q_{{{1}}}^{\nu}).
  \end{equation}
 By the change of variable and the $1$-homogeneity of $f^{\infty}$ , we have
  \begin{align*}
    \Functeps^{\infty,x}_{\rho,\eps_j}(u^{\rho}_{j,i^{\rho}_j},v^{\rho}_{j,i^{\rho}_j},Q_{{{1}}}^{\nu}) 
    =\frac{1}{r_{\rho,j}^{n-1}}\int_{Q^{\nu}_{r_{\rho,j}}(\frac{r_{\rho,j}}{\rho}x)}\left(
    \overline{v}_{\rho,j}^2 f^{\infty}(y,\nabla \overline{u}_{\rho,j})+(1-\overline{v}_{\rho,j})^2+|\nabla \overline{v}_{\rho,j}|^2 \right)\dy,
  \end{align*}
where $r_{\rho,j}:=\frac{\rho}{\eps_j}$, $\overline{u}^{\rho}_{j}(y):=u^{\rho}_{j,i^{\rho}_j}\textstyle(\frac{y}{r_{\rho,j}}-\frac{x}{\rho})$ and  $\overline{v}^{\rho}_{j}(y):=v^{\rho}_{j,i^{\rho}_j}\textstyle(\frac{y}{r_{\rho,j}}-\frac{x}{\rho})$. In this way $\overline{u}^{\rho}_{j} \in W^{1,1}\textstyle(Q^{\nu}_{r_{\rho,j}}(\frac{r_{\rho,j}}{\rho}x),\R^N)$, $ \overline{v}^{\rho}_{j}\in W^{1,2}(\textstyle Q^{\nu}_{r_{\rho,j}}(\frac{r_{\rho,j}}{\rho}x),[0,1])$ with $(\overline{u}^{\rho}_{j} ,\overline{v}^{\rho}_{j})=(u_{\frac{r_{\rho,j}}{\rho}x,\zeta,\nu},1)$ on $\partial Q^{\nu}_{r_{\rho,j}}(\frac{r_{\rho,j}}{\rho}x)$. In particular, by \eqref{e:ultima stima lower bound superficie}, the definition of $\minprobs^{f^\infty}$ in \eqref{e:def minprobs} and the assumption (b) of Theorem~\ref{t:Deterministic Gamma-conv} we obtain
\begin{equation*}
      \frac{\dd \Fsucc(u,1\cdot)}{\dd \calH^{n-1}\res J_u}(x) \geq  \limsup_{\tilde I(x) \ni \rho \to 0}\limsup_{j\to +\infty} \frac{\minprobs^{f^{\infty}}(u_{\frac{r_{\rho,j}}{\rho}x,\zeta,\nu},Q^{\nu}_{r_{\rho,j}}(\frac{r_{\rho,j}}{\rho}x))}{r_{\rho,j}^{n-1}}=g\homm(\zeta,\nu),
\end{equation*}
deducing the claim.
\end{proof}

\subsection{Identification of the Cantor term}\label{s:cantor}
Eventually, in this subsection we identify the density of the Cantor part of the $\Gamma$-limit $\Fsucc$.
\begin{proposition}[Homogenised Cantor integrand]\label{p:homogenised cantor integrand}
Let $f\in \mathcal F(C,\alpha)$ satisfy \eqref{e:fhom-ass-ii}. 
Let $\Fsucc$ be as in \eqref{e:Gamma-conv sottosucc}. 
Then for every $A\in \corA$ and every $u\in L^1_{\textup{loc}}(\R^n,\R^N)$, with $u\in BV(A,\R^N)\cap L^{\infty}(A,\R^N)$, we have that
\begin{equation*}
    \frac{\dd \Fsucc(u,1,\cdot)}{\dd |D^cu|}(x)=f^{\infty}\homm\Big(\frac{\dd D^cu}{\dd |D^cu|}\Big) \quad \textup{for $|D^cu|$-a.e. $x\in A$},
\end{equation*}
where $f\homm^\infty$ is the recession function of $f\homm$ as in \eqref{e:fhom-ass}. 
\end{proposition}

\begin{proof}
Let us fix $A\in \corA$ and $u\in L^1_{\textup{loc}}(\R^n,\R^N)$ with $u\in BV\cap L^\infty(A,\R^N)$. We divide the proof into two steps.

\textit{Step 1:} We claim that 
\begin{equation*}
    \frac{\dd \Fsucc(u,1,\cdot)}{\dd |D^cu|}(x) \leq f^{\infty}\homm\Big(\frac{\dd D^cu}{\dd |D^cu|}\Big) \quad \textup{for $|D^cu|$-a.e. $x\in A$}.
\end{equation*}
By Alberti's Rank-one Theorem \cite{rankonealb} we know that for $|D^cu|$-a.e. $x\in A$ we have
\begin{equation}\label{e:rankonethm}
   \frac{\dd D^cu}{\dd |D^cu|}(x)=a(x)\otimes \nu(x) 
\end{equation}
where $(a(x),\nu(x))\in \R^N \times \Sf^{n-1}$. By Theorem~\ref{t:Sottosucc gamma-conv.} and by \cite[Lemma~3.9]{BouchitteFonsecaMascarenhas} we have that for $|D^c u|$-a.e. $x\in A$ there exists a doubly indexed positive sequence $(t_{\rho,k})$, with $\rho\in (0,\infty)$ and $k\in \N$, such that for every $k\in \N$ 
\begin{equation}\label{e:blow-up cantor 1 limsup}
    \textup{$t_{\rho,k}\to +\infty$ \quad and \quad $\rho {{t_{\rho,k}}}\to 0$ \quad as $\rho\to 0$,}
\end{equation}
and for every $q\in \Q\cap(0,1)$
\begin{equation}\label{e:blow-up cantor 2 limsup}
    \frac{\dd \Fsucc(u,1,\cdot)}{\dd |D^cu|}(x)+q a(x)\otimes \nu(x)=\lim_{k\to +\infty}\limsup_{\rho \to 0} \frac{m_{\Fsucc_q}(\ell_{t_{\rho,k}a(x)\otimes \nu(x)},Q^{\nu(x),k}_r(\textstyle{\frac{r}{\rho}}x)}{k^{n-1}\rho^nt_{\rho,k}}\,,
\end{equation}
where for every $A\in\corA$ and $q\in \Q\cap (0,1)$ let $\Fsucc_q(u,1,A):=\Fsucc(u,1,A)+q|Du|(A)$, and $Q^{\nu,k}_r(z)$ is the parallelepiped defined in \ref{e:cubo ruotato} of the notation list.
Let $x\in A$ be such that \eqref{e:rankonethm}-\eqref{e:blow-up cantor 2 limsup} hold true, and set $a:=a(x)$ and $\nu:=\nu(x)$. 
Thanks to Proposition~\ref{p:finfhom prop}, for every $\rho>0$ and every $k\in \N$ we have 
\begin{equation}\label{e:finfhom prop cantor}
    f\homm^{\infty}(a\otimes \nu)=\lim_{r\to \infty} \frac{\minprobv^{f^{\infty}}(\ell_{a\otimes \nu},Q^{\nu,k}_r(\textstyle{\frac{r}{\rho}}x))}{k^{n-1}r^n}.
\end{equation}
Let us fix $\eta\in (0,1)$. By the very definition of $\minprobv^{f^{\infty}}$, for every $k\in \N$, $\rho\in (0,1)$ and $r\in (0,\infty)$ 
there exists a function $\hat u^{\rho,k}_r\in W^{1,1}(Q^{\nu,k}_{r}(\textstyle{\frac{r}{\rho}}x),\R^N)$ with $\hat u^{\rho,k}_r=\ell_{a\otimes \nu}$ on $\partial Q^{\nu,k}_r(\textstyle{\frac{r}{\rho}}x)$ such that
\begin{equation}\label{e:stima 1 cantor part}
    \Functminv^{f^{\infty}}\left(\hat u^{\rho,k}_r,Q^{\nu,k}_r(\textstyle{\frac{r}{\rho}}x)\right)\leq \minprobv^{f^{\infty}}\left(\ell_{a\otimes \nu},Q^{\nu,k}_r(\textstyle{\frac{r}{\rho}}x)\right)+ \eta k^{n-1}r^n \leq C|a|k^{n-1}r^n + \eta k^{n-1}r^n
\end{equation}
and in particular
\begin{equation*}
   \int_{Q^{\nu,k}_r(\frac{r}{\rho}x)}|\nabla \hat u^{\rho,k}_r | \dy \leq C^2(|a|+1)k^{n-1}r^n.
\end{equation*}
Therefore, by Lemma~\ref{l:lemma funz recess integrale}, we have that
\begin{align}\label{e:stima 2 prop parte di cantor}
   & \frac{1}{k^{n-1}r^n}\int_{Q^{\nu,k}_r(\frac{r}{\rho}x)}\Big|f^{\infty}(y,\nabla \hat u^{\rho,k}_r)-\frac{1}{t_{\rho,k}}f(y,t_{\rho,k}\nabla \hat u^{\rho,k}_r)\Big|\dy  \\ & \leq \frac{K}{t_{\rho,k}}k^{n-1}r^n+ \frac{K}{t_{\rho,k}^{\alpha}}(k^{n-1}r^n)^{\alpha}\Big(\int_{Q^{\nu,k}_r(\frac{r}{\rho}x)} |\nabla u_r|\dy\Big)^{1-\alpha} \leq \frac{\Hat K}{t^{\alpha}_{\rho,k}}k^{n-1}r^n\nonumber 
\end{align}
where $\Hat K$ depends only on $C$, $\alpha$ and $a$. Collecting \eqref{e:finfhom prop cantor}-\eqref{e:stima 2 prop parte di cantor}, we infer that
\begin{equation*}
\limsup_{r\to +\infty}\frac{1}{k^{n-1}r^nt_{\rho,k}}\int_{Q^{\nu,k}_r(\frac{r}{\rho}x)}f(y,t_{\rho,k}\nabla \hat u^{\rho,k}_r) \dy \leq f\homm^{\infty}(a\otimes \nu) + \eta+ \frac{\Hat K}{t_{\rho,k}^{\alpha}}.
\end{equation*}
For $k\in \N$, $\rho\in (0,1)$ and $\eps\in (0,\infty)$ we define the function $u^{\rho,k}_{\eps}:\R^n\to \R^N$ given by 
\begin{equation*}
  u^{\rho,k}_{\eps}(y):=
  \begin{cases}
  \eps t_{\rho,k}\hat u^{\rho,k}_r(\frac{y}{\eps}) \; & \; \textup{if $y\in Q^{\nu,k}_{\rho}(x)$} \\
  t_{\rho,k}\ell_{a \otimes \nu} \; & \; \textup{if $y\in \R^n\setminus Q^{\nu,k}_{\rho}(x)$},
  \end{cases}
\end{equation*}
where $r:=\frac{\rho}{\eps}$. Thus $u^{\rho,k}_{\eps}\in W^{1,1}_{\textup{loc}}(\R^n,\R^N)$ with $u^{\rho,k}_{\eps}=t_{\rho,k} \ell_{a\otimes \nu}$ on $\partial Q^{\nu,k}_{\rho}(x)$ and changing variable we obtain
\begin{align}\label{e:stima 3 prop cantor part}
 \limsup_{\eps \to 0} \frac{1}{k^{n-1}\rho^n t_{\rho,k}}&\int_{Q^{\nu,k}_{\rho(1+\eta)}(x)}\textstyle f\left(\frac{x}{\eps},\nabla u^{\rho,k}_{\eps}\right) \dd y\nonumber \\ 
 & \leq f\homm^{\infty}(a\otimes \nu)+\eta+ \frac{\Hat K}{t^{\alpha}_{\rho,k}}+C\left(\frac{1}{t_{\rho,k}}+|a|\right)((1+\eta)^n-1),
\end{align}
since $u^{\rho,k}_{\eps}$ coincides with $t_{\rho,k} \ell_{a\otimes \nu}$ on $Q^{\nu,k}_{\rho(1+\eta)}(x)\setminus Q^{\nu,k}_{\rho}(x)$. By \eqref{e:stima 3 prop cantor part} and Poincaré inequality, we can extract a subsequence (not relabelled) of $(\eps_j)_{j\in\N}$, for every $\rho \in (0,1)$ and $k\in \N$, such that $u^{\rho,k}_{\eps_j}\to u^{\rho,k}$ in $L^1_{\textup{loc}}(\R^n,\R^N)$, where $u^{\rho,k}\in BV(Q^{\nu,k}_{\rho(1+\eta)}(x),\R^N)$ with 
$u^{\rho,k}=t_{\rho,k}\ell_{a \otimes \nu}$ on $Q^{\nu,k}_{\rho(1+\eta)}(x)\setminus Q^{\nu,k}_{\rho}(x)$. As a consequence of the $\Gamma$-convergence stated in Theorem~\ref{t:Sottosucc gamma-conv.}, of the superadditivity of the inferior limit operator, of \ref{e:crescita lineare} and of estimate \eqref{e:stima 3 prop cantor part} we obtain
\begin{align*}
   \frac{m_{\Fsucc_q}(\ell_{t_{\rho,k}a\otimes \nu},Q^{\nu,k}_{\rho(1+\eta)}(x))}{k^{n-1}\rho^nt_{\rho,k}} 
   & \leq \frac{\Fsucc_q(u^{\rho,k},1,Q^{\nu,k}_{\rho(1+\eta)}(x))}{k^{n-1}\rho^nt_{\rho,k}} \\ 
   & \leq  \liminf_{j \to +\infty}\Big(\frac{\Functeps_{\eps_j}(u^{\rho,k}_{\eps_j},1,Q^{\nu,k}_{\rho(1+\eta)}(x))}{k^{n-1}\rho^n t_{\rho,k}} +  \frac{q}{k^{n-1}\rho^n t_{\rho,k}}\int_{Q^{\nu,k}_{\rho(1+\eta)}(x)}|\nabla u^{\rho,k}_{\eps_j}| \dd y\Big) \\ 
   &\leq  (1+qC)\Big(f\homm^{\infty}(a\otimes \nu)+\eta+ \frac{\Hat K}{t^{\alpha}_{\rho,k}}+C\left(\frac{1}{t_{\rho,k}}+|a|\right)((1+\eta)^n-1)\Big).
\end{align*}
We can pass to the limit in the last inequality for $\rho\to 0$ and then for $k\to +\infty$, using \eqref{e:blow-up cantor 1 limsup} 
and \eqref{e:blow-up cantor 2 limsup} we arrive to
\begin{equation*}
    (1+\eta)^n\Big(\frac{\dd \Fsucc(u,1,\cdot)}{\dd |D^cu|}(x) +q a \otimes \nu \Big)
    \leq (1+qC)(f\homm^{\infty}(a\otimes \nu)+\eta+C|a|((1+\eta)^n-1))\,.
\end{equation*}
The claim follows by letting $\eta$ and $q\to 0$.
\smallskip

\textit{Step 2:} We claim that 
\begin{equation}\label{e:lowerbound parte di cantoor}
    \frac{\dd \Fsucc(u,1,\cdot)}{\dd |D^cu|}(x) \geq f^{\infty}\homm\Big(\frac{\dd D^cu}{\dd |D^cu|}\Big) \quad \textup{for $|D^cu|$-a.e. $x\in A$}\,.
\end{equation}
Let $A'\in\corA(A)$, then by Theorem~\ref{t:Sottosucc gamma-conv.} we can find a sequence $(u_j,v_j)_{j\in\N}\in L^1_{\textup{loc}}(\R^n,\R^{N+1})$ such that $(u_j,v_j)\in W^{1,1}(A',\R^N)\times W^{1,2}(A',[0,1])$, $(u_j,v_j)\to (u,1)$ in $L^1_{\textup{loc}}(\R^n,\R^{N+1})$, $v_j\to 1$ for $\calL^n$-a.e. $x\in A'$ as $j\to \infty$ and 
\begin{equation*}
    \lim_{j\to \infty}\Functeps_{\eps_j}(u_j,v_j,A')=\Fsucc(u,1,A').
\end{equation*}
{ Arguing as in Step~2 of the proof of Proposition~\ref{p:Homogenised volume integrand} we may assume $u_j\in L^\infty(A';\R^N)$.}
Fix $\delta\in (0,1)$; by Lemma~\ref{l:Lemma di troncamento} we have
\begin{equation*}
    H^{\delta}_{\eps_j}(u^{\delta}_j,A') \leq \Functeps_{\eps_j}(u_j,v_j,A')+C\calL^n(\{v_j\leq \delta\}\cap A'),
\end{equation*}
where $u^\delta_j\in SBV(A',\R^N)$ with $u^{\delta}_j\to u$ in $L^1(A',\R^N)$ as $j\to \infty$, and therefore
\begin{equation}\label{e:Hdeltaepsj vv hatF}
    \liminf_{j\to \infty} H^{\delta}_{\eps_j}(u^{\delta}_j,A')\leq \Fsucc(u,1,A').
\end{equation}
Define on $A'$ the measures $\mu^{\delta}_{j}$ given by
\begin{equation*}
    \mu^{\delta}_j:=\textstyle\alpha_{\delta}f\left(\frac{x}{\eps_j},\nabla {u}^{\delta}_{j}\right) \mathcal{L}^n \res A' 
    + \beta_{\delta} \mathcal{H}^{n-1} \res (J_{{u}^{\delta}_{j}}\cap A')\,.
\end{equation*}
By definition of $H^{\delta}_{\varepsilon}$ and the compactness of Radon measures, there exists subsequence (not relabeled) of $(\eps_j)_{j\in\N}$ 
and a finite Radon measure $\mu^{\delta}$ on $A'$ such that $\mu^{\delta}_j\to \mu^{\delta}$ weakly* in the sense of measures on $A'$ as $j\to \infty$. 
For $|D^c u|$-a.e. $x\in A'$ (cf. \cite[Proposition 3.92 and Theorem 3.94]{AFP}) there exists $a(x){{\in \Sf^{N-1}}}$ and $\nu(x)\in \Sf^{n-1}$ such that for every $k \in \N$ we have
\begin{equation}\label{e:rankone lowerbound}
    \lim_{\rho \to 0}\frac{Du(Q^{\nu(x),k}_{\rho}(x))}{|Du|(Q^{\nu(x),k}_{\rho}(x))}= \frac{\dd D^cu}{\dd |D^cu|}(x)=a(x)\otimes \nu(x)
\end{equation}
\begin{equation}\label{e:blow-up cantor lowebound 1}
    \lim_{\rho \to 0}\frac{|Du|(Q^{\nu(x),k}_{\rho}(x))}{\rho^n}=\infty
\end{equation}
\begin{equation}\label{e:blow-up cantor lowebound 2}
    \lim_{\rho\to 0}\frac{|Du|(Q^{\nu(x),k}_{\rho}(x))}{\rho^{n-1}}=0
\end{equation}
\begin{equation}\label{e:ultima blow-up cantor}
    \frac{\dd \mu^{\delta}}{\dd |Du|}(x)=\frac{\dd \mu^{\delta}}{\dd |D^cu|}(x).
\end{equation}
Fix now $x_0\in A'$ such that \eqref{e:rankone lowerbound}-\eqref{e:ultima blow-up cantor} hold true and set $a:=a(x_0)$ and $\nu:=\nu(x_0)$. For $k\in \N$ and $\rho$ set
\begin{equation*}
    t_{\rho,k}:=\frac{|Du|(Q^{\nu,k}_{\rho}(x_0))}{k^{n-1}\rho^n},
\end{equation*}
therefore
\begin{equation}\label{e:i trhok cantor}
    t_{\rho,k}\to \infty \quad \textup{and} \quad \rho t_{\rho,k}\to 0 \; \; \; \textup{as $\rho\to 0$}.
\end{equation}
and 
\begin{equation*}
    \frac{\dd \mu^{\delta}}{\dd |D^cu|}(x_0)
    =\lim_{\rho \to 0} \frac{\mu^{\delta}(Q^{\nu,k}_{\rho}(x_0))}{ t_{\rho,k}k^{n-1}\rho^n}
\end{equation*}
By the weak$^*$-convergence of $(\mu^{\delta}_{\eps_j})_{j\in\N}$ to $\mu^{\delta}$ we infer that
\begin{equation}\label{e:doppia convergenza radon Cantor}
    \frac{\dd \mu^{\delta}}{\dd |D^cu|}(x_0)=\lim_{\rho \to 0, \rho\in I(x_0)} \lim_{j\to \infty} \frac{\mu^{\delta}_j(Q^{\nu,k}_{\rho}(x_0))}{ t_{\rho,k}k^{n-1}\rho^n},
\end{equation}
where $I(x_0):=\{\rho \in (0,1):\,\mu^{\delta}(\partial Q^{\nu,k}_{\rho}(x_0))=0 \text{ for every $k\in \N$ s.t. } Q^{\nu,k}_{\rho}(x_0)\subset \subset A'\}$. 
Note that $I(x_0)$ has full measure in $(0,1)$. 
Fix $k\in\N$ and consider the rescaled functions $u^{\rho,k}_j,\,u^{\rho,k}:Q^{\nu,k}\to \R^N$ 
\begin{align*}
    u^{\rho,k}_j(y)&=\frac{1}{k^{n-1}\rho t_{\rho,k}}\Big(u^{\delta}_j(x_0+\rho y)-\frac{1}{k^{n-1}\rho^n}\int_{Q^{\nu,k}_{\rho}(x_0)}u^{\delta}_j(z) \dd z \Big) \\ u^{\rho,k}(y)&=\frac{1}{k^{n-1}\rho t_{\rho,k}}\Big(u(x_0+\rho y)-\frac{1}{k^{n-1}\rho^n}\int_{Q^{\nu,k}_{\rho}(x_0)}u(z) \dd z \Big).
\end{align*}
From now on we work at $k\in\N$ fixed and this will tend to $\infty$ only at the very end of the proof. Therefore, for those parameters infinitesimal
as $j\to\infty$ and $\rho\to0$ the possible dependence on $k$ will not be highlighted.
For every $\rho$ small enough, depending on $k$, $u^{\rho,k}_j\in SBV(Q_1^{\nu,k},\R^N)$, $u^{\rho,k}\in BV(Q_1^{\nu,k},\R^N)$,
$u^{\rho,k}_j\to u^{\rho,k}$ in $L^1(Q_1^{\nu,k},\R^N)$ as $j\to \infty$, and the function $u^{\rho,k}$ satisfies the following
\begin{equation*}
    \int_{Q^{\nu,k}}u^{\rho,k}(y) \dy=0 \quad \textup{with} \quad Du^{\rho,k}(Q_1^{\nu,k})=\frac{Du(Q^{\nu,k}_{\rho}(x_0))}{|Du|(Q^{\nu,k}_{\rho}(x_0))}\to a\otimes \nu \quad \textup{as $\rho \to 0$}.
\end{equation*}
Moreover, from \eqref{e:doppia convergenza radon Cantor} we obtain 
\begin{align}\label{e:seconda doppia convergenza Radon Cantor}
  &  \frac{\dd \mu^{\delta}}{\dd |D^c u|}(x_0)
  =\lim_{\rho\to 0,\rho\in I(x_0)}\lim_{j \to \infty}\frac{\mu^{\delta}_{j}(Q_{\rho}^{\nu,k}(x_0))}{k^{n-1}\rho^nt_{\rho,k}} \nonumber\\ 
  & =  \lim_{\rho\to 0,\rho\in I(x_0)}\lim_{j\to \infty}\Big(\frac{\alpha_{\delta}}{k^{n-1}\rho^n t_{\rho,k}}\int_{Q^{\nu,k}_{\rho}(x_0)}\textstyle f(\frac{x}{\eps_j},\nabla u^{\delta}_{j}) \dx + \displaystyle \frac{\beta_{\delta}}{k^{n-1}\rho^n t_{\rho,k}}\mathcal{H}^{n-1}(J_{u^{\delta}_{j}}\cap Q^{\nu,k}_{\rho}(x_0)) \Big)= \nonumber\\ 
  & = \lim_{\rho\to 0,\rho\in I(x_0)}\lim_{j\to \infty}\Big(\frac{\alpha_{\delta}}{ k^{n-1}t_{\rho,k}}\int_{Q_1^{\nu,k}}\textstyle f(\frac{x_0+\rho y}{\eps_j},k^{n-1}t_{\rho,k}\nabla u^{\rho,k}_{j}) \dx + \displaystyle\frac{\beta_{\delta}}{k^{n-1}\rho t_{\rho,k}}\mathcal{H}^{n-1}(J_{u^{\rho,k}_{j}}\cap Q_1^{\nu,k}) \Big). 
\end{align}
{{By \cite[Lemma~4.5]{ContiFocardiIurlano2024} 
(see also \cite[Lemma 5.1]{Larsen})}} there exists a subsequence (not relabeled), depending on $k$, such that
$u^{\rho,k}\to u^{k}$ in $L^1(Q_1^{\nu,k},\R^N)$ as $\rho \to 0$, where $u^k\in BV(Q_1^{\nu,k},\R^N)$, 
$u^k(y)=\chi_k(y\cdot \nu)a$ for every $y\in Q_1^{\nu,k}$,
$\chi_k:[-1/2,1/2]\to \R$ is a nondecreasing function such that $D\chi_k((-1/2,1/2))=\chi_k(1/2)- \chi_k(-1/2)=\frac{1}{k^{n-1}}$, $-\frac{1}{k^{n-1}}\leq \chi_k(-1/2)\leq 0 \leq \chi_k(1/2)\leq \frac{1}{k^{n-1}}$, and
\begin{equation*}
    \int_{-1/2}^{1/2}\chi_k(t) \dt=0.
\end{equation*}
Furthermore, being $\chi_k$ continuous in $-1/2$ and $1/2$, thanks to the trace's properties of $BV$ functions, we have that the inner trace of $u^k$ satisfies $u^k=\ell^k$ on $\partial ^\perp Q_1^{\nu,k}$ (cf. \ref{e:cubo ruotato} of the notation list) where 
\begin{equation*}
  \ell^k(y):=\frac{1}{k^{n-1}}\ell_{a\otimes \nu}(y)+\Big(\chi_k(1/2)-\frac{1}{2k^{n-1}}\Big) a.
\end{equation*}
To obtain a uniform $L^{\infty}$-bound on the scaled sequence, we let $M\in\N$ and use Lemma~\ref{l:truncation lemma} with $v\equiv 1$ to get
for every $k\in \N$, every $\rho$ small enough and every $j$ (up to a subsequence), $m_{\rho}\in \{M+1,\dots, 2M\}$ such that
\begin{equation}\label{e:bound L^1 gradienti cantor}
    \Big(1+\frac{C^2}{M}\Big) \frac{\dd \mu^{\delta}}{\dd |D^cu|}(x_0)\geq \limsup_{\rho \to 0, \; \rho\in I(x_0)} \limsup_{j\to \infty} \frac{\alpha_{\delta}}{k^{n-1}t_{\rho,k}}\int_{Q_1^{\nu,k}}\textstyle f\left(\frac{x_0+\rho y}{\eps_j},k^{n-1}t_{\rho,k}\nabla \hat{u}^{\rho,k}_{j}\right) \dy,
\end{equation}
{{where the \(\gamma\)-term in Lemma~\ref{l:truncation lemma} disappears because \(u^{\rho,k}_j\to u^k\) in \(L^1(Q_1^{\nu,k};\R^N)\) and \(u^k\in L^\infty\).}} Furthermore 
\[
\mathcal{H}^{n-1}(J_{\hat{u}^{\rho,k}_{j}}\cap Q_1^{\nu,k})+ \|\hat u^{\rho,k}_j-u^k\|_{L^1(Q_1^{\nu,k})}\leq 
\mathcal{H}^{n-1}(J_{{u}^{\rho}_{j}}\cap Q_1^{\nu,k})+ \|u^{\rho,k}_j-u^k\|_{L^1(Q_1^{\nu,k})}
\] 
where $\hat u^{\rho,k}_j:=\mathcal{T}_{m_{\rho}}(u^{\rho,k}_j)\in SBV(Q^{\nu,k}_1,\R^N)$. 
Therefore, choosing $M$ such that $a_M>|a|$,  
it follows from \eqref{e:i trhok cantor} and \eqref{e:seconda doppia convergenza Radon Cantor} that
\begin{equation}\label{e:la misura dei salti sta andando a zero cantor}
    \lim_{\rho \to 0,\rho\in I(x_0)} \lim_{j\to \infty}( \mathcal{H}^{n-1}(J_{\hat{u}^{\rho,k}_{j}}\cap Q_1^{\nu,k})+ \|\hat u^{\rho,k}_j-u^k\|_{L^1(Q_1^{\nu,k})})=0\,.
\end{equation}
Next we change the boundary datum $\hat u^{\rho,k}_j$ on a neighborhood of $\partial^\perp Q^{\nu,k}$ with $u^k$.
For every $\rho$ small enough and every $j$ large enough (depending on $\rho$), we have that
\begin{equation*}
    \|\hat u^{\rho,k}_j-u^k\|_{L^1(Q_1^{\nu,k})}+\frac{\rho}{j}=:\tau_{\rho,j}\in (0,1)
\end{equation*}
and thanks to Fubini's Theorem and the trace properties of $BV$ functions on rectifiable sets, there exists $q_{\rho,j}\in (1/2-\tau_{\rho,j}^{1/2},1/2)$
\begin{equation}\label{e:dato al bordo oculato 2}
    \int_{\partial^{\perp} R_{\nu}(B^{k}_{\rho,j})}|(\hat u^{\rho,k}_j)^--u^k| \dd \calH^{n-1} \leq \frac{ \|{{\hat u^{\rho,k}_j-u^k}}\|_{L^1(Q_1^{\nu,k})}}{\tau_{\rho,j}^{1/2}}\leq \tau_{\rho,j}^{1/2},
\end{equation}
where $(\hat u^{\rho,k}_j)^-$ is the inner trace of $\hat u^{\rho,k}_j$ on $R_{\nu}(B^{k}_{\rho,j})$, where 
$B^{k}_{\rho,j}:=(-\frac{k}{2},\frac{k}{2})^{n-1}\times (-q_{\rho,j},q_{\rho,j})$ and $R_\nu$ is the rotation in \ref{e:Rnu} of the notation list.
Defining the functions $w^{\rho,k}_{j}\in BV(Q^{\nu,k}_1,\R^N)$ as 
\begin{equation*}
w^{\rho,k}_{j}(y)=
\begin{cases}
\hat u^{\rho,k}_j(y) \; & \textup{if $y\in R_{\nu}(B^{k}_{\rho,j})$} \\
 u^k \; & \textup{if $y\in Q_1^{\nu,k}\setminus  R_{\nu}(B^{k}_{\rho,j})$}, 
\end{cases}
\end{equation*}
we have that 
\begin{align*}
  & \frac{\alpha_{\delta}}{k^{n-1}t_{\rho,k}}\int_{Q_1^{\nu,k}}\textstyle f(\frac{x_0+\rho y}{\eps_j},k^{n-1}t_{\rho,k}\nabla \hat{u}^{\rho,k}_{j}) \dy+ \alpha_{\delta}C\|\nabla u^k\|_{L^1(Q_1^{\nu,k}\setminus R_{\nu}(B^{k}_{\rho,j}))} +  \\  & + \frac{\alpha_{\delta}}{k^{n-1}t_{\rho,k}}\mathcal{L}^n(Q_1^{\nu,k}\setminus R_{\nu}(B^{k}_{\rho,j})) \geq \frac{\alpha_{\delta}}{k^{n-1}t_{\rho,k}}\int_{Q_1^{\nu,k}} \textstyle f(\frac{x_0+\rho y}{\eps_j},k^{n-1}t_{\rho,k}\nabla w^{\rho,k}_{j}) \dy.
\end{align*}
Since $\displaystyle{\lim_{\rho\to 0}\lim_{j\to \infty} \mathcal{L}^n(Q_1^{\nu,k}\setminus R_{\nu}(B^{k}_{\rho,j}))=0}$, 
we get from \eqref{e:bound L^1 gradienti cantor}
\begin{equation}\label{e:terzo doppio limite cantor}
    \Big(1+\frac{C^2}{M}\Big) \frac{\dd \mu^{\delta}}{\dd |D^cu|}(x_0)\geq \limsup_{\rho \to 0, \; \rho\in I(x_0)} \limsup_{j\to \infty} \frac{\alpha_{\delta}}{k^{n-1}t_{\rho,k}}\int_{Q_1^{\nu,k}}\textstyle f(\frac{x_0+\rho y}{\eps_j},k^{n-1}t_{\rho,k}\nabla w^{\rho,k}_{j}) \dy\,.
\end{equation}
In particular, we infer that 
\begin{equation}\label{e:bound gradiente wrho,kj}
    \Big(1+\frac{C^2}{M}\Big)  \frac{\dd \mu^{\delta}}{\dd |D^cu|}(x_0)\geq 
    C\alpha_{\delta}\limsup_{\rho \to 0, \; \rho\in I(x_0)} \limsup_{j\to \infty} \int_{Q_1^{\nu,k}}|\nabla w^{\rho,k}_j| \dy\,,
\end{equation}
and thus by \eqref{e:dato al bordo oculato 2} we conclude 
\begin{align*}
  |D^sw^{\rho,k}_j|(Q_1^{\nu,k})&\leq |D^s \hat u^{\rho,k}_j|(R_{\nu}(B^k_{\rho,j}))+|D^s u^k|(Q_1^{\nu,k}\setminus R_{\nu}(B^k_{\rho,j}))+\tau_{\rho,j}^{\frac{1}{2}} \\ 
  & \leq 2a_{2M+1}\mathcal{H}^{n-1}(J_{\hat{u}^{\rho,k}_j}\cap Q_1^{\nu,k})+ |D^s u^k|(Q_{{{1}}}^{\nu,k}\setminus R_{\nu}(B^k_{\rho,j}))+\tau_{\rho,j}^{\frac{1}{2}} \nonumber
\end{align*}
and therefore
\begin{equation}\label{e:bound singular part of the gradient cantor}
    \lim_{\rho\to 0, \rho\in I(x_0)}\lim_{j\to \infty}|D^sw^{\rho,k}_j|(Q_1^{\nu,k})=0.
\end{equation}
In addition, thanks to Lemma~\ref{l:lemma funz recess integrale} and \eqref{e:terzo doppio limite cantor} we deduce that 
\begin{equation}\label{e:quarto doppio limite cantor}
   \Big(1+\frac{C^2}{M}\Big) \frac{\dd \mu^{\delta}}{\dd |D^cu|}(x_0)\geq \limsup_{\rho \to 0, \; \rho\in I(x_0)} \limsup_{j\to \infty} \alpha_{\delta}\int_{Q_1^{\nu,k}}\textstyle f^{\infty}(\frac{x_0+\rho y}{\eps_j},\nabla w^{\rho,k}_{j}) \dy.
\end{equation}

We now change the boundary datum $w^{\rho,k}_j$ on a neighborhood of $\partial^{\|} Q_1^{\nu,k}$ with $\ell_k$.
Let $h_k\in (0,k)$ be such that $\calL^n(Q_1^{\nu,k}\setminus Q_1^{\nu,h_k})\to 0$ as $k\to \infty$, necessarily $h_k\to \infty$. 
Then, by Fubini's Theorem there exists $\lambda^k_{\rho,j} \in (k-h_k,k)$ such that  
\begin{equation}\label{e:traccia parallela cantor}
  \int_{\partial^{\|}Q_1^{\nu,\lambda^k_{\rho,j}}}|(w^{\rho,k}_j)^--\ell_k| \mathcal{H}^{n-1} \leq \frac{2}{h_k}\int_{Q_1^{\nu,k}}|w^{\rho,k}_j-\ell_k|\dy,
\end{equation}
where $(w^{\rho,k}_j)^-$ is the inner trace of $w^{\rho,k}_j$ on $\partial^{\|}Q_1^{\nu,\lambda^k_{\rho,j}}$. Furthermore, since $w^{\rho,k}_j=u^k=\ell_{k}$ 
on $\partial^{\perp}Q_1^{\nu,k}$, using Poincarè inequality on the one-dimensional restrictions of $w^{\rho,k}_j$ in the $\nu$ direction, we obtain that 
\begin{equation*}
    \int_{Q_1^{\nu,k}}|w^{\rho,k}_j-\ell_k|\dy \leq 2\Big|Dw^{\rho,k}_j-\frac{a\otimes \nu}{k^{n-1}}\Big|(Q_1^{\nu,k})\leq 2|Dw^{\rho,k}_j|(Q_1^{\nu,k})+2|a|\,.
\end{equation*}
Therefore, by \eqref{e:traccia parallela cantor} we infer that
\begin{equation}\label{e:seconda traccia parallela cantor}
      \int_{\partial^{\|}Q_1^{\nu,\lambda^k_{\rho,j}}}|(w^{\rho,k}_j)^--\ell_k| \mathcal{H}^{n-1} \leq  \frac{4}{h_k}|Dw^{\rho,k}_j|(Q_1^{\nu,k})+\frac{4}{h_k}|a|.
\end{equation}
Defining $\hat{w}^{\rho,k}_j$ as 
\begin{equation*}
  \hat{w}^{\rho,k}_j(y)=
  \begin{cases}
      w^{\rho,k}_j \; & \textup{if $y\in Q_1^{\nu,\lambda^k_{\rho,j}}$} \\
      \ell_k(y) \; & \textup{if $y\in Q_1^{\nu,k}\setminus  Q_1^{\nu,\lambda^k_{\rho,j}}$}
   \end{cases}
\end{equation*}
we have that $\hat{w}^{\rho,k}_j\in BV(Q_1^{\nu,k},\R^N)$, $\hat{w}^{\rho,k}_j=\ell_k$ on $\partial Q_1^{\nu,k}$, and by \eqref{e:seconda traccia parallela cantor}
\begin{align*}
   |D^s \hat{w}^{\rho,k}_j|(Q_1^{\nu,k}) &\leq |D^sw^{\rho,k}_j|(Q_1^{\nu,k})+ \frac{4}{h_k}|Dw^{\rho,k}_j|(Q_1^{\nu,k})+\frac{4}{h_k}|a|\\ 
   & \leq 2|D^s w^{\rho,k}_j|(Q_1^{\nu,k}) + \frac{4}{h_k}\|\nabla w^{\rho,k}_j\|_{L^1(Q_1^{\nu,k})}+\frac{4}{h_k}|a|.    
\end{align*}
In particular, by \eqref{e:bound singular part of the gradient cantor}, we obtain  
\begin{equation}\label{e:limsup k,rho,j}
     \limsup_{\rho \to 0, \rho\in I(x_0)}\limsup_{j\to \infty}  |D^s \hat{w}^{\rho,k}_j|(Q_1^{\nu,k})=\frac{4}{h_k}  \limsup_{\rho \to 0, \rho\in I(x_0)}\limsup_{j\to \infty}\|\nabla w^{\rho,k}_j\|_{L^1(Q_1^{\nu,k})}+ \frac{4|a|}{h_k}
\end{equation}
and, from \eqref{e:quarto doppio limite cantor},
\begin{align}\label{e:quarto bis doppio limite cantor}
  \Big(1+\frac{C^2}{M}\Big) \frac{\dd \mu^{\delta}}{\dd \calL^n}(x_0)&+\frac{C|a|\mathcal{L}^n(Q_1^{\nu,k}\setminus Q_1^{\nu,h_k})}{k^{n-1}}\nonumber\\ 
  & \geq \limsup_{\rho \to 0, \; \rho\in I(x_0)} \limsup_{j\to \infty} \alpha_{\delta}\int_{Q_1^{\nu,k}}\textstyle f^{\infty}(\frac{x_0+\rho y}{\eps_j},\nabla \hat{w}^{\rho,k}_{j}) \dy.
\end{align}
We now argue as in Proposition~\ref{p:Homogenised volume integrand} (cf. \eqref{e:teo rilassamento BFM in volume part 1}, \eqref{e:approx al rilassato volume part 1}), and for every $\rho$ and $j$ fixed we use Lemma~\ref{l:lemma del rilassato} to infer the existence of $\overline{w}^{\rho,k}_j\in W^{1,1}(Q_1^{\nu,k},\R^N)$ such that 
$\overline{w}^{\rho,k}_j=\ell_k$ on $\partial Q_1^{\nu,k}$ and
\begin{align*}
 \int_{Q_1^{\nu,k}}& \textstyle f^{\infty}(\frac{x_0+\rho y}{\eps_j},\nabla \hat{w}^{\rho,k}_{j}) \dy
 \geq \displaystyle \int_{Q_1^{\nu,k}} \textstyle \widehat{f^{\infty}}(\frac{x_0+\rho y}{\eps_j},\nabla \hat{w}^{\rho,k}_{j}) \dy\\ 
 &\geq
 sc^-(L^1)F^\infty_{\rho,j}(\hat{w}^{\rho,k}_{j})-C|D^s \hat{w}^{\rho,k}_j|(Q_1^{\nu,k})\\
& \geq\int_{Q_1^{\nu,k}}\textstyle\widehat{f^{\infty}}(\frac{x_0+\rho y}{\eps_j},\nabla \overline{w}^{\rho,k}_{j})\dy
-C|D^s \hat{w}^{\rho,k}_j|(Q_1^{\nu,k})- \displaystyle\frac\rho j\,.
\end{align*}
Therefore, by \eqref{e:bound gradiente wrho,kj}, \eqref{e:limsup k,rho,j}, \eqref{e:quarto bis doppio limite cantor} 
we conclude {{from}} the last inequality above that
\begin{align}\label{e:quinto doppio limite cantor}
  \Big(1+\frac{C^2}{M}-\frac{4}{h_k}\Big) &\frac{\dd \mu^{\delta}}{\dd \calL^n}(x_0) +\frac{C|a|\mathcal{L}^n(Q_1^{\nu,k}\setminus Q_1^{\nu,h_k})}{k^{n-1}}
  +C\alpha_\delta\frac{4|a|}{h_k}\nonumber \\ 
  & \geq \limsup_{\rho \to 0, \; \rho\in I(x_0)} \limsup_{j\to \infty} \alpha_{\delta}
  \int_{Q^{\nu,k}}\textstyle\widehat{f^\infty}(\frac{x_0+\rho y}{\eps_j},\nabla \overline{w}^{\rho,k}_{j}) \dy \,.
\end{align}
Set $r_{\rho,j}:=\frac{\rho}{\eps_j}$ and $\tilde{w}^{\rho,k}_j(x):=r_{\rho,j}\overline w^{\rho,k}_j\left(\frac{x}{r_{\rho,j}}-\frac{x_0}{\rho}\right)$, 
we have that $\tilde{w}^{\rho,k}_j\in W^{1,1}(Q^{\nu,k}_{r_{\rho,j}}\left(\frac{r_{\rho,j}}{\rho}x_0\right),\R^N)$ with $\tilde{w}^{\rho,k}_j=\ell_{k}-\frac{x_0}{\eps_j}$ on $\partial Q_{r_{\rho,j}}\left(\frac{r_{\rho,j}}{\rho}x_0\right)$ and, thanks to the $1$-homogeneity of $\widehat{f^{\infty}}$ (cf. item (i) in 
Lemma~\ref{l:lemma del rilassato}), we infer that
\begin{equation*}
   \int_{Q^{\nu,k}}\widehat{f^{\infty}}\left(\frac{x_0+\rho y}{\eps_j},\nabla \overline{w}^{\rho,k}_{j}\right) \dy
   =\frac{1}{k^{n-1}r_{\rho,j}^n}\int_{Q^{\nu,k}_{r_{\rho,j}}\big(\frac{r_{\rho,j}}{\rho}x_0\big)}\widehat{f^{\infty}}(y,k^{n-1}\nabla \tilde{w}^{\rho,k}_j) \dy.
\end{equation*}
In particular, since $k^{n-1}\tilde{w}^{\rho,k}_j=\ell_{a\otimes \nu}-\frac{k^{n-1}}{\eps_j}x_0$ on $\partial Q^{\nu,k}_{r_{\rho,j}}\left(\frac{r_{\rho,j}}{\rho}x_0\right)$, thanks to \eqref{e:quinto doppio limite cantor} we obtain
\begin{align*}
\Big(1+\frac{C^2}{M}-\frac{4}{h_k}\Big) &\frac{\dd \mu^{\delta}}{\dd \calL^n}(x_0) +\frac{C|a|\mathcal{L}^n(Q_1^{\nu,k}\setminus Q_1^{\nu,h_k})}{k^{n-1}}
+C\alpha_\delta\frac{4|a|}{h_k}\\ 
& \geq \limsup_{\rho \to 0, \; \rho\in I(x_0)} \limsup_{j\to \infty}\frac{\alpha_{\delta}}{k^{n-1}r_{\rho,j}^n}\int_{Q^{\nu,k}_{r_{\rho,j}}\left(\frac{r_{\rho,j}}{\rho}x_0\right)}\widehat{f^{\infty}}(y,k^{n-1}\nabla \tilde{w}^{\rho,k}_j) \dy \\ 
& \geq  \limsup_{\rho \to 0, \; \rho\in I(x_0)} \limsup_{j\to \infty}\alpha_{\delta}\frac{\minprobv^{\widehat{f^{\infty}}}\left(\ell_{a\otimes \nu},Q^{\nu,k}_{r_{\rho,j}}\left(\frac{r_{\rho,j}}{\rho}x_0\right)\right)}{k^{n-1}r_{\rho,j}^n} \\ 
& =  \limsup_{\rho \to 0, \; \rho\in I(x_0)} \limsup_{j\to \infty}\alpha_{\delta}\frac{\minprobv^{f^{\infty}}(\ell_{a\otimes \nu},Q^{\nu,k}_{r_{\rho,j}}\left(\frac{r_{\rho,j}}{\rho}x_0\right))}{k^{n-1}r_{\rho,j}^n}=\alpha_{\delta}f^{\infty}\homm(a\otimes \nu)\,,
\end{align*}
where the last-but-one equality follows from Lemma~\ref{l:lemma del rilassato} (vi), and 
the last equality {{follows}} from  Proposition~\ref{p:finfhom prop}.
Then, taking $k\to \infty$ and $M\to \infty$ in  this order, we infer that
\begin{equation*}
    \frac{\dd \mu^{\delta}}{\dd |D^c u|}(x_0)\geq \alpha_{\delta}f^{\infty}\homm(a\otimes \nu)=\alpha_{\delta}f^{\infty}\homm\Big(\frac{\dd D^c u}{\dd |D^c u|}(x_0)\Big).
\end{equation*}
Therefore, using \eqref{e:Hdeltaepsj vv hatF} we conclude that
\begin{equation*}
   \Fsucc(u,1,A')\geq  \liminf_{j\to \infty}H^{\delta}_{\eps_j}(u^{\delta}_j,A')= \liminf_{j\to \infty}\mu^{\delta}_j(A')\geq \mu^{\delta}(A')\geq \alpha_{\delta}\int_{A'} \textstyle f^{\infty}\homm(\frac{\dd D^c u}{\dd |D^c u|}(x))\dd |D^cu|,
\end{equation*}
thus \eqref{e:lowerbound parte di cantoor} follows by letting $\delta \to 0$ recalling that $\alpha_\delta\to 1$ and by the arbitrariness of
$A'\in\corA(A)$.
\end{proof}
Finally, we are in a position to prove the deterministic homogenisation result Theorem \ref{t:Deterministic Gamma-conv}. 
\begin{proof}[Proof of Theorem~\ref{t:Deterministic Gamma-conv}]
Theorem~\ref{t:Sottosucc gamma-conv.} implies that from any strictly positive infinitesimal sequence we can extract a subsequence $(\eps_j)$ such that 
  \begin{equation*}
     \Gamma(L^1_{\textup{loc}}(\R^n,\R^{N+1}))\text{-}\lim_{j \to \infty}\Functeps_{\eps_j}(u,1,A)=\Fsucc(u,1,A)
\end{equation*}
with $\Fsucc:L^1_{\textup{loc}}(\R^n,\R^{N+1})\times \corA \longrightarrow [0,+\infty]$. Moreover,
$\Fsucc(u,1,\cdot)$ is the restriction to open sets of a finite Radon measure on $A$ and $\Fsucc(u,1,A)\leq C(|Du|(A)+\calL^n(A))$
for every $A\in \corA$ and every $u\in L^1_{\textup{loc}}(\R^n,\R^N)$ such that $u\in BV(A,\R^N)$. 
Therefore, $\Fsucc(u,1,\cdot)$ is absolutely continuous respect to the measure $\calL^n \res A+ |D^cu|\res A+ \calH^{n-1}\res J_u\cap A$. 
Since $\calL^n \res A, |D^cu|\res A, \calH^{n-1}\res J_u\cap A$ are mutually singular,
by the properties of Radon-Nikodym derivatives, for every $B\in \corA(A)$ we have that
\begin{equation*}
    \Fsucc(u,1,B)=\int_B \frac{\dd \Fsucc(u,1,\cdot)}{\dd \calL^n} \dx + \int_B \frac{\dd \Fsucc(u,1,\cdot)}{\dd |D^cu|} \dd |D^cu| 
    + \int_{J_u\cap B} \frac{\dd \Fsucc}{\dd \calH^{n-1}\res J_u} \dd \calH^{n-1}\,.
\end{equation*}
In particular, if $A\in \corA$ and $u\in L^1_{\textup{loc}}(\R^n,\R^N)$ with $u\in BV(A,\R^N)\cap L^{\infty}(A,\R^N)$, 
Propositions~\ref{p:Homogenised volume integrand}, \ref{p:homogenised surface integrand} and \ref{p:homogenised cantor integrand} give that
\begin{equation*}
     \Gamma(L^1_{\textup{loc}}(\R^n,\R^{N+1}))\text{-}\lim_{j \to \infty}\Functeps_{\eps_j}(u,1,A)=\Fsucc(u,1,A)=F\homm(u,1,A),
\end{equation*}
where $F\homm$ is as in \eqref{e:det-G-lim}.
From \ref{e:crescita lineare}, for every $A\in \corA$ and every $(u,v)\in L^1_{\textup{loc}}(\R^n,\R^{N+1})$ with $(u,v)\in W^{1,1}(A,\R^N)\times W^{1,2}(A,[0,1])$ we have that 
\begin{equation*}
    \int_A\Big(C^{-1} v^2 |\nabla u| + \frac{(1-v)^2}{\eps_j}+ \eps_j |\nabla v|^2\Big)\dx\leq \Functeps_{\eps_j}(u,v,A),
\end{equation*}
hence, by \cite[Theorem 4.1]{AlicBrShah} and \cite[Remark~3.5]{AlFoc},
for every $(u,v)\in L^1_{\textup{loc}}(\R^n,\R^{N+1})$ such that $u \not \in GBV(A,\R^N)$ or $v\neq 1$ on $A$ we get
\begin{equation*}
   \Gamma(L^1_{\textup{loc}}(\R^n,\R^{N+1}))\text{-}\liminf_{j \to \infty}\Functeps_{\eps_j}(u,v,A)=F\homm(u,v,A)=+\infty\,.
\end{equation*}
Eventually, arguing exactly as in \cite[Section 6]{AlFoc} we obtain 
\begin{equation*}
     \Gamma(L^1_{\textup{loc}}(\R^n,\R^{N+1}))\text{-}\lim_{j \to \infty}\Functeps_{\eps_j}(u,1,A)=F\homm(u,1,A)
\end{equation*}
for every $A\in \corA$ and every $u\in L^1_{\textup{loc}}(\R^n,\R^N)$ such that $u\in GBV(A,\R^N)$. 
Indeed, the lower bound inequality for general $GBV$ maps follows easily from Lemma~\ref{l:truncation lemma} and 
the result in the $BV\cap L^\infty$-setting. Instead, the upper bound inequality is a consequence of the latter together 
with both the $L^1_{\textup{loc}}(\R^n,\R^{N+1})$ lower semicontinuity of $\Gamma\hbox{-}\limsup_{j\to\infty}\Functeps_{\eps_j}(\cdot,1,A)$ 
and the continuity of $F\homm$ along sequences of maps obtained via the smooth truncations $(\mathcal T_k)_{k\in\N}$,
namely $F\homm(\mathcal T_k(u),1,A)\to F\homm(u,1,A)$ as $k\to\infty$ for every $u\in GBV(A,\R^N)$ (cf. \cite[Lemma~6.1]{AlFoc}).

Since the $\Gamma$-limit does not depend on the extracted subsequence {{Urysohn's}} property of $\Gamma$-convergence yields the claim.
\end{proof}

\section{Stochastic homogenisation}\label{s:stochastic}

This section is devoted to the proof of the stochastic homogenisation result stated in Theorem~\ref{t:Stochastic homogenisation}. The proof will be achieved by showing that if $f$ is a stationary random integrand in the sense of Definition \ref{d:def stationary integrand}, then the assumptions of Theorem \ref{t:Deterministic Gamma-conv} are satisfied for $P$-a.e. $\omega \in \Omega$. Here a pivotal role is played by the Subadditive Ergodic Theorem, Theorem \ref{t:ergodic theorem}.  

\medskip

The following proposition establishes the existence and spatial homogeneity of $f\homm$. The proof can be found in \cite[Proposition~9.1]{CDMSZGlobal} and in \cite[Lemma~4.1]{RufZepp}.
\begin{proposition}[Homogenized random volume integrand]\label{p:Volume part stochastic proposition}
 Let $f$ be a stationary random integrand. Then there exist $\Omega'\in \calT$, with $P(\Omega')=1$ and a $\calT \otimes \Bor^{N\times n}$-measurable function $f\homm:\Omega \times \R^{N\times n}\to [0,+\infty)$ such that for every $\omega \in \Omega'$, $x\in \R^n$, $\xi \in \R^{N\times n}$, $\nu \in \Sf^{n-1}$ and $k\in \N$ 
 \begin{equation*}
     f\homm(\omega,\xi)=\lim_{r\to +\infty}\frac{\minprobv^{f_{\omega}}(\ell_{\xi},Q^{\nu,k}_{r}(rx))}{k^{n-1}r^n}=\lim_{r\to +\infty}\frac{\minprobv^{f_{\omega}}(\ell_{\xi},Q_{r})}{r^n}.
     \end{equation*}
If in addition $f$ is ergodic, then $f\homm$ is independent of $\omega$ and
 \begin{equation*}
     f\homm(\xi)=\lim_{r\to +\infty}\frac{1}{r^n}\int_{\Omega}\minprobv^{f_{\omega}}(\ell_{\xi},Q_{r}) \dd P(\omega).
 \end{equation*}
\end{proposition}
Propositions~\ref{p:finfhom prop} and \ref{p:Volume part stochastic proposition} readily imply the following result.
\begin{proposition}[Homogenized random Cantor integrand]\label{p:cantor part stochastic}
   Let $f$ be a stationary random integrand. Then there exist $\Omega'\in \calT$, with $P(\Omega')=1$ and a $\calT \otimes \Bor^{N\times n}$-measurable function $f^{\infty}\homm :\Omega \times \R^{N\times n}\to [0,+\infty)$ such that for every $\omega \in \Omega'$, every $\xi \in \R^{N\times n}$ every $k\in \N$, every $x\in \R^n$ and every $\nu \in \Sf^{n-1}$ 
   \begin{equation*}
       f^{\infty}\homm(\omega,\xi)=\lim_{t\to +\infty}\frac{f\homm(\omega,t\xi)}{t}
   \end{equation*}
   and
   \begin{equation*}
       f^{\infty}\homm(\omega,\xi)=\lim_{r\to +\infty}\frac{\minprobv^{f^{\infty}_{\omega}}(\ell_{\xi},Q^{\nu,k}_{r}(rx))}{k^{n-1}r^n}=\lim_{r\to +\infty}\frac{\minprobv^{f^{\infty}_{\omega}}(\ell_{\xi},Q_{r})}{r^n}.
   \end{equation*}
If in addition $f$ is ergodic, then $f^{\infty}\homm$ is independent of $\omega$ and
\begin{equation}
    f^{\infty}\homm(\xi)=\lim_{r\to +\infty}\frac{1}{r^n}\int_{\Omega}\minprobv^{f^{\infty}_{\omega}}(\ell_{\xi},Q_{r}) \dd P(\omega).
\end{equation}
\end{proposition}
The analogous result for the surface integrand is more involved and requires a new proof.
\begin{proposition}[Homogenized random surface integrand]\label{p:surface part stochastic}
    Let $f$ be a stationary random integrand. Then there exist $\Omega'\in \calT$, with $P(\Omega')=1$ and a $\calT \otimes \Bor^{N}\otimes \Bor^n_S$-measurable function $g\homm:\Omega \times \R^{N}\times \Sf^{n-1}\to [0,+\infty)$ such that for every $\omega \in \Omega'$, $x\in \R^n$, $\zeta \in \R^{N}$ and $\nu \in \Sf^{n-1}$
    \begin{equation*}
        g\homm(\omega,\zeta,\nu)=\lim_{r\to +\infty}\frac{\minprobs^{f^{\infty}_{\omega}}(u_{rx,\zeta,\nu},Q^{\nu}_r(rx))}{r^{n-1}}=\lim_{r\to +\infty}\frac{\minprobs^{f^{\infty}_{\omega}}(u_{\zeta,\nu},Q^{\nu}_r)}{r^{n-1}}.
    \end{equation*}
    If in addition $f$ is ergodic, then $g\homm$ is independent of $\omega$ and 
    \begin{equation}\label{e:formula ghom ergodic case}
        g\homm(\zeta,\nu)=\lim_{r\to +\infty}\frac{1}{r^{n-1}}\int_{\Omega}\minprobs^{f^{\infty}_{\omega}}(u_{\zeta,\nu},Q^{\nu}_r) \dd P(\omega).
    \end{equation}
\end{proposition}

\begin{proof}
We divide the proof into a number of steps. 

\medskip

    \textit{Step 1:} Let $\overline{u}_{\zeta,\nu}$ be as in \ref{e:funzione salto} of the notation list. In this step we prove that for every $\zeta \in \Q^N$ and $\nu \in \Sf^{n-1}\cap \Q^{n}$ and for $P$-a.e. $\omega \in \Omega$ there exists the limit   
    \[\lim_{r\to +\infty}\frac{\minprobs^{f^{\infty}_{\omega}}(\overline{u}_{\zeta,\nu},Q^{\nu}_r)}{r^{n-1}}
    \] 
    and defines an $x$-independent random variable.

    To prove the claim let $\nu \in \Sf^{n-1}\cap \Q^{n}$ and $\zeta \in \Q^N$ be fixed, $R_{\nu}\in O(n)\cap \Q^{n\times n}$ be the orthogonal matrix as in \ref{e:Rnu} of the notation list, and $M_{\nu}$ be a positive integer such that $M_{\nu}R_{\nu}\in \Z^{n\times n}$, so that $M_{\nu}R_{\nu}(z',0)\in \Pi^{\nu}_0 \cap \Z^n$. Given $A'=[a_1,b_1)\times \dots \times [a_{n-1},b_{n-1})\in \calI_{n-1}$ we define the $n$-dimensional interval $T_{\nu}(A')$ as 
    \begin{equation}\label{e:def intervallo n-dim}
        T_{\nu}(A'):= M_{\nu}R_{\nu}(A'\times [-c,c)), \quad \textup{with $c:=\frac{1}{2}\max_{1\leq j \leq n-1}(b_j-a_j)$}.
    \end{equation}
For every $\omega \in \Omega$ and every $A'\in \calI_{n-1}$ we set
\begin{equation}\label{e:def. processo subadditivo}
    \mu_{\zeta,\nu}(\omega,A'):=\frac{1}{M^{n-1}_{\nu}}\minprobs^{f^{\infty}_{\omega}}(\overline{u}_{\zeta,\nu},T_{\nu}(A')).
\end{equation}
We now show that $\mu_{\zeta,\nu}: \Omega \times \calI_{n-1} \to [0,+\infty)$ defines an $(n-1)$-dimensional subadditive process on $(\Omega,\calT, P)$. The separability and completeness of $W^{1,1}(A,\R^N)\times W^{1,2}(A,[0,1])$ for every $A\in \corA$ combined with \cite[Lemma C.2]{RufRuf} and $(f2)$ in Definition~\ref{d:Admissible deterministic integrand} give the $\calT$-measurability of the map $\omega \mapsto \minprobs^{f^{\infty}_{\omega}}(\overline{u}_{\zeta,\nu},T_{\nu}(A'))$ for every $A'\in \calI_{n-1}$. 

Next, we prove that $\mu_{\zeta,\nu}$ is stationary with respect to an $(n-1)$-dimensional group of $P$-preserving transformations $(\tau^\nu_{z'})_{z'\in \Z^{n-1}}$. 
To this end, fix $z'\in \Z^{n-1}$ and $A'\in \calI_{n-1}$. By \eqref{e:def intervallo n-dim} we have that
\begin{equation*}
    T_{\nu}(A'+z')
    =M_{\nu}R_{\nu}(A'\times [-c,c))+M_{\nu}R_{\nu}(z',0)=T_{\nu}(A')+z'_{\nu},
\end{equation*}
where $z'_{\nu}:=M_{\nu}R_{\nu}(z',0)\in \Pi^{\nu}\cap \Z^n$. Thus by \eqref{e:def. processo subadditivo} we get
\begin{align}\label{e:cov_0}
    \mu_{\zeta,\nu}(\omega,A'+z')&=\frac{1}{M_{\nu}^{n-1}}\minprobs^{f^{\infty}_{\omega}}(\overline{u}_{\zeta,\nu},T_{\nu}(A'+z'))
    =\frac{1}{M_{\nu}^{n-1}}\minprobs^{f^{\infty}_{\omega}}(\overline{u}_{\zeta,\nu},T_{\nu}(A')+z'_{\nu}).
\end{align}
Now let $u,v$ be test functions in the definition of $\minprobs^{f^{\infty}_{\omega}}(\overline{u}_{\zeta,\nu},T_{\nu}(A')+z'_{\nu})$ and for $x\in T_\nu(A')$ set 
\[
\tilde u(x):= u(x+z'_\nu) \quad \text{and} \quad \tilde v(x):= v(x+z'_\nu).
\]
Then, a change of variables together with the stationarity of $f$ yield 
\[
S^{f^\infty(\omega)}(u,v,T_{\nu}(A')+z'_{\nu})=S^{f^\infty(\tau_{z'_{\nu}}\omega)}(\tilde u,\tilde v,T_{\nu}(A')).
\]
Set $(\tau^\nu_{z'})_{z'\in \Z^{n-1}}:=(\tau_{z'_\nu})_{z'\in \Z^{n-1}}$; we notice that 
$(\tau^\nu_{z'})_{z'\in \Z^{n-1}}$ is well defined since $z'_{\nu}\in \mathbb Z^n$ and it defines a group of $P$-preserving transformations on $(\Omega, P, \mathcal T)$. Then, the equality above can be rewritten as
\begin{equation}\label{e:cov}
S^{f^\infty(\omega)}(u,v,T_{\nu}(A')+z'_{\nu})=S^{f^\infty(\tau^\nu_{z'}\omega)}(\tilde u,\tilde v,T_{\nu}(A')).
\end{equation}
Moreover, since $z'_{\nu}\in \Pi^{\nu}\cap \Z^n$ we also have that $\tilde u=\bar u_{\zeta,\nu}$ on $\partial T_{\nu}(A')$. Thus gathering \eqref{e:cov_0} and \eqref{e:cov}, by the arbitrariness of $\tilde u, \tilde v$ we infer 
\[
\mu_{\zeta,\nu}(\omega,A'+z')=\mu_{\zeta,\nu}(\tau^\nu_{z'}\omega,A'),
\]
and hence the stationarity of $\mu_{\zeta,\nu}$ with respect to $(\tau^\nu_{z'})_{z'\in \Z^{n-1}}$. 
    
    To show that $\mu_{\zeta,\nu}$ is subadditive in $\calI_{n-1}$, fix $\omega\in \Omega$ 
    and $A'\in \calI_{n-1}$ and let $(A'_i)_{1\leq i\leq M}\subset \calI_{n-1}$ be a finite family of pairwise disjoint sets such that $A'=\cup_{i=1}^MA'_i$. 
    For every $\eta>0$ and $i\in\{1,\dots,M\}$, let $(u_i,v_i) \in W^{1,1}(T_{\nu}(A'_i),\R^N)\times W^{1,2}(T_{\nu}(A'_i),[0,1])$ with $(u_i,v_i)=(\overline{u}_{\zeta,\nu},1)$ on $\partial T_{\nu}(A'_i)$ such that
    \begin{equation*}
        \int_{T_{\nu}(A'_i)}\left(v_i^2f^{\infty}(\omega,y,\nabla u_i)+(1-v_i)^2+ |\nabla v_i|^2\right) \dy 
        \leq \minprobs^{f^{\infty}_{\omega}}(\overline{u}_{\zeta,\nu},T_{\nu}(A'_i))+\eta.
    \end{equation*}
Note that 
by construction we always have $\cup_{i=1}^M T_{\nu}(A'_i)\subseteq T_{\nu}(A')$, thus we define 
\begin{equation*}
(u(y),v(y)):=
    \begin{cases}
       (u_i(y),v_i(y)) \; & \; \textup{if $y\in T_{\nu}(A'_i)$} \\
       (\overline{u}_{\zeta,\nu},1) \; & \; \textup{if $y\in T_{\nu}(A')\setminus \cup_i T_{\nu}(A'_i)$}.
    \end{cases}
\end{equation*}
In particular, $(u,v)\in W^{1,1}(T_{\nu}(A'),\R^N)\times W^{1,2}(T_{\nu}(A'),[0,1])$ with $(u,v)=(\overline{u}_{\zeta,\nu},1)$ on $\partial T_{\nu}(A')$. 
Hence, we get
\begin{align*}
 & \minprobs^{f^{\infty}_{\omega}}(\overline{u}_{\zeta,\nu},T_{\nu}(A')) \leq 
 \int_{T_{\nu}(A')}\left(v^2f^{\infty}(\omega,y,\nabla u)+(1-v)^2+ |\nabla v|^2\right) \dy \\ 
 & = \sum_{i=1}^M \int_{T_{\nu}(A'_i)}\left( v_i^2f^{\infty}(\omega,y,\nabla u_i)+(1-v_i)^2+ |\nabla v_i|^2 \right)\dy 
 \leq \sum_{i=1}^M\minprobs^{f^{\infty}_{\omega}}(\overline{u}_{\zeta,\nu},T_{\nu}(A'_i)) + M\eta
\end{align*}
and the subadditivity follows by the arbitrariness of $\eta> 0$. 

Finally, we show that $\mu_{\zeta,\nu}$ is bounded. To this end we observe that for every $A'\in \calI_{n-1}$ and every $\omega\in \Omega$ we have 
\begin{equation*}
    \mu_{\zeta,\nu}(\omega,A')=\frac{1}{M^{n-1}_{\nu}}\minprobs^{f^{\infty}_{\omega}}(\overline{u}_{\zeta,\nu},T_{\nu}(A'))\leq \frac{C}{M_{\nu}^{n-1}}\int_{T_{\nu}(A')}|\nabla \overline{u}_{\zeta,\nu}| \dy \leq C|\zeta|\|\overline{\textup{u}}'\|_{L^{\infty}(\R)}\calL^{n-1}(A'),
\end{equation*}
where we used $T_{\nu}(A')\cap \Pi^{\nu}=M_{\nu}R_{\nu}(A'\times \{0\})$ and $\{|\nabla (\overline{u}_{\zeta,\nu}(y)|>0\}\subseteq    \{|y\cdot \nu|\leq 1/2\}$. Therefore, for every $\zeta \in \Q^N$ and $\nu \in \Sf^{n-1}\cap \Q^{n}$, ${\mu_{\zeta,\nu}}$ defines a subadditive process.  
Then, we can apply Theorem~\ref{t:ergodic theorem}  to deduce the existence of a $\calT$-measurable function $\psi_{\nu,\zeta}:\Omega \to [0,+\infty)$, and a set $\Omega_{\zeta,\nu} \in \calT$ with $P(\Omega_{\zeta,\nu})=1$, such that
for every $\omega \in \Omega_{\zeta,\nu}$ 
\begin{equation*}
   \psi_{\zeta,\nu}(\omega)= \lim_{r\to +\infty}\frac{\mu_{\zeta,\nu}(\omega,Q'_r)}{r^{n-1}}= \lim_{r\to +\infty}\frac{\minprobs^{f^{\infty}_{\omega}}(\overline{u}_{\zeta,\nu},Q^{\nu}_{ {{{M_{\nu}}}}r})}{({{{M_{\nu}}}}r)^{n-1}}\,.
\end{equation*}
where $Q'_r:=Q_r\cap\{x_n=0\}$, with $r>0$, {{is a $(n-1)$-regular family of sets (cf. Definition~\ref{d:regular family}).}}
\smallskip

\textit{Step 2:} In this step we prove the existence of $\tilde \Omega \in \mathcal T$ with $P(\tilde \Omega)=1$ such that for every $\omega \in \tilde \Omega$ and for every $\zeta \in \R^N$ and $\nu \in \Sf^{n-1}$ the following limit exists
\[
\lim_{r\to +\infty}\frac{\minprobs^{f^{\infty}_{\omega}}(\overline{u}_{\zeta,\nu},Q^{\nu}_r)}{r^{n-1}}\,
\] 
and defines an $x$-independent $\calT\otimes \Bor^N \otimes \Bor^n_S$-measurable function. 

To prove the claim let $\tilde \Omega$ denote the intersection of the sets $\Omega_{\zeta,\nu}$, as in Step 1, for $\zeta \in \Q^N$ and $\nu \in \Sf^{n-1}\cap \Q^{n}$. 
Clearly, $\tilde \Omega \in \calT$ and $P(\tilde \Omega)=1$. Let $\underline{g},\overline{g}:\tilde \Omega \times \R^N \times \Sf^{n-1}\to [0,+\infty]$ be the functions given by
 \begin{equation*}
\underline{g}(\omega,\zeta,\nu):=\liminf_{r\to +\infty}\frac{\minprobs^{f^{\infty}_{\omega}}(\overline{u}_{\zeta,\nu},Q^{\nu}_r)}{r^{n-1}}\,,\qquad
\overline{g}(\omega,\zeta,\nu):=\limsup_{r\to +\infty}\frac{\minprobs^{f^{\infty}_{\omega}}(\overline{u}_{\zeta,\nu},Q^{\nu}_r)}{r^{n-1}}.
 \end{equation*}
 By Step 1, for every $\omega \in \tilde \Omega$, every $\zeta \in \Q^{N}$ and every $\nu \in \Sf^{n-1}\cap \Q^n$ we have that 
 \begin{equation}\label{e:esistenza limite sui razionali}
     \underline{g}(\omega,\zeta,\nu)=\overline{g}(\omega,\zeta,\nu)\,.
 \end{equation}
 Furthermore, fixed $\omega \in \tilde \Omega$ and $\nu \in \Sf^{n-1}$, arguing as in Proposition~\ref{p:la ghom} (i) we have
 \begin{equation}\label{e:liminf lipschitz}
     |\underline{g}(\omega,\zeta_1,\nu)-\underline{g}(\omega,\zeta_2,\nu)|
     +     |\overline{g}(\omega,\zeta_1,\nu)-\overline{g}(\omega,\zeta_2,\nu)|\leq 2C\calH^{n-1}(\partial Q_1)|\zeta_1-\zeta_2|
 \end{equation}
for every $\zeta_1, \zeta_2\in \R^N$. 
 From \eqref{e:esistenza limite sui razionali} and \eqref{e:liminf lipschitz} we deduce that for every $\omega \in \tilde \Omega$, every $\zeta \in \R^N$ and every $\nu \in \Sf^{n-1}\cap \Q^n$
 \begin{equation}\label{e:esistenza per normali razionali}
       \underline{g}(\omega,\zeta,\nu)=\overline{g}(\omega,\zeta,\nu)\,,
 \end{equation}
 and that $\underline{g}(\cdot,\zeta,\nu):\tilde \Omega \to [0,+\infty)$ is $\calT$-measurable for every $\zeta \in \R^N$ and every $\nu \in \Sf^{n-1}\cap \Q^n$.
 
 We now claim that for every $\omega \in \tilde \Omega$ and every $\zeta \in \R^N$, the restrictions of the functions $\nu \mapsto \underline{g}(\omega,\zeta,\nu)$ and  $\nu \mapsto \overline{g}(\omega,\zeta,\nu)$ to the sets $\Hat \Sf^{n-1}_{\pm}$ are continuous. We show only the continuity of $\underline{g}$ on $\Hat \Sf^{n-1}_+$, the proof for $\overline{g}$ is analogous. To this end, let $\omega \in \tilde \Omega$, $\zeta \in \R^N$, $\nu \in \Hat \Sf^{n-1}_+$, then by density 
 let $(\nu_j)_{j\in\N}\subset\Hat \Sf^{n-1}_+\cap \Q^n$ be such that $\nu_j \to \nu$ as $j\to +\infty$. By the continuity of $\nu \mapsto R_{\nu}$ on $\Hat \Sf^{n-1}_+$, for every $\delta \in (0,1/2)$ there exists a $j_{\delta} \in \N$ such that
 \begin{equation}\label{e:cubi compattamente contenuti}
    Q^{\nu}_r \subset \subset Q^{\nu_j}_{(1+\delta)r} \subset \subset Q^{\nu}_{(1+2\delta)r}
 \end{equation}
 for every $j\geq j_{\delta}$ and every $r>0$.
 
 Setting $\kappa_j:=\max \{|R_{\nu_j}(e_i)\cdot \nu| \; : \; i=1,\dots,n-1\}$ we have that $\kappa_j \to 0$ as $j\to +\infty$, thanks to the continuity of $\nu \mapsto R_{\nu}$ on $\Hat \Sf^{n-1}_{\pm}$. We observe that for every $y\in \overline{Q^{\nu}_{r(1+\delta)}}$, we have $y=y'+(y\cdot \nu_j)\nu_j$ where 
 \begin{equation*}
     y'\in  R_{\nu_j}\Big(\big(-\frac{r}{2}(1+\delta),\frac{r}{2}(1+\delta)\big)^{n-1}\times \{0\}\Big)
 \end{equation*}
 and in particular, if in addition $|y\cdot \nu|\leq \frac{1}{2}$, for $j$ large enough depending only on $\delta$, we get 
 \begin{equation*}
   \textstyle  |y\cdot \nu_j|\leq \frac{|y'\cdot \nu|}{|\nu_j \cdot \nu|}+\frac{1}{2(1-\delta)}< \frac{(n-1)\kappa_jr (1+\delta)}{2(1-\delta)}+1=K(\delta)r\kappa_j+1,
 \end{equation*}
 where $K(\delta):=\frac{(n-1) (1+\delta)}{2(1-\delta)}$. Then, by applying Lemma \ref{l:lemma prob di minimo superficie su cubi diversi}, with $R=K(\delta)r\kappa_j+1$, we obtain
  \begin{align*}
    \minprobs^{f^{\infty}_{\omega}}(\overline{u}_{\zeta,\nu_j},Q^{\nu_j}_{r(1+\delta)})
    &\leq \minprobs^{f^{\infty}_{\omega}}(\overline{u}_{\zeta,\nu},Q^{\nu}_{r}) + \eta r^{n-1} \\ & + \tilde K(2\delta (1+2\delta)^{n-2}r^{n-1}+(K(\delta)r\kappa_j+1)|\zeta|(1+\delta)^{n-2}r^{n-2}).
    \end{align*} 
 Therefore, dividing by $r^{n-1}$, passing to the liminf as $r\to +\infty$, and to the limsup as $j\to+\infty$, and finally 
 letting $\eta,\delta \to 0$ we obtain
 \begin{equation*}
    \limsup_{j\to +\infty}\underline{g}(\omega,\zeta,\nu_j)\leq \underline{g}(\omega,\zeta,\nu). 
 \end{equation*}
 An analogous argument using the cubes $Q^{\nu_j}_{(1-\delta)r}$ shows that
\begin{equation*}
    \underline{g}(\omega,\zeta,\nu)\leq \liminf_{j\to +\infty} \underline{g}(\omega,\zeta,\nu_j)
\end{equation*}
 implying the claim. 

In particular, thanks to \eqref{e:esistenza per normali razionali}  we deduce that for every $\omega \in \tilde \Omega$, $\zeta \in \R^N$ and $\nu \in \Sf^{n-1}$
 \begin{equation}\label{e:esistenza per tutte le normali}
       \underline{g}(\omega,\zeta,\nu)=\overline{g}(\omega,\zeta,\nu)\,.
 \end{equation}
The $\calT$-measurability of $\underline{g}(\cdot,\zeta,\nu):\tilde \Omega \to [0,+\infty)$  for every $\zeta \in \R^N$ and $\nu \in \Sf^{n-1}$
follows from the analogous property for $\nu\in \Sf^{n-1}\cap \Q^n$. Furthermore, the map
$\underline{g}(\omega,\cdot,\cdot):\R^N\times \Hat\Sf^{n-1}_{\pm}\to [0,+\infty)$ is continuous for every $\omega \in \tilde \Omega$ thanks to 
 \eqref{e:liminf lipschitz}.

 Thus, defining $g\homm :\Omega \times \R^N \times \Sf^{n-1}\to [0,+\infty)$ by
 \begin{equation*}
 g\homm(\omega,\zeta,\nu):= 
\begin{cases}
        \underline{g}(\omega,\zeta,\nu) \; & \; \textup{if $\omega \in \tilde \Omega$} \\
        \frac{2|\zeta|}{C(|\zeta|+2)} \; & \; \textup{if $\omega \not \in \tilde \Omega$}, \\
     \end{cases}
\end{equation*}
 we have that $g\homm$ is $\calT\otimes \Bor^N \otimes \Bor^n_S$-measurable and, thanks to Corollary~\ref{c:comparison minimum problems},
 \begin{equation}\label{e:ghom achieved in stoch hom}
     g\homm(\omega,\zeta,\nu)=\lim_{r\to +\infty}\frac{\minprobs^{f^{\infty}_{\omega}}(\overline{u}_{\zeta,\nu},Q^{\nu}_r)}{r^{n-1}}=\lim_{r\to +\infty}\frac{\minprobs^{f^{\infty}_{\omega}}(u_{\zeta,\nu},Q^{\nu}_r)}{r^{n-1}}
 \end{equation}
 for every $\omega\in \tilde \Omega$, every $\zeta \in \R^N$ and every $\nu \in \Sf^{n-1}$.\\

\textit{Step 3:} In this step we show the existence of $\Omega'\in \calT$ with $\Omega'\subseteq \tilde \Omega$ and $P(\Omega')=1$, such that for every $\omega\in \Omega'$,
$z\in \Z^n$, $\zeta\in \Q^N$, $\nu \in \Sf^{n-1}\cap \Q^n$, and for every integer sequence $(r_k)$ with $r_k\geq k$ for every $k$ 
\begin{equation}\label{e:esistenza limite per interi step 4 stoch hom}
    \lim_{k\to +\infty}\frac{\minprobs^{f^{\infty}_{\omega}}(\overline{u}_{-kz,\zeta,\nu},Q^{\nu}_{r_k}(-kz))}{r_k^{n-1}}=g\homm(\omega,\nu,\zeta)\,.
\end{equation}
Let $z\in \Z^n$, $\zeta\in \Q^N$, $\nu \in \Sf^{n-1}\cap\Q^n$, $\eta>0$ and $\delta\in (0,1/4)$. Arguing exactly as in \cite[Theorem 6.1]{CDMSZStochom} we can prove the existence of a set $\Omega^{\zeta,\nu,\eta}_z \in \calT$, with $\Omega^{\zeta,\nu,\eta}_z\subseteq \tilde \Omega$, $P(\Omega^{\zeta,\nu,\eta}_z)=1$, and an integer $m_0=m_0(\zeta,\nu,\eta,z,\omega,\delta)>\frac{1}{\delta}$ satisfying the following property: for every $\omega\in \Omega^{\zeta,\nu,\eta}_z$ and for every integer $m\geq m_0$ there exists $i=i(\zeta,\nu,\eta,z,\omega,\delta,m)\in \{m+1,\dots,m+\ell\}$, with $\ell:=\lfloor 5m\delta \rfloor$, such that
\begin{equation}\label{e:stima ghom su discrete size}
    \Big|\frac{\minprobs^{f^{\infty}_{\omega}}(\overline{u}_{-iz,\zeta,\nu},Q^{\nu}_h(-iz))}{h^{n-1}}-g\homm(\omega,\zeta,\nu)\Big|\leq \eta \quad \textup{for every $h\in \N$ with $h\geq j_0$},
\end{equation}
where $j_0=j_0(\zeta,\nu,\eta,z,\omega,\delta)$, and $\lfloor s\rfloor$ denotes the integer part of $s\in\R$. 

Define $\Omega'$ as the intersection of the sets $\Omega^{\zeta,\nu,\eta}_z$ for $\zeta\in \Q^N$, $\nu \in \Sf^{n-1}\cap \Q^n$, $\eta\in \Q$, with $\eta>0$ and $z\in \Z^n$. Thus $\Omega'\subseteq \tilde \Omega$ and $P(\Omega')=1$. Let $\omega\in\Omega'$ and $r_k$ be as required, $\delta>0$ with $20\delta(|z|+1)<1$ and $\eta\in \Q$ with $\eta>0$. For every $k\geq 2m_0(\zeta,\nu,\eta,z,\omega,\delta)$, let $\underline{r}_k, \overline{r}_k \in \N$ be defined as
\begin{equation*}
    \underline{r}_k:=r_k-2(i_k-k)\lfloor |z|+1\rfloor \quad \textup{and} \quad \overline{r}_k:=r_k+2(i_k-k)\lfloor |z|+1\rfloor,
\end{equation*}
where 
\begin{equation}\label{e:discrete size}
    i_k=i(\zeta,\nu,\eta,z,\omega,\delta,k)\in \{k+1,\dots ,k+\lfloor 5k\delta\rfloor\},
\end{equation}
therefore, by construction, we have that $Q^{\nu}_{\underline{r}_k}(-i_kz)\subset \subset Q^{\nu}_{r_k}(-kz) \subset \subset Q^{\nu}_{\overline{r}_k}(-i_kz)$.

Since $20\delta(|z|+1)<1$, $k\leq r_k$ and $i_k-k\leq 5k\delta $ by \eqref{e:discrete size}, for every $y\in Q^{\nu}_{\overline{r}_k}(-i_{k}z)$ such that $|(y+i_kz)\cdot \nu|\leq \frac{1}{2}$, we obtain that 
\begin{align*}
|(y+kz)\cdot \nu|&= |(y+i_kz)\cdot \nu+ (kz-i_kz)\cdot \nu| \leq  \frac{1}{2}+
    |(kz-i_kz)\cdot \nu|    \\ 
    & \leq (i_k-k)|z|+\frac{1}{2}\leq 5k\delta |z| +\frac{1}{2}\leq 5r_k \delta |z|+\frac{1}{2},
\end{align*}
and ${r}_k-\underline{r}_k=2(i_k-k)\lfloor |z|+1 \rfloor \leq 10k\delta \lfloor |z|+1 \rfloor \leq 10 r_k \delta \lfloor |z|+1 \rfloor<\frac{r_k}{2}$. Applying Lemma \ref{l:lemma prob di minimo superficie su cubi diversi}, with $R=5r_k\delta |z|+\frac{1}{2}$, we obtain
\begin{align*}
   & \minprobs^{f^{\infty}_{\omega}}(\overline{u}_{-kz,\zeta,\nu},Q^{\nu}_{r_k}(-kz)) \leq \minprobs^{f^{\infty}_{\omega}}(\overline{u}_{-i_kz,\zeta,\nu},Q^{\nu}_{\underline{r}_k}(-i_kz))+\eta \underline{r}^{n-1}_k \\ & + \frac{\tilde K }{2}(10\delta ( |z|+1 ) r_k^{n-1}+(10r_k\delta |z|+1)|\zeta|r_k^{n-2}).
\end{align*}
In particular, from the latter estimate, \eqref{e:stima ghom su discrete size} and $\underline{r}_k\leq r_k$, for every $k$ large enough such that $\underline{r}_k\geq j_0(\zeta,\nu,\eta,z,\omega,\delta)$, we obtain
\begin{align*}
   & g\homm(\omega,\zeta,\nu)+\eta \geq  \frac{\minprobs^{f^{\infty}_{\omega}}(\overline{u}_{-i_kz,\zeta,\nu},Q^{\nu}_{\underline{r}_k}(-i_kz))}{\underline{r}_k^{n-1}} \geq \frac{\minprobs^{f^{\infty}_{\omega}}(\overline{u}_{-i_kz,\zeta,\nu},Q^{\nu}_{\underline{r}_k}(-i_kz))}{{r}_k^{n-1}}    \\ & \geq \frac{\minprobs^{f^{\infty}_{\omega}}(\overline{u}_{-kz,\zeta,\nu},Q^{\nu}_{r_k}(-kz))}{r_k^{n-1}}-\eta-\frac{\tilde K }{2}(10\delta ( |z|+1 )+(10\delta |z|+\frac{1}{r_k})|\zeta|)
\end{align*}
and thus, taking the limsup for $k\to +\infty$ and letting $\eta,\delta \to 0$, we get
\begin{equation*}
    g\homm(\omega,\zeta,\nu) \geq \limsup_{k\to +\infty}\frac{\minprobs^{f^{\infty}_{\omega}}(\overline{u}_{-kz,\zeta,\nu},Q^{\nu}_{r_k}(-kz))}{{r}_k^{n-1}}. 
\end{equation*}
Arguing analogously with the external cubes $Q^{\nu}_{\overline{r}_k}(-i_kz)$ we get
\begin{equation*}
 g\homm(\omega,\zeta,\nu) \leq  \liminf_{k\to +\infty}\frac{\minprobs^{f^{\infty}_{\omega}}(\overline{u}_{-kz,\zeta,\nu},Q^{\nu}_{r_k}(-kz))}{{r}_k^{n-1}} , 
\end{equation*}
obtaining the claim. 

\textit{Step 4:} Let $\Omega'$ be the set introduced in Step~3, then for every $\omega\in \Omega'$, $x\in \R^n$, $\zeta\in \Q^N$ 
and $\nu \in \Sf^{n-1}\cap\Q^n$ there holds 
\begin{equation}\label{e:esistenza limite per ogni x con zeta e nu razionali}
\lim_{r\to +\infty}\frac{\minprobs^{f^{\infty}_{\omega}}(\overline{u}_{rx,\zeta,\nu},Q^{\nu}_r(rx))}{r^{n-1}}=g\homm(\omega,\zeta,\nu)\,.
\end{equation}
Fix $\omega,x,\zeta,\nu$ as required, $\eta\in(0,\frac12)$, $q\in \Q^n$ with $|x-q|<\eta$, and $h\in \Z$ such that $z:=hq\in \Z^n$. Consider a sequence of real numbers $t_k\to +\infty$ as $k\to +\infty$ and let $s_k:=\frac{t_k}{h}$. Fixing an integer $j>2|z|+1$ and setting $r_k:=\lfloor t_k+2\eta t_k\rfloor +j$ we have that
$Q^{\nu}_{t_k}(t_kx) \subset \subset Q^{\nu}_{r_k}(\lfloor s_k \rfloor z)$. Since $|(t_kx-\lfloor s_k \rfloor z)\cdot \nu|\leq |t_k x - t_k q|+|s_k z- \lfloor s_k \rfloor z|\leq t_k \eta +|z|$, for every $y\in Q^{\nu}_{r_k}(\lfloor s_k \rfloor z )$ such that $|(y-t_kx)\cdot \nu |\leq \frac{1}{2}$ we have that $|(y-\lfloor s_k \rfloor z)\cdot \nu|\leq t_k \eta +|z|+\frac{1}{2}$.
In particular, for $k$ large enough depending only on $z$, we can apply Lemma \ref{l:lemma prob di minimo superficie su cubi diversi}, with $R=t_k \eta +|z|+\frac{1}{2}$, to obtain
\begin{align}\label{e:penultimo passagio parte stocastica}
   & \minprobs^{f^{\infty}_{\omega}}(\overline{u}_{\lfloor s_k \rfloor z,\zeta,\nu},Q^{\nu}_{r_k}(\lfloor s_k \rfloor z))  \leq \minprobs^{f^{\infty}_{\omega}}(\overline{u}_{t_kx,\zeta,\nu},Q^{\nu}_{t_k}(t_kx)) + \eta r_k^{n-1} \nonumber \\ & + \tilde K((2\eta r_k+j)r_k^{n-2}+(r_k\eta+|z|+\frac{1}{2})|\zeta|r_k^{n-2}),  
\end{align}
where we used $r_k\geq t_k$. From \eqref{e:penultimo passagio parte stocastica},
dividing by $t_k^{n-1}$ and recalling that $r_k\geq t_k\geq s_k \geq \lfloor s_k \rfloor $, we obtain that
\begin{align*}
   &
   \frac{ \minprobs^{f^{\infty}_{\omega}}(\overline{u}_{\lfloor s_k \rfloor z,\zeta,\nu},Q^{\nu}_{r_k}(\lfloor s_k \rfloor z)) }{r_k^{n-1}} - \eta - \tilde K\Big((2\eta +\frac{j}{r_k})+(\eta+\frac{|z|}{r_k}+\frac{1}{2r_k})|\zeta|\Big)
   \leq \frac{ \minprobs^{f^{\infty}_{\omega}}(\overline{u}_{t_kx,\zeta,\nu},Q^{\nu}_{t_k}(t_kx))}{t_k^{n-1}}.
\end{align*}
Since $\omega\in \Omega'$ and $r_k\geq \lfloor s_k \rfloor$, we can apply \eqref{e:esistenza limite per interi step 4 stoch hom}, taking the $\liminf$ as $k\to \infty$ and letting $\eta \to 0$ we obtain
\begin{equation*}
    g\homm(\omega,\zeta,\nu)\leq \liminf_{k\to +\infty}\frac{ \minprobs^{f^{\infty}_{\omega}}(\overline{u}_{t_kx,\zeta,\nu},Q^{\nu}_{t_k}(t_kx))}{t_k^{n-1}}.
\end{equation*}
Arguing analogously we obtain
\begin{equation*}
  g\homm(\omega,\zeta,\nu)\geq \limsup_{k\to +\infty}\frac{ \minprobs^{f^{\infty}_{\omega}}(\overline{u}_{t_kx,\zeta,\nu},Q^{\nu}_{t_k}(t_kx))}{t_k^{n-1}}.
\end{equation*}
deducing the claim, {{thanks}} to the generality of the sequence $(t_k)_{k\in\N}$. 
\smallskip

\textit{Step 5:} Let $\Omega'$ be the set introduced in Step~3, then for every $\omega\in \Omega'$, $x\in \R^n$, $\zeta \in \R^N$, 
and $\nu \in \Sf^{n-1}$ 
\[
 g\homm(\omega,\zeta,\nu)=\lim_{r\to +\infty}\frac{ \minprobs^{f^{\infty}_{\omega}}(\overline{u}_{rx,\zeta,\nu},Q^{\nu}_{r}(rx))}{r^{n-1}}\,.
\]
For $\omega,\,x,\,\zeta,\nu$ as above define
\begin{equation*}
    \underline{g}(\omega,x,\zeta,\nu):=\liminf_{r\to +\infty}\frac{ \minprobs^{f^{\infty}_{\omega}}(\overline{u}_{rx,\zeta,\nu},Q^{\nu}_{r}(rx))}{r^{n-1}}\,,\qquad
    \overline{g}(\omega,x,\zeta,\nu):=\limsup_{r\to +\infty}\frac{ \minprobs^{f^{\infty}_{\omega}}(\overline{u}_{rx,\zeta,\nu},Q^{\nu}_{r}(rx))}{r^{n-1}}.
\end{equation*}
Arguing exactly as in Proposition~\ref{p:la ghom} (i)  and in Step~2, we obtain from Step~4 that 
\begin{equation}\label{e:penultima esistenza limite ghom}
   \underline{g}(\omega,x,\zeta,\nu)  = g\homm(\omega,\zeta,\nu)=\overline{g}(\omega,x,\zeta,\nu)
\end{equation}
for every $\omega \in \Omega'$, $x\in \R^n$, $\zeta \in \R^N$, and $\nu \in \Sf^{n-1}\cap \Q^n$. 

Now let $\omega \in \Omega'$, $x\in \R^n$, $\zeta\in \R^N$ and $\nu \in \Hat\Sf^{n-1}_{+}$, by density there is $(\nu_j)_{j\in\N}$ in $\Hat\Sf^{n-1}_{+}\cap \Q^n$ 
such that $\nu_j\to \nu$ as $j\to +\infty$. Thanks to the continuity on $\Hat \Sf^{n-1}_+$ of the map $\nu \mapsto R_{\nu}$, for every $\delta \in (0,\frac12)$ 
there exists $j_{\delta}$, such that
\begin{equation}\label{e:cubi compatt contenuti step 4 stoch hom}
 Q^{\nu}_{r}(rx) \subset \subset Q^{\nu_j}_{(1+\delta)r}(rx)\subset \subset Q^{\nu}_{(1+2\delta)r}(rx)
\end{equation}
for every $j\geq j_{\delta}$ and every $r>0$. Let us fix $j\geq j_{\delta}$, $r>0$ and $\eta>0$. Setting $c_j:=\max \{|R_{\nu_j}(e_i)\cdot \nu| \; : \; i=1,\dots,n-1\}$ we have that $c_j \to 0$ as $j\to +\infty$, by continuity of $\nu \mapsto R_{\nu}$ on $\Hat \Sf^{n-1}_{\pm}$, and recalling that $R_\nu\in O(n)$ and $R_\nu e_n=\nu$ (cf. \ref{e:Rnu} of the notation list). 
For every $y\in \overline{Q^{\nu}_{r(1+\delta)}(rx)}$ we have that $y-rx=y'+((y-rx)\cdot \nu_j)\nu_j$ where 
 \begin{equation*}
    y'\in R_{\nu_j}\Big(\big(-\frac{r}{2}(1+\delta),\frac{r}{2}(1+\delta)\big)^{n-1}\times \{0\}\Big),
 \end{equation*}
with, if $j$ is large enough depending only on $\delta$,
\begin{equation*}
   \textstyle  |(y-rx)\cdot \nu_j|\leq \frac{|y'\cdot \nu|}{|\nu_j \cdot \nu|}+\frac{1}{2(1-\delta)}< \frac{(n-1)c_jr (1+\delta)}{2(1-\delta)}+1=K(\delta)rc_j+1,
 \end{equation*}
 where $K(\delta):=\frac{(n-1) (1+\delta)}{2(1-\delta)}$, if in addition $|(y-rx)\cdot \nu|\leq \frac{1}{2}$.  
 Therefore, we can apply Lemma \ref{l:lemma prob di minimo superficie su cubi diversi}, with $R=K(\delta)rc_j+1$, and we get
 \begin{align*}
    \minprobs^{f^{\infty}_{\omega}}(\overline{u}_{rx,\zeta,\nu_j},Q^{\nu_j}_{r(1+\delta)}(rx))
    &\leq \minprobs^{f^{\infty}_{\omega}}(\overline{u}_{rx,\zeta,\nu},Q^{\nu}_{r}(rx)) + \eta r^{n-1} \\ & + \tilde K(2\delta (1+2\delta)^{n-2}r^{n-1}+(K(\delta)rc_j+1)|\zeta|(1+\delta)^{n-2}r^{n-2}).
\end{align*}
Dividing by $r^{n-1}$ and letting $r\to +\infty$, we obtain
\begin{align*}
   (1+\delta)^{n-1}&\underline{g}\big(\omega,\frac{x}{1+\delta},\zeta,\nu_j\big)\leq \underline{g}(\omega,x,\zeta,\nu) \\ 
   & +\eta + \tilde K(2\delta (1+2\delta)^{n-2}+K(\delta)c_j|\zeta|(1+\delta)^{n-2})\,. 
\end{align*}
Hence, we may use \eqref{e:penultima esistenza limite ghom} as  $\nu_j \in \Sf^{n-1}\cap \Q^n$ and deduce 
by taking the superior limit as $j\to +\infty$ and letting $\eta\to 0$ in the latter estimate
\begin{align*}
 (1+\delta)^{n-1}\limsup_{j\to +\infty} g\homm(\omega,\zeta,\nu_j)&=
 (1+\delta)^{n-1}\limsup_{j\to +\infty}\underline{g}\big(\omega,\frac{x}{1+\delta},\zeta,\nu_j\big)\\ 
    & \leq \underline{g}(\omega,x,\zeta,\nu)+  2\tilde K\delta (1+2\delta)^{n-2}.
\end{align*}
Therefore, by the continuity of $g\homm$ established in Step~2, letting $\delta \to 0$ we obtain
\begin{equation*}
  g\homm(\omega,\zeta,\nu)\leq \underline{g}(\omega,x,\zeta,\nu). 
\end{equation*}
Arguing analogously we have $\overline{g}(\omega,x,\zeta,\nu)\leq   g\homm(\omega,\zeta,\nu)$,
and recalling Corollary~\ref{c:comparison minimum problems} we conclude. 

\textit{Step 5:} In this step we show that if $f$ is ergodic then $g\homm$ is deterministic. 

Set $\Hat \Omega=\bigcap_{z\in \Z^n}\tau_z(\tilde \Omega)$; we clearly have that $\Hat \Omega \in \calT$, $\Hat \Omega \subseteq \tilde \Omega$ and $\tau_z(\Hat \Omega)=\Hat \Omega$ for every $z\in \Z^n$. Moreover, since $\tau_z$ is a $P$-preserving transformation and $P(\tilde \Omega)=1$, we have that $P(\Hat \Omega)=1$. We claim that 
\begin{equation}\label{e:invarianza ghom rispetto al gruppo}
    g\homm(\tau_z \omega,\zeta,\nu)\leq g\homm(\omega,\zeta,\nu)
\end{equation}
for every $\omega \in \Hat \Omega$, every $\zeta \in \R^N$ and every $\nu \in \Sf^{n-1}$. Fix $z\in \Z^n$, $\omega \in {{\Hat \Omega}}$ and $\nu \in \Sf^{n-1}$. For every $r>3|z|$, let $(u_r,v_r)\in W^{1,1}(Q^{\nu}_r,\R^N)\times W^{1,2}(Q^{\nu}_r,[0,1])$, with $(u_r,v_r)=(\overline{u}_{\zeta,\nu},1)$ on $\partial Q^{\nu}_r$ such that 
\begin{equation}\label{e:prob minimo per il caso ergodico}
    \int_{Q^{\nu}_r}\left(v_r^2f^{\infty}(\omega,y,\nabla u_r) + (1-v_r)^2+|\nabla v_r|^2 \right)\dy 
    \leq \minprobs^{f_{\omega}^{\infty}}(\overline{u}_{\zeta,\nu},Q^{\nu}_r)+1. 
\end{equation}
By the stationarity of $f$ (and hence of $f^{\infty}$) we infer that 
\begin{equation}\label{e:finfty stationarity}
    \minprobs^{f^{\infty}_{\tau_z\omega}}(\overline{u}_{\zeta,\nu},Q^{\nu}_r)=\minprobs^{f^{\infty}_{\omega}}(\overline{u}_{z,\zeta,\nu},Q^{\nu}_r(z)).
\end{equation}
Observe that $Q^{\nu}_r \subset \subset Q^{\nu}_{r+3|z|}(z)$ for every $r>3|z|$, and for every $y\in Q^{\nu}_{r+3|z|}(z)$ such that $|y\cdot \nu|\leq \frac{1}{2}$ we have that
\begin{equation*}
  1\geq   \frac{1}{2}+|\nu\cdot z|\geq |y\cdot \nu|=|(y-z)\cdot \nu+ z\cdot \nu| +|\nu \cdot z| \geq |(y-z)\cdot \nu|.
\end{equation*}
Then we can apply Lemma \ref{l:lemma prob di minimo superficie su cubi diversi}, with $R=1$, and for every $\eta>0$ we obtain
\begin{align*}
  \minprobs^{f^{\infty}_{\omega}}&(\overline{u}_{z,\zeta,\nu},Q^{\nu}_{r+3|z|}(z) ) \leq  \minprobs^{f_{\omega}^{\infty}}(\overline{u}_{\zeta,\nu},Q^{\nu}_r)+\eta r^{n-1}\\ &+\tilde K(r+3|z|)^{n-2}(3|z|+|\zeta|).
\end{align*}
Therefore, by definition of $g\homm$, $\Hat \Omega \subseteq \tilde \Omega$, and  \eqref{e:finfty stationarity} we obtain
\begin{align*}
   & g\homm(\tau_z\omega,\zeta,\nu)=\lim_{r\to +\infty}\frac{\minprobs^{f^{\infty}_{\tau_z(\omega)}}(\overline{u}_{\zeta,\nu},Q^{\nu}_{r})}{r^{n-1}}=\lim_{r\to +\infty}\frac{\minprobs^{f^{\infty}_{\omega}}(\overline{u}_{z,\zeta,\nu},Q^{\nu}_{r}(z))}{r^{n-1}} \\ 
   & =\lim_{r\to +\infty}\frac{\minprobs^{f^{\infty}_{\omega}}(\overline{u}_{z,\zeta,\nu},Q^{\nu}_{r+3|z|}(z))}{r^{n-1}}\leq \lim_{r\to +\infty}\frac{\minprobs^{f^{\infty}_{\omega}}(\overline{u}_{\zeta,\nu},Q^{\nu}_{r})}{r^{n-1}} \leq g\homm(\omega,\zeta,\nu)
\end{align*}
thus deducing the claim. 

By \eqref{e:invarianza ghom rispetto al gruppo} and the properties of $(\tau_z)_{z\in \Z^n}$, we clearly infer that 
\begin{equation*}
    g\homm(\tau_z\omega,\zeta,\nu)=g\homm(\omega,\zeta,\nu)
\end{equation*}
and hence, using the same argument as in \cite[Corollary~6.3]{CDMSZStochom}, if $(\tau_z)_{z\in \Z^n}$ is ergodic we deduce that $g\homm$ does not depend on $\omega$ and thus is deterministic. To conclude, we just observe that the representation of $g\homm(\zeta,\nu)$ as in \eqref{e:formula ghom ergodic case} 
is a direct consequence of \eqref{e:ghom achieved in stoch hom}, and the Dominated Convergence theorem (cf.  \eqref{e:crescita ghom}). 
\end{proof}

Finally, we are in aposition to prove the main result of this paper, Theorem~\ref{t:Stochastic homogenisation}.
\begin{proof}[Proof of Theorem~\ref{t:Stochastic homogenisation}]
The proof readily follwos by combining Theorem~\ref{t:Deterministic Gamma-conv}, Proposition~\ref{p:Volume part stochastic proposition}, \ref{p:cantor part stochastic}, and \ref{p:surface part stochastic}.
\end{proof}

\section*{Acknowledgements}  
F.~Colasanto wishes to thank  the excellence cluster ``Mathematics M\"unster: Dynamics--Geometry--Structure'' for the financial support, moreover he thanks the hospitality of the Institute for Applied Mathematics of the University of M\"unster where this work was initiated. 
C.~I.~Zeppieri was supported by the Deutsche Forschungsgemeinschaft (DFG, German Research Foundation) under Germany's Excellence Strategy EXC 2044 -390685587, Mathematics M\"unster: Dynamics--Geometry--Structure. 

F.~Colasanto and M.~Focardi have been supported by the European Union - Next Generation EU, Mission 4 Component 1 CUP B53D2300930006, codice 
2022J4FYNJ, PRIN2022 project ``Variational methods for stationary and evolution problems with singularities and interfaces''.
F.~Colasanto and M.~Focardi are members of GNAMPA - INdAM.

{The authors would like to thank the referee for a careful and thorough reading of the manuscript, as well as for their constructive comments and suggestions, which have helped to improve the presentation of the paper.}


\bibliography{biblio}
\bibliographystyle{plain}

\end{document}